\documentclass{lmcs} 
\usepackage[utf8]{inputenc}
\pdfoutput=1

\usepackage{lastpage}
\lmcsdoi{18}{4}{10}
\lmcsheading{}{\pageref{LastPage}}{}{}%
{Mar.~18,~2020}{Dec.~08,~2022}{}


\usepackage{hyperref}

\usepackage[mathscr]{euscript}
\usepackage{amssymb}
\usepackage{stmaryrd}
\usepackage{tikz-cd}
  \usetikzlibrary{arrows,arrows.meta}
  \tikzcdset{arrow style=tikz, diagrams={>=stealth'}}
\usepackage{multicol}
\usepackage{bbm}

\usepackage{scalerel}
\usepackage{stackengine,wasysym}
\newcommand\reallywidetilde[1]{\ThisStyle{%
  \setbox0=\hbox{$\SavedStyle#1$}%
  \stackengine{-.1\LMpt}{$\SavedStyle#1$}{%
    \stretchto{\scaleto{\SavedStyle\mkern.2mu\AC}{.5150\wd0}}{.6\ht0}%
  }{O}{c}{F}{T}{S}%
}}

\newcommand{\mc}[1]{\mathcal{#1}}
\newcommand{\ms}[1]{\mathscr{#1}}

\newcommand{\mf}[1]{\mathfrak{#1}}

\newcommand{\topo}[1]{\mathbb{#1}}
\newcommand{\cat}[1]{\mathbf{#1}}
\newcommand{\fun}[1]{\mathsf{#1}}

\renewcommand{\hat}[1]{\widehat{#1}}

\renewcommand{\tilde}[1]{\widetilde{#1}}

\renewcommand{\phi}{\varphi}
\newcommand{\und}[1]{\underline{#1}}
\renewcommand{\epsilon}{\varepsilon}

\renewcommand{\iff}{\quad\text{iff}\quad}
\DeclareMathOperator{\Ax}{Ax}
\DeclareMathOperator{\id}{id}

\DeclareMathOperator{\op}{op}

\newcommand{\wi}{\eqslantless}

\renewcommand{\th}{\operatorname{th}}

\newcommand{\Sier}{\topo{S}}
\newcommand{\two}{\mathbbm{2}}
\DeclareMathOperator{\GML}{GML}
\DeclareMathOperator{\GL}{GL}
\DeclareMathOperator{\ML}{ML}
\newcommand{\modeq}{\equiv_{\Lambda}}
\newcommand{\bigveeup}{\textstyle{\bigvee^{\uparrow}}}
\newcommand{\nega}{{\sim}}
\newcommand{\bigcupup}{\textstyle{\bigcup^{\hspace{-1.1mm}{\uparrow}}}}
\newcommand{\cp}{\mathrel{\triangleleft}}

\newcommand{\llb}{\llbracket}
\newcommand{\rrb}{\rrbracket}
\newcommand{\llp}{\llparenthesis\,}
\newcommand{\rrp}{\,\rrparenthesis}

\newcommand{\dbox}{\boxbar}
\newcommand{\ddiamond}{\mathbin{\rotatebox[origin=c]{45}{$\boxslash$}}}
\newcommand{\diamonddot}{\mathbin{\rotatebox[origin=c]{45}{$\boxdot$}}}

\begin{document}

\title[Coalgebraic Geometric Logic]{Coalgebraic Geometric Logic: Basic Theory}
\titlecomment{{\lsuper*}Paper appeared at CALCO 2019}

\author[N.~Bezhanishvili]{Nick Bezhanishvili\rsuper{a}}	
\address{University of Amsterdam, Amsterdam, The Netherlands}	
\email{\{n.bezhanishvili,y.venema\}@uva.nl}  

\author[J.~de Groot]{Jim de Groot\lmcsorcid{0000-0003-1375-6758}\rsuper{b}}	
\address{The Australian National University, Canberra, Ngunnawal country, Australia}	
\email{jim.degroot@anu.edu.au}  

\author[Y.~Venema]{Yde Venema\rsuper{a}}	




\begin{abstract}
  \noindent Using the theory of coalgebra, we introduce a uniform framework for
  adding modalities to the language of propositional geometric logic.
  Models for this logic are based on coalgebras for an endofunctor on
  some full subcategory of the category of topological spaces and
  continuous functions.
  We investigate derivation systems, soundness and completeness
  for such geometric modal logics, and we specify a method
  of lifting an endofunctor on $\cat{Set}$, accompanied by a collection of
  predicate liftings, to an endofunctor on the category of topological spaces,
  again accompanied by a collection of (open) predicate liftings.
  Furthermore, we compare the notions of modal equivalence, behavioural
  equivalence and bisimulation on the resulting class of models, and we provide
  a final object for the corresponding category.
\end{abstract}

\maketitle

\section{Introduction}

Propositional geometric logic arose at the interface of (pointfree) topology,
logic and theoretical computer science as the logic of \emph{finite
observations}~\cite{Abr87,Vic89}.
Its language is constructed from a set of proposition letters by applying finite
conjunctions and arbitrary disjunctions, these being the propositional
operations preserving the property of finite observability.
Through an interesting topological connection, formulas of geometric logic
can be interpreted in the frame of open sets of a topological space.
Central to this connection is the well-known dual adjunction between the category $\cat{Frm}$ of
frames and frame morphisms and the category $\cat{Top}$ of topological spaces
and continuous maps, which restricts to several interesting Stone-type
dualities~\cite{Joh82}.

Coalgebraic logic is a framework in which generalised versions of modal logics
are developed parametric in the signature of the language and a functor
$\fun{T} : \cat{C} \to \cat{C}$ on some base category $\cat{C}$.
With classical propositional logic as base logic, two natural choices for the base
category are $\cat{Set}$, the category of sets and functions, and $\cat{Stone}$,
the category of Stone spaces and continuous functions, i.e.~the
topological dual to the algebraic category of Boolean algebras.
Coalgebraic logic for endofunctors on $\cat{Set}$ has been well investigated and
still is an active area of research, see e.g.~\cite{CirEA08,KupPat11}.
In this setting, modal operators can be defined using the notion of relation
lifting~\cite{Mos99} or predicate lifting~\cite{Pat03b}.
Coalgebraic logic in the category of Stone coalgebras has been studied
in~\cite{KupKurVen04,HanKup04,EnqSou17,BezEnqGro20}, and there is
a fairly extensive literature on the design of a coalgebraic modal logic based
on a general Stone-type duality (or dual adjunction), see for
instance~\cite{BonKur05,BonKur06,chen:cate14,Kli07} and references therein.

In this paper we investigate some links between coalgebraic logic and geometric
logic.
That is, we use methods from coalgebraic logic to introduce modal operators
to the language of geometric logic, with the intention of studying
interpretations of these logics in certain topological coalgebras.
Note that extensions of geometric logic with the basic modalities $\Box$ and
$\Diamond$, which are closely related to the topological Vietoris construction,
have received much attention in the literature, see~\cite{Vic89} for some early
history.
A first step towards developing coalgebraic geometric logic was taken
in~\cite{VVV13}, where a method is explored to lift a functor on $\cat{Set}$
to a functor on the category $\cat{KHaus}$ of compact Hausdorff spaces, and the
connection is investigated between the lifted functor and a relation-lifting
based ``cover'' modality.

Our aim here is to develop a framework for the coalgebraic geometric logics that
arise if we extend geometric logic with modalities that are induced by
appropriate \emph{predicate liftings}.
Guided by the connection between geometric logic and topological spaces, we
choose the base category of our framework to be $\cat{Top}$ itself, or one of
its full subcategories such as $\cat{Sob}$ (sober spaces), $\cat{KSob}$ (compact
sober spaces) or $\cat{KHaus}$ (compact Hausdorff spaces).
On this base category $\cat{C}$ we then consider an arbitrary endofunctor
$\fun{T}$ which serves as the type of our topological coalgebras.
Furthermore, we shall see that if we want our formulas to be interpreted as
\emph{open sets} of the coalgebra carrier, we need the predicate liftings that
interpret the modalities of the language to satisfy some natural \emph{openness}
condition.
Summarizing, we shall study the coalgebraic geometric logic induced by
(1) a functor $\fun{T}: \cat{C} \to \cat{C}$, where $\cat{C}$ is a full
subcategory of $\cat{Top}$, and
(2) a set $\Lambda$ of open predicate liftings for $\fun{T}$.
As running examples we take the combination of the basic modalities for the
Vietoris functor, and that of the monotone box and diamond modalities for
various topological manifestations of the monotone neighbourhood functor on
$\cat{Set}$.
The structures providing the semantics for our coalgebraic geometric logics are
the \emph{$\fun{T}$-models} comprised of a $\fun{T}$-coalgebra together with a
valuation mapping proposition letters to open sets in the coalgebra carrier.
\medskip

\noindent
The main results that we report on here are the following:
\begin{itemize}
\item
Section~\ref{sec:mon} contains a detailed description of the \emph{monotone
neighbourhood functor} on $\cat{KHaus}$, which naturally extends the monotone
functor on $\cat{Stone}$~\cite{HanKup04} that corresponds to monotone modal logic.
\item
In Section~\ref{sec:ax} we discuss derivation systems for coalgebraic geometric
logic, based on \emph{consequence pairs}, and derive a general completeness result.
\item
After that, in Section~\ref{sec:lift} we adapt the method of~\cite{KupKurPat04} in order to
lift a $\cat{Set}$-functor together with a collection of predicate liftings to
an endofunctor on $\cat{Top}$.
We obtain the Vietoris functor and monotone functor on $\cat{KHaus}$ as
restrictions of such lifted functors.
\item
In Section~\ref{sec:final}, we construct a final object in the
category of $\fun{T}$-models, where $\fun{T}$ is an endofunctor on $\cat{Top}$
which preserves sobriety and admits a Scott-continuous, characteristic geometric
modal signature.
\item
Finally, in Section~\ref{sec-bisim} we transfer the notion of
$\Lambda$-bisimilarity from~\cite{GorSch13,BakHan17} to our setting, and we
compare this to geometric modal equivalence, behavioural equivalence and
Aczel-Mendler bisimilarity.
Our main finding is that on the categories $\cat{Top}$, $\cat{Sob}$ and
$\cat{KSob}$, the first three notions coincide, provided $\Lambda$ and $\fun{T}$
meet some reasonable conditions.
\end{itemize}
We finish the paper with listing some questions for further research.

\section{Preliminaries}

  We briefly fix notation and review some preliminaries.

\subsection{Categories and functors}
  We use a $\cat{bold}$ $\cat{font}$ for categories. We assume familiarity with
  the following categories and functors:
  \begin{itemize}
    \item $\cat{Set}$ is the category of sets and functions;
    \item $\cat{Top}$ is the category of topological spaces and continuous
          functions;
    \item $\cat{KHaus}$ and $\cat{Stone}$ are the full subcategories of
          $\cat{Top}$ whose objects are compact Hausdorff spaces and Stone spaces,
          respectively;
    \item $\cat{BA}$ is the category of Boolean algebras and Boolean algebra
          morphisms.
  \end{itemize}
  Categories can be connected by functors. We use a $\fun{sans}$ $\fun{serif}$ $\fun{font}$ for
  functors. In particular, the following functors are regularly used in this
  paper:
  \begin{itemize}
    \item $\fun{U} : \cat{Top} \to \cat{Set}$ is the forgetful functor sending a
          topological space to its underlying set. The functor $\fun{U}$
          restricts to every subcategory of $\cat{Top}$, in which case we shall
          abuse notation and also call it $\fun{U}$;
    \item $\fun{P} : \cat{Set} \to \cat{Set}$ and $\breve{\fun{P}} :
          \cat{Set}^{\op} \to \cat{Set}$ are the covariant and contravariant
          powerset functor respectively;
    \item $\fun{Q} : \cat{Set}^{\op} \to \cat{BA}$ sends a set to its powerset
          Boolean algebra and a function to the inverse image map viewed as
          morphism in $\cat{BA}$;
    \item $\fun{\Omega} : \cat{Top} \to \cat{Set}$ is the contravariant functor
          that sends a topological space to
          its set of opens.
  \end{itemize}
  Note that $\breve{\fun{P}} = \fun{U}_{\cat{BA}} \circ \fun{Q}$,
  where $\fun{U}_{\cat{BA}} : \cat{BA} \to \cat{Set}$ is the obvious forgetful functor.
  More categories and
  functors will be defined along the way. We use the symbol $\equiv$ for
  categorical equivalence.

\subsection{Coalgebra}\label{subsec:coalg}

  Let $\cat{C}$ be a category and $\fun{T}$ an endofunctor on $\cat{C}$. A
  \emph{$\fun{T}$-coalgebra} is a pair $(X, \gamma)$ where $X$ is an object in
  $\cat{C}$ and $\gamma : X \to \fun{T}X$ is a morphism in $\cat{C}$.
  A \emph{$\fun{T}$-coalgebra morphism} between two $\fun{T}$-coalgebras
  $(X, \gamma)$ and $(X', \gamma')$ is a morphism $f : X \to X'$ in $\cat{C}$
  satisfying $\gamma' \circ f = \fun{T}f \circ \gamma$.
  The collection of $\fun{T}$-coalgebras and $\fun{T}$-coalgebra morphisms forms
  a category, which we shall denote by $\cat{Coalg}(\fun{T})$. The category
  $\cat{C}$ is called the \emph{base category} of $\cat{Coalg}(\fun{T})$.

  The notion of an \emph{algebra} for $\fun{T}$ is defined dually, and gives
  rise to the category $\cat{Alg}(\fun{T})$.

\begin{exa}[Kripke frames]
  Kripke frames correspond 1-1 with $\fun{P}$-coalgebras. For a Kripke frame
  $(X, R)$ define $\gamma_R : X \to \fun{P}X : x \mapsto \{ y \mid xRy \}$.
  Then $(X, \gamma_R)$ is a $\fun{P}$-coalgebra. Conversely, for a
  $\fun{P}$-coalgebra $(X, \gamma)$ define $R_{\gamma}$ by $xR_{\gamma}y$ iff
  $y \in \gamma(x)$. Then $(X, R_{\gamma})$ is a Kripke frame. It is not hard to
  see that $R_{\gamma_R} = R$ and $\gamma_{R_{\gamma}} = \gamma$, so we obtain a
  bijection between Kripke frames and $\fun{P}$-coalgebras. Moreover, bounded
  morphisms between Kripke frames are precisely $\fun{P}$-coalgebra morphisms.
  Thus, we have
  \begin{equation*}
    \cat{Krip} \cong \cat{Coalg}(\fun{P}),
  \end{equation*}
  where $\cat{Krip}$ is the category of Kripke frames and bounded morphisms.
\end{exa}

\begin{exa}[Monotone neighbourhood frames]\label{exm-mon-set}
  Let $\fun{D} : \cat{Set} \to \cat{Set}$ be the functor given on objects by
  \begin{equation*}
    \fun{D}X = \{ W \subseteq \fun{P}X \mid \text{ if }a \in W
      \text{ and } a \subseteq b \text{ then } b \in W \},
  \end{equation*}
  where $X$ is a set. For a morphism $f : X \to X'$ define
  \begin{equation*}
    \fun{D}f
      : \fun{D}X \to \fun{D}X'
      : W \mapsto \{ a' \in \fun{P}X' \mid f^{-1}(a') \in W \}.
  \end{equation*}
  Then the category of monotone frames and bounded morphisms is isomorphic to
  $\cat{Coalg}(\fun{D})$ \cite{Che80, Han03, HanKup04}.
\end{exa}

\subsection{Coalgebraic logic for \texorpdfstring{$\cat{Set}$}{Set}-coalgebras}

  Let $\fun{T}$ be a $\cat{Set}$-functor and $\Phi$ a set of proposition letters.
  A \emph{$\fun{T}$-model} is a triple $(X, \gamma, V)$ where $(X, \gamma)$ is a
  $\fun{T}$-coalgebra and $V : \Phi \to \fun{P}X$ is a valuation of the
  proposition letters. An \emph{$n$-ary predicate lifting for $\fun{T}$} is a
  natural transformation
  \begin{equation*}
    \lambda : \breve{\fun{P}}^n \to \breve{\fun{P}} \circ \fun{T},
  \end{equation*}
  where $\breve{\fun{P}}^n$ denotes the $n$-fold product of the contravariant
  powerset functor.
  A predicate lifting is called \emph{monotone in its $i$-th argument} if for
  all sets $X$ and subsets $a_1, \ldots, a_n, b \subseteq X$ we have
  $\lambda_X(a_1, \ldots, a_i, \ldots, a_n) \subseteq
  \lambda_X(a_1, \ldots, a_i \cup b, \ldots, a_n)$.

  For a set $\Lambda$ of predicate liftings for $\fun{T}$, define the
  \emph{language} $\ML(\Lambda)$ by
  \begin{equation*}
    \phi ::= p \mid \neg \phi
               \mid \phi \wedge \phi
               \mid \heartsuit^{\lambda}(\phi_1, \ldots, \phi_n),
  \end{equation*}
  where $p \in \Phi$ and $\lambda \in \Lambda$ is $n$-ary.
  The \emph{semantics} of $\phi \in \ML(\Lambda)$ on a $\fun{T}$-model
  $\mf{X} = (X, \gamma, V)$ is given recursively by
  \begin{align*}
    &\llb p \rrb^{\mf{X}} = V(p),
      \quad \llb \phi_1 \wedge \phi_2 \rrb^{\mf{X}}
        = \llb \phi_1 \rrb^{\mf{X}} \cap \llb \phi_2 \rrb^{\mf{X}},
      \quad \llb \neg \phi \rrb^{\mf{X}}
        = X \setminus \llb \phi \rrb^{\mf{X}}, \\
    &\llb \heartsuit^{\lambda}(\phi_1, \ldots, \phi_n) \rrb^{\mf{X}}
        = \gamma^{-1}(\lambda(\llb \phi_1 \rrb^{\mf{X}}, \ldots, \llb
          \phi_n \rrb^{\mf{X}})),
  \end{align*}
  where $p \in \Phi$ and $\lambda$ ranges over $\Lambda$.

\begin{exa}[Kripke models]\label{exm-krip-pred}
  Consider for $\fun{P}$-models the predicate liftings $\lambda^{\Box},
  \lambda^{\Diamond} : \breve{\fun{P}} \to \breve{\fun{P}} \circ \fun{P}$
  given by
  \[
    \lambda^{\Box}_X(a) = \{ b \in \fun{P}X \mid b \subseteq a \}, \qquad
    \lambda^{\Diamond}_X(a) = \{ b \in \fun{P}X \mid b \cap a \neq \emptyset \}.
  \]
  Then $\lambda^{\Box}$ and $\lambda^{\Diamond}$ yield the usual Kripke semantics
  of $\Box$ and $\Diamond$.
\end{exa}

\begin{exa}[Monotone neighbourhood frames]\label{exm-mon-set-pred}
  Monotone neighbourhood models are precisely $\fun{D}$-models, where $\fun{D}$
  is the functor defined in Example~\ref{exm-mon-set}. The usual semantics for
  the box and diamond in this setting can be obtained from the predicate liftings
  given by
  \begin{equation}
    \lambda^{\Box}_X(a) = \{ W \in \fun{D}X \mid a \in W \}, \qquad
    \lambda^{\Diamond}_X(a) = \{ W \in \fun{D}X \mid X \setminus a \notin W \}.
  \end{equation}
\end{exa}

  We refer to~\cite{KupPat11} for many more examples of coalgebraic logics for
  $\cat{Set}$-functors.

\subsection{Geometric logic}\label{subsec:gl}

  Let $\Phi$ be a set of proposition letters. The language $\GL(\Phi)$ of geometric
  formulas is given by
  \[
    \phi ::= \top \mid \bot
                  \mid p
                  \mid \phi \wedge \phi
                  \mid \bigvee_{i \in I} \phi_i
  \]
  where $p \in \Phi$ and $I$ is some index set. \emph{Coherent} formulas are
  defined in the same way, but without infinitary disjuctions. These are also known
  as the formulas from \emph{positive} logic.

  We describe the logical system as a collection of
  binary consequence pairs, written as $\phi \cp \psi$.
  A \emph{geometric logic} is a collection consequence pairs closed under the
  following rules:
  \emph{identity}
          \[
            \phi \cp \phi,
          \]
  \emph{cut}
          \[
            \dfrac{\phi \cp \psi \quad \psi \cp \chi}{\phi \cp \chi},
          \]
  the \emph{conjunction rules}
          \[
            \phi \cp \top, \qquad
            \phi \wedge \psi \cp \phi, \qquad
            \phi \wedge \psi \cp \psi, \qquad
            \dfrac{\phi \cp \psi \quad \phi \cp \chi}{\phi \cp \psi \wedge \chi},
          \]
  the \emph{disjunction rules}
          \[
            \phi \cp \bigvee S \quad (\phi \in S), \qquad
            \dfrac{\phi \cp \psi \quad (\text{all } \phi \in S)}{\bigvee S \cp \psi}
          \]
  and \emph{frame distributivity}
          \[
            \phi \wedge \bigvee S \cp \bigvee \{ \phi \wedge \psi \mid \psi \in S \}.
          \]

  Note that these are in fact all schemata.
  We write $\ms{GL}$ for the minimal geometric logic, i.e.~the
  smallest collection of consequence pairs closed under the axioms and rules
  given above. We write $\phi \vdash_{\ms{GL}} \psi$ if the consequence
  pair $\phi \cp \psi$ is in $\ms{GL}$.

  Note that frame distributivity allows us to reduce every formula to a
  disjunction of finite conjunctions of proposition letters. Therefore, modulo
  equivalence, the formulas form a set. A collection $S$ of geometric formulas is
  called \emph{directed} if for every pair $\phi, \psi \in S$ there exists
  $\chi \in S$ such that $\phi \vdash \chi$ and $\psi \vdash \chi$.

  The topological semantics and algebraic semantics of geometric logic are given
  by topological spaces and frames.

\subsection{Frames and spaces}

  A \emph{frame} is a complete lattice $F$ in which for all $a \in F$ and
  $S \subseteq F$ the infinite distributive law holds:
  \begin{equation*}
    a \wedge \bigvee S = \bigvee \{ a \wedge s \mid s \in S \}.
  \end{equation*}
  A \emph{frame homomorphism} is a function between frames that preserves finite
  meets and arbitrary joins.

  For $a, b \in F$ we say that $a$ is \emph{well inside} $b$, notation:
  $a \wi b$, if there is a $c \in F$ such that $c \wedge a = \bot$ and
  $c \vee b = \top$. An element $a \in F$ is called \emph{regular} if
  $a = \bigvee \{ b \in F \mid b \wi a \}$ and a frame is called \emph{regular}
  if all of its elements are regular. The \emph{negation} of $a \in F$ is defined
  as $\nega a = \bigvee \{ b \in F \mid a \wedge b = \bot \}$.
  A frame is said to be \emph{compact} if $\bigvee S = \top$ implies that there
  is a finite subset $S' \subseteq S$ such that $\bigvee S' = \top$.

\begin{lem}\label{lem-wi-nega}
  For all elements $a, b$ in a frame $F$ we have $a \wi b$ iff
  $\nega a \vee b = \top$.
\end{lem}
\begin{proof}
  See~\cite[III1.1]{Joh82}.
\end{proof}

\begin{lem}\label{lem-reg-elt}
  Finite meets and arbitrary joins of regular elements are regular.
\end{lem}
\begin{proof}
  It is known that $d \leq c \wi a \leq b$ implies $d \wi b$.
  We first show that $c \wi a$ and $d \wi b$ implies $c \wedge d \wi a \wedge b$.
  It is clear that $c \wedge d \wi a$ and $c \wedge d \wi b$. Since
  $\nega(c \wedge d) \vee (a \wedge b) = (\nega(c \wedge d) \vee a) \wedge
  (\nega(c \wedge d) \vee b) = \top \wedge \top = \top$ we know $c \wedge d \wi a
  \wedge b$.

  Now suppose $a$ and $b$ are regular elements, then
  \begin{equation*}
    a \wedge b
      = \bigvee \{ c \mid c \wi a \} \wedge \bigvee \{ d \mid d \wi b \}
      = \bigvee \{ c \wedge d \mid c \wi a, d \wi b \}
      \leq \bigvee \{ c \mid c \wi a \wedge b \}
      \leq a \wedge b,
  \end{equation*}
  so $a \wedge b$ is regular. If $a_i$ is regular for all $i$ in some index set
  $I$, then
  \begin{equation*}
    \bigvee_{i \in I} a_i
      = \bigvee_{i \in I} \Big( \bigvee \{ c \mid c \wi a_i \} \Big)
      \leq \bigvee \Big\{ c \mid c \wi \bigvee_{i \in I} a_i \Big\}
      \leq \bigvee_{i \in I} a_i,
  \end{equation*}
  so an arbitrary join of regular elements is regular.
\end{proof}

  Frames can be presented by generators and relations.

\begin{defi}\label{ch3-def-frm-pres}
  A \emph{presentation} is a pair $\langle G, R \rangle$ where $G$ is a set of
  generators and $R$ is a collection of relations between expressions constructed
  from the generators using arbitrary joins and finite meets.

  Let $F$ be a frame and $\fun{Z}F$ its underlying set. We say that
  $\langle G, R \rangle$ \emph{presents} $F$ if there is an assignment
  $f : G \to \fun{Z}F$ of the generators such that~\ref{eq-presen-1}, \ref{eq-presen-2} and~\ref{eq-presen-3} hold:
  \begin{enumerate}[(i)]
    \item\label{eq-presen-1} The set $\{ f(g) \mid g \in G \}$ generates $F$,
          that is, every element of $F$ can be obtained from
          $\{ f(g) \mid g \in G \}$ using finite meets and arbitrary joins in $F$.
  \end{enumerate}
  The assignment $f$ can be extended to an assignement $\tilde{f}$ for any
  expression $x$ build from the generators in $G$ using $\wedge$ and $\bigvee$.
  We require:
  \begin{enumerate}[(i)]
  \setcounter{enumi}{1}
    \item\label{eq-presen-2} If $x = x'$ is a relation in $R$, then
          $\tilde{f}(x) = \tilde{f}(x')$ in $F$.
    \item\label{eq-presen-3} For any frame $F'$ and assignment $f' : G \to \fun{Z}F'$
          satisfying property~\ref{eq-presen-2} there exists a frame homomorphism
          $h : F \to F'$ such that the diagram
          \begin{equation*}
            \begin{tikzcd}[arrows=-latex]
              G     \arrow[r, "f"]
                    \arrow[dr, "f'" below left]
                & \fun{Z}F
                    \arrow[d, "\fun{Z}h"] \\
                & \fun{Z}F'
            \end{tikzcd}
          \end{equation*}
          commutes.
  \end{enumerate}
\end{defi}

  The frame homomorphism from~\ref{eq-presen-3} is necessarily unique, because
  the image of the generating set $\{ f(g) \mid g \in G \}$ under $h$ is
  determined by the diagram. A detailed account of frame presentations may be
  found in Chapter~4 of~\cite{Vic89}.

  We may also use inequalities when presenting a frame.
  An inequality $x \leq x'$ can simply be viewed as shorthand for
  $x = x \wedge x'$.

\begin{rem}\label{rem-gen-frm-hom}
  We will regularly define a frame homomorphism $F \to F'$ from a frame $F$
  presented by $\langle G, R \rangle$ to some frame $F'$.
  By Definition~\ref{ch3-def-frm-pres} it suffices to give an assignment
  $f' : G \to F'$ such that~\ref{eq-presen-2} holds, because this yields
  a unique frame homomorphism $F \to F'$. By
  abuse of notation, we will denote the unique frame homomorphism $F \to F'$ such
  that the diagram in~\ref{eq-presen-3} commutes with $f'$ as well.
\end{rem}

  The next fact allows us to define a frame by specifying generators and
  relations. A proof can be found in~\cite[Proposition~II2.11]{Joh82}.

\begin{fact}
  Any presentation by generators and relations presents a unique frame.
\end{fact}

  The collection of open sets of a topological space $\topo{X}$ forms a frame,
  denoted $\fun{opn}\topo{X}$. A continuous map $f : \topo{X} \to \topo{X}'$
  induces $\fun{opn} f = f^{-1} : \fun{opn}\topo{X}' \to \fun{opn}\topo{X}$ and
  with this definition $\fun{opn}$ is a contravariant functor $\cat{Top} \to
  \cat{Frm}$. A frame is called \emph{spatial} if it isomorphic to $\fun{opn}
  \topo{X}$ for some topological space $\topo{X}$.

  A \emph{point} of a frame $F$ is a frame homomorphism $p : F \to 2$, with
  $2 = \{ \top, \bot \}$ the two-element frame. Let $\fun{pt} F$ be the
  collection of points of $F$ endowed with the topology $\{ \tilde{a} \mid a \in
  F \}$, where $\tilde{a} = \{ p \in \fun{pt} F \mid p(a) = \top \}$. For a frame
  homomorphism $f : F \to F'$ define $\fun{pt} f : \fun{pt} F' \to \fun{pt} F$ by
  $p \mapsto p \circ f$. The assignment $\fun{pt}$ defines a contravariant functor
  $\cat{Frm} \to \cat{Top}$.
  A topological space that arises as the space of points of a lattice is called
  \emph{sober}. The \emph{sobrification} of a topological space $\topo{X}$ is
  $\fun{pt}(\fun{opn}\topo{X})$.

  We denote by $\cat{Sob}$ and $\cat{KSob}$ the full subcategories of $\cat{Top}$
  whose objects are sober spaces and compact sober spaces, respectively. Where
  $\cat{Frm}$ is the category of frames and frame homomorphisms, $\cat{SFrm}$,
  $\cat{KSFrm}$ and $\cat{KRFrm}$ are the full subcategories of $\cat{Frm}$ whose
  objects are spatial frames, compact spatial frames and compact regular frames,
  respectively. The functor $\fun{Z} : \cat{Frm} \to \cat{Set}$ is the forgetful
  functor sending a frame to the underlying set, and restricts to every
  subcategory of $\cat{Frm}$. Note that $\fun{\Omega} = \fun{Z} \circ \fun{opn}$.

\begin{fact}\label{fact-pt-opn}
  The functor $\fun{pt}$ is dually adjoint to $\fun{opn}$.
  This adjunction restricts to a duality between the category of spatial frames
  and the category of sober spaces,
  \begin{equation*}
    \cat{SFrm} \equiv \cat{Sob}^{\op}.
  \end{equation*}
  This duality restricts to the dualities
  \begin{equation*}
    \cat{KSFrm} \equiv \cat{KSob}^{\op}
  \end{equation*}
  and
  \begin{equation*}
    \cat{KRFrm} \equiv \cat{KHaus}^{\op}.
  \end{equation*}
\end{fact}

  For a more thorough exposition of frames and spaces, and a proof of the
  statements in Fact~\ref{fact-pt-opn}, we refer to Section C1.2 of~\cite{Joh02}.
  We explicitly mention one isomorphism which is part of this duality, for we
  will encounter it later on.

\begin{rem}\label{rem-isbell-iso}
  Let $\topo{X}$ be a sober space. Then Fact~\ref{fact-pt-opn} entails that
  there is an isomorphism $\topo{X} \to \fun{pt}(\fun{opn}\topo{X})$. This
  isomorphism is given by $x \mapsto p_x$, where $p_x$ is the point given by
  \begin{equation*}
    p_x
      : \fun{opn}\topo{X} \to 2
      : \left\{\begin{array}{ll}
          a \mapsto \top &\text{ if $x \in a$} \\
          a \mapsto \bot &\text{ if $x \notin a$}
        \end{array}\right.
  \end{equation*}
for all $x \in \topo{X}$ and $a \in \fun{\Omega}\topo{X}$.
\end{rem}

\section{Logic for topological coalgebras}\label{sec:logic}

  Although not all of our results can be proved for every full subcategory of
  $\cat{Top}$, we will give the basic definitions in full generality. To this
  end, we let $\cat{C}$ be some full subcategory of $\cat{Top}$ and define
  coalgebraic logic with $\cat{C}$ as base category.
  In particular $\cat{C} =
  \cat{KHaus}$ and $\cat{C} = \cat{Sob}$ will be of interest. Throughout this
  section $\fun{T}$ is an arbitrary endofunctor on $\cat{C}$.
  Recall that $\fun{\Omega} : \cat{Top} \to \cat{Set}$ sends a topological
  spaces to its set of opens, while $\fun{opn} : \cat{Top} \to \cat{Frm}$ takes
  a space to its collection of opens viewed as a frame.
  Also, recall that $\Phi$ is an arbitrary but fixed set of proposition letters.
  We begin with defining the topological version of a predicate
  lifting, called an \emph{open} predicate lifting.

\subsection{Open predicate liftings}
\begin{defi}\label{def:open-pred-lift}
  An \emph{open predicate lifting} for $\fun{T}$ is a natural transformation
  \begin{equation*}
    \lambda : \fun{\Omega}^n \to \fun{\Omega} \circ \fun{T}.
  \end{equation*}
  An open predicate lifting is called \emph{monotone in its $i$-th argument}
  if for every $\topo{X} \in \cat{C}$ and all $a_1, \ldots, a_n, b \in
  \fun{\Omega}\topo{X}$ we have $\lambda_{\topo{X}}(a_1, \ldots, a_i, \ldots, a_n)
  \subseteq \lambda_{\topo{X}}(a_1, \ldots,a_i \cup b, \ldots, a_n)$, and
  \emph{monotone} if it is monotone in every argument.
  It is called \emph{Scott-continuous} in its $i$-th argument if for every
  $\topo{X} \in \cat{C}$ and every directed set $A \subseteq \fun{\Omega}\topo{X}$
  we have
  \[
    \lambda_{\topo{X}}(a_1, \ldots, \bigcup A, \ldots, a_n) =
    \bigcup_{b \in A}\lambda_{\topo{X}}(a_1, \ldots, b, \ldots, a_n)
  \]
  and \emph{Scott-continuous} if it is Scott-continuous in every argument.

  A collection of open predicate liftings for $\fun{T}$ is called a
  \emph{geometric modal signature} for $\fun{T}$.
  A geometric modal signature for a functor $\fun{T}$ is called \emph{monotone}
  if every open predicate lifting in it is monotone, \emph{Scott-continuous} if
  every open predicate lifting in it is Scott-continuous, and \emph{characteristic}
  if for every topological space $\topo{X}$ in $\cat{C}$ the collection
  \[
    \{ \lambda_{\topo{X}}(a_1, \ldots, a_n) \mid \lambda \in \Lambda
        \text{ $n$-ary}, a_i \in \fun{\Omega}\topo{X} \}
  \]
  is a sub-base for the
  topology on $\fun{T}\topo{X}$.
\end{defi}

\begin{rem}
  Using the fact that for any two (open) sets $a, b$ the set $\{ a, a \cup b \}$
  is directed, it is easy to see that Scott-continuity implies monotonicity.

  Scott-continuity will play a r\^{o}le in Section~\ref{sec-final}, where it is
  used to show that the collection of formulas modulo (semantic) equivalence is
  a set, rather than a proper class.
\end{rem}

  Let $\Sier$ be the Sierpinski space, i.e.~the two-element set $2 = \{ 0, 1 \}$
  topologised by $\{ \emptyset, \{ 1 \}, 2 \}$. For a topological space $\topo{X}$
  and $a \subseteq \fun{U}\topo{X}$ let $\chi_a : \topo{X} \to \Sier$ be the
  characteristic map (i.e.~$\chi_a(x) = 1$ iff $x \in a$). Note that $\chi_a$ is
  continuous if and only if $a \in \fun{\Omega}\topo{X}$. Analogously to
  predicate liftings for $\cat{Set}$-functors~\cite[Proposition 43]{Sch05}, one
  can classify $n$-ary predicate liftings as open subsets of $\fun{T}\Sier^n$.
  This elucidates the analogy with predicate liftings for $\cat{Set}$-functors.

\begin{prop}\label{prop-pred-lift-sier}
  Suppose $\Sier \in \cat{C}$, then there is a bijective correspondence between
  $n$-ary open predicate liftings and elements of $\fun{\Omega T}\Sier^n$. This
  correspondence is given as follows: To an open predicate lifting $\lambda$
  assign the set
  $\lambda_{\Sier^n}(\pi_1^{-1}(\{ 1 \}), \ldots, \pi_n^{-1}(\{ 1 \}))
  \in \fun{\Omega T}\Sier^n$, where $\pi_i : \Sier^n \to \Sier$ is the $i$-th
  projection, and conversely, for $c \in \fun{\Omega T}\Sier^n$ define
  $\lambda^c : \fun{\Omega}^n \to \fun{\Omega T}$ by
  $\lambda^c_{\topo{X}}(a_1, \ldots, a_n) = (\fun{T}\langle \chi_{a_1},
  \ldots, \chi_{a_n}\rangle)^{-1}(c)$.
\end{prop}

  Furthermore, there is a bijective correspondence between open
  predicate liftings and continuous functions $\fun{T}\Sier^n \to \Sier$.
  This is established by identifying elements of $\fun{\Omega}\fun{T}\Sier^n$
  with their characteristic map $\fun{T}\Sier^n \to \Sier$.
  This view on predicate liftings has been investigated in~\cite[Section~7]{BalKurVel15}.

\begin{defi}\label{def:language}
  The \emph{language} induced by a geometric modal signature $\Lambda$ is the
  collection $\GML(\Phi, \Lambda)$ of formulas defined by the grammar
  \begin{equation*}
    \phi ::= \top \mid p
                  \mid \phi_1 \wedge \phi_2
                  \mid \bigvee_{i \in I} \phi_i
                  \mid \heartsuit^{\lambda}(\phi_1, \ldots, \phi_n),
  \end{equation*}
  where $p$ ranges over the set $\Phi$ of proposition letters, $I$ is some index
  set, and $\lambda \in \Lambda$ is $n$-ary. Abbreviate $\bot := \bigvee \emptyset$.
  We call a formula in $\GML(\Phi, \Lambda)$ \emph{coherent} if it does not involve any
  infinite disjunctions.
\end{defi}

\subsection{Interpretation and examples}
  The language $\GML(\Phi, \Lambda)$ is interpreted in so-called geometric
  $\fun{T}$-models.

\begin{defi}
  A \emph{geometric $\fun{T}$-model} is a triple $\mf{X} = (\topo{X}, \gamma, V)$
  where $(\topo{X}, \gamma)$ is a $\fun{T}$-coalgebra and
  $V : \Phi \to \fun{\Omega}\topo{X}$ is a valuation of the proposition letters.
  A map $f : \topo{X} \to \topo{X}'$ is a \emph{geometric $\fun{T}$-model morphism}
  from $(\topo{X}, \gamma, V)$ to $(\topo{X}', \gamma', V')$ if $f$ is a coalgebra
  morphism between the underlying coalgebras and $V = f^{-1} \circ V'$. The
  collection of geometric $\fun{T}$-models and geometric $\fun{T}$-model morphisms
  forms a category, which we denote by $\cat{Mod}(\fun{T})$.
\end{defi}

\begin{defi}\label{def-gml-sem}
  The \emph{semantics} of $\phi \in \GML(\Phi, \Lambda)$ on a geometric $\fun{T}$-model
  $\mf{X} = (\topo{X}, \gamma, V)$ is given recursively by
  \begin{align*}
    &\llb \top \rrb^{\mf{X}} = X, \quad
      \llb p \rrb^{\mf{X}} = V(p), \quad
      \llb \phi \wedge \psi \rrb^{\mf{X}}
        = \llb \phi \rrb^{\mf{X}} \cap \llb \psi \rrb^{\mf{X}}, \quad
      \llb \bigvee_{i \in I} \phi_i \rrb^{\mf{X}} = \bigcup_{i \in I} \llb \phi_i \rrb^{\mf{X}}, \\
    &\llb \heartsuit^{\lambda}(\phi_1, \ldots, \phi_n) \rrb^{\mf{X}}
        = \gamma^{-1}( \lambda_{\topo{X}}(\llb \phi_1 \rrb^{\mf{X}}, \ldots,
        \llb \phi_n \rrb^{\mf{X}})).
  \end{align*}
  We write $\mf{X}, x \Vdash \phi$ iff $x \in \llb \phi \rrb^{\mf{X}}$. Two
  states $x$ and $x'$ are called \emph{modally equivalent} if they satisfy the
  same formulas, notation: $x \modeq x'$.
  We say that $\phi$ is a \emph{semantic consequence} of $\psi$ in $\cat{Mod}(\fun{T})$,
  notation: $\phi \Vdash_{\fun{T}} \psi$, if $\llb \phi \rrb^{\mf{X}} \subseteq
  \llb \psi \rrb^{\mf{X}}$ for all $\mf{X} \in \cat{Mod}(\fun{T})$.
\end{defi}

  The following proposition shows that morphisms preserve truth. Its proof is
  similar to the proof of Theorem~6.17 in~\cite{Ven17}.

\begin{prop}\label{prop-mor-pres-truth}
  Let $\Lambda$ be a geometric modal signature for $\fun{T}$.
  Let $\mf{X} = (\topo{X}, \gamma, V)$ and $\mf{X}' = (\topo{X}', \gamma', V')$
  be geometric $\fun{T}$-models and let $f : \mf{X} \to \mf{X}'$ be a geometric
  $\fun{T}$-model morphism. Then for all $\phi \in \GML(\Phi, \Lambda)$ and
  $x \in \topo{X}$ we have
  \begin{equation*}
    \mf{X}, x \Vdash \phi \quad\text{iff}\quad \mf{X}', f(x) \Vdash \phi.
  \end{equation*}
\end{prop}

  We state the notion of behavioural equivalence for future reference.

\begin{defi}\label{def-beh-eq}
  Let $\mf{X} = (\topo{X}, \gamma, V)$ and $\mf{X}' = (\topo{X}', \gamma', V')$
  be two geometric $\fun{T}$-models and $x \in \topo{X}$, $x' \in \topo{X}'$ two
  states. We say that $x$ and $x'$ are \emph{behaviourally equivalent} in
  $\cat{Mod}(\fun{T})$, notation: $x \simeq_{\cat{Mod}(\fun{T})} x'$, if there
  exists a geometric $\fun{T}$-model $\mf{Y}$ and $\fun{T}$-model morphisms
  \begin{equation*}
    \begin{tikzcd}
      \mf{X} \arrow[r, "f" above] &\mf{Y} &\mf{X}' \arrow[l, "f'" above]
    \end{tikzcd}
  \end{equation*}
  such that $f(x) = f'(x')$.
\end{defi}

  As an immediate consequence of Proposition~\ref{prop-mor-pres-truth} we find
  that behavioural equivalence implies modal equivalence. We will see in
  Section~\ref{sec:final} that, under mild conditions, the converse is true
  as well.

  Let us give some concrete examples of functors.

\begin{exa}[Trivial functor]\label{exm-functor-F}
  Let $\two = \{ 0, 1 \}$ be topologised by $\{ \emptyset, \{ 0, 1 \} \}$
  (the trivial topology). Define the functor $\fun{F} : \cat{Top} \to \cat{Top}$
  by $\fun{F}\topo{X} = \two$ for every $\topo{X} \in \cat{Top}$ and
  $\fun{F}f = \id_{\two}$, the identity map on $\two$, for every
  continuous function $f$. This is clearly a functor.
  Consider the open predicate lifting
  $\lambda : \fun{\fun{\Omega}} \to \fun{\fun{\Omega}} \circ \fun{F}$ given by
  $\lambda_{\topo{X}}(a) = \fun{U}\two$ for all $a \in \fun{\Omega}\topo{X}$.
  For an $\fun{F}$-model $\mf{X} = (\topo{X}, \gamma, V)$ we then have
  $\mf{X}, x \Vdash \heartsuit^{\lambda}\phi$ iff
  $\gamma(x) \in \lambda(\llb \phi \rrb^{\mf{X}})$ iff
  $\llb \phi \rrb^{\mf{X}} \in \fun{\Omega}\topo{X}$.
  So $\heartsuit^{\lambda} = \top$.
\end{exa}

  Next we have a look at the Vietoris functor on $\cat{KHaus}$. Coalgebras for
  this functor have also been studied in~\cite{BBH15}, where they are used to
  interpret the positive modal logic from~\cite{Dun95,CelJan99}.
  In Section~\ref{sec:mon} we study the example of the \emph{monotone functor},
  which gives rise to monotone modal geometric logic.

\begin{exa}[Vietoris functor]\label{exm-viet}
  For a compact Hausdorff space $\topo{X}$, let $\fun{V}_{\cat{kh}}\topo{X}$
  be the collection of closed subsets of $\topo{X}$ topologized by the subbase
  \begin{equation*}
    \boxdot a := \{ b \in \fun{V}_{\cat{kh}}\topo{X} \mid b \subseteq a \},
    \quad \diamonddot a := \{ b \in \fun{V}_{\cat{kh}}\topo{X} \mid a \cap b
    \neq \emptyset \},
  \end{equation*}
  where $a$ ranges over $\fun{\Omega}\topo{X}$. For a continuous map $f : \topo{X}
  \to \topo{X}'$ define $\fun{V}_{\cat{kh}}f : \fun{V}_{\cat{kh}}\topo{X} \to
  \fun{V}_{\cat{kh}}\topo{X}'$ by $\fun{V}_{\cat{kh}}f(a) = f[a]$. If $\topo{X}$
  is compact Hausdorff, then so is
  $\fun{V}_{\cat{kh}}\topo{X}$~\cite[Theorem~4.9]{Mic51},
  and if $f : \topo{X} \to \topo{X}'$ is a continuous
  map between compact Hausdorff spaces, then $\fun{V}_{\cat{kh}}f$ is well defined
  and continuous~\cite[Lemma~3.8]{KupKurVen04}, so $\fun{V}_{\cat{kh}}$ defines an
  endofunctor on $\cat{KHaus}$.

  Let $\mf{X} = (\topo{X}, \gamma, V)$ be a $\fun{V}_{\cat{kh}}$-model. If we set
  \begin{equation*}
    \lambda^{\Box}_{\topo{X}}
      : \fun{\Omega}\topo{X} \to \fun{\Omega}(\fun{V}_{\cat{kh}}\topo{X})
      : a \mapsto \{ b \in \fun{V}_{\cat{kh}}\topo{X} \mid b \subseteq a \},
  \end{equation*}
  where $\topo{X} \in \cat{Top}$, then we have $\mf{X}, x \Vdash \Box\phi$ iff
  $\gamma(x) \in \lambda^{\Box}_{\topo{X}}(\llb \phi \rrb^{\mf{X}})$ iff
  $\gamma(x) \subseteq \llb \phi \rrb^{\mf{X}}$ iff every successor of $x$
  satisfies $\phi$. Similarly $\lambda^{\Diamond}_{\topo{X}} : \fun{\Omega}\topo{X}
  \to \fun{\Omega} \circ \fun{V}_{\cat{kh}}\topo{X}$, given by
  $\lambda^{\Diamond}_{\topo{X}}(a) = \diamonddot a$ yields the usual semantics
  of the diamond modality.
\end{exa}

\subsection{Strong predicate liftings}
  In Section~\ref{sec-bisim} it turns out to be useful to have a slightly
  stronger notion of open predicate liftings, called strong open predicate
  liftings, as this allows us to prove that behavioural equivalence implies
  so-called $\Lambda$-bisimilarity.
  Whereas the action of open predicate liftings is defined only on open subsets,
  a strong open predicate lifting acts on \emph{every} subset of elements of a
  topological space.
  Recall that $\fun{U} : \cat{Top} \to \cat{Set}$ denotes the forgetful functor.

\begin{defi}
  A \emph{strong open predicate lifting} for $\fun{T} : \cat{C} \to \cat{C}$ is
  a natural transformation
  \[
    \mu : (\breve{\fun{P}} \circ \fun{U})^n
      \to \breve{\fun{P}} \circ \fun{U} \circ \fun{T}
  \]
  such that for all $\topo{X} \in \cat{C}$ and $a_1, \ldots, a_n \in
  \fun{\Omega}\topo{X}$ the set $\mu_{\topo{X}}(a_1, \ldots, a_n)$ is open
  in $\fun{T}\topo{X}$. Monotonicity and Scott-continuity of strong open
  predicate liftings are defined in the standard way.

  We call an open predicate lifting (from Definition~\ref{def:open-pred-lift})
  \emph{extendable} if it is the restriction of some
  strong open predicate lifting and \emph{monotone extendable} if it is the
  restriction of a monotone strong open predicate lifting.
  We call a geometric modal signature $\Lambda$ \emph{(monotone) extendable}
  if all predicate liftings in it are (monotone) extendable.
\end{defi}

  Evidently, every strong open predicate lifting restricts to an open predicate
  lifting, and it is only this weaker notion of open predicate lifting that has
  an effect on the semantics.
  Our notion of strong open predicate lifting is similar to the notion of a
  \emph{topological predicate lifting} for endofunctors on $\cat{Stone}$, which
  was introduced in~\cite{EnqSou17}.

\begin{exa}
  The predicate lifting corresponding to the box modality from Example~\ref{exm-viet}
  is (monotone) extendable, for it is the restriction of $\mu : \fun{U} \to
  \fun{U} \circ \fun{V}_{\cat{kh}}$ given by $\mu_{\topo{X}}(u)
  = \{ b \in \fun{V}_{\cat{kh}}\topo{X} \mid b \subseteq u \}$.
  Likewise, all other predicate liftings from Examples~\ref{exm-functor-F}, \ref{exm-viet} and the monotone functor from Section~\ref{sec:mon} are extendable
  as well.
\end{exa}

  We devote the remainder of this section to investigating strong open predicate
  liftings.
  Recall from Example~\ref{exm-functor-F} that $\two$ denotes the two-element set
  with the trivial topology. We claim that natural transformations
  $\mu : (\breve{\fun{P}} \circ \fun{U})^n \to \breve{\fun{P}} \circ \fun{U}
  \circ \fun{T}$ correspond one-to-one with elements of $\fun{\breve{P}UT}\two$,
  provided $\two \in \cat{C}$: To a natural transformation $\mu$ associate
  the set $\mu_{\two}(p_1^{-1}(\{ 1 \}), \ldots, p_n^{-1}(\{ 1 \}))$, where
  $p_i : \two^n \to \two$ denotes the $i$-th projection.
  Conversely, for $c \in \fun{\breve{P}UT}\two$ define $\mu^c$ by
  $\mu^c_{\topo{X}}(a_1, \ldots, a_n) = (\fun{T}\langle \chi_{a_1}', \ldots,
  \chi_{a_n}' \rangle)^{-1}(c)$, where $\topo{X}$ is a topological space,
  $a \subseteq \fun{U}\topo{X}$ and $\chi_a' : \topo{X} \to \two$ is the
  characteristic map. Note that $\chi'_a$ is continuous regardless of whether $a$
  is open or not, hence $\fun{T}$ acts on all $\chi_a'$. Details of the bijection
  are left to the reader.

\begin{prop}\label{prop-ext-open}
  Let $\fun{T}$ be an endofunctor on $\cat{C}$ and suppose that $\cat{C}$ contains
  the spaces $\two$ and $\Sier$.
  Let $s : \Sier \to \two$ be the identity map and let
  $c \in \fun{\breve{P}UT}\two^n$. The natural transformation $\mu^c$ is a
  strong open predicate lifting if and only if $(\fun{T}s^n)^{-1}(c) \subseteq
  \fun{T}\Sier^n$ is open.
\end{prop}
\begin{proof}
  We give the proof for the case $n = 1$, the general case being similar.
  Left to right follows from the fact that $\{ 1 \}$ is open in $\Sier$, hence
  $\mu_{\Sier}^c(\{ 1 \}) = (\fun{T}\chi'_{\{ 1 \}})^{-1}(c) = (\fun{T}s)^{-1}(c)$
  must be open in $\fun{T}\Sier$.
  For the converse, let $\topo{X}$ be a topological space and
  $a \in \fun{\Omega}\topo{X}$. We need to show that $\mu^c_{\topo{X}}(a)$ is open.
  Since $a$ is open, the characteristic map $\chi_a: \topo{X} \to \Sier$ is
  continuous and hence $\chi'_a = s \circ \chi_a$. We have
  \begin{align*}
    \mu^c_{\topo{X}}(a)
      &= (\fun{T}\chi_a')^{-1}(c) &\text{(Definition of $\mu^c$)} \\
      &= (\fun{T}(s \circ \chi_a))^{-1}(c) &\text{($\chi'_a = s \circ \chi_a$)} \\
      &= (\fun{T}s \circ \fun{T}\chi_a)^{-1}(c) &\text{(Definition of functors)} \\
      &= (\fun{T}\chi_a)^{-1} \circ (\fun{T}s)^{-1}(c). &\text{(Definition of inverse)}
  \end{align*}
  Since $\fun{T}\chi_a$ is continuous and $(\fun{T}s)^{-1}(c)$ is assumed to be
  open in $\fun{T}\Sier$, the set $\mu^c_{\topo{X}}(a)$ is open in
  $\fun{T}\topo{X}$.
\end{proof}

  The following proposition gives two sufficient conditions on $\fun{T}$ for its
  open predicate liftings to be extendable. For a full subcategory $\cat{C}$ of
  $\cat{Top}$ let $\cat{preC}$ denote the category of topological spaces in
  $\cat{C}$ and (not necessarily continuous) functions.

\begin{prop}
  Let $\fun{T}$ be an endofunctor on $\cat{C}$ and suppose $\two, \Sier \in
  \cat{C}$.
  \begin{enumerate}
    \item If $\fun{T}$ preserves injective functions then every open predicate
          lifting for $\fun{T}$ is extendable.
    \item If $\fun{T}$ extends to $\cat{preC}$, then every open predicate
          lifting for $\fun{T}$ is extendable.
  \end{enumerate}
\end{prop}
\begin{proof}
  For the first item, let $c \in \fun{\Omega T}\Sier^n$ determine the $n$-ary
  open predicate lifting $\lambda^c$. Since $s^n$ is injective, by assumption
  $\fun{T}s^n$ is as well, and hence $c = (\fun{UT}s^n)^{-1}((\fun{UT}s^n)[c])$.
  Proposition~\ref{prop-ext-open} now implies that $\mu^{(\fun{UT}s^n)[c]}$ is a
  strong open predicate lifting. It is easy to see that $\mu^{(\fun{UT}s^n)[c]}$
  extends $\lambda^c$, hence the latter is extendable.

  For the second item we show that, under the assumption, $\fun{T}$ preserves
  injective functions. Let $f : \topo{X} \to \topo{Y}$ be an injective function
  in $\cat{C}$, then there exists a (not necessarily continuous) function
  $g : \topo{Y} \to \topo{X}$ satisfying $g \circ f = \id_{\topo{X}}$. Then
  $\fun{T}g \circ \fun{T}f = \fun{T}(g \circ f) = \fun{T}\id_{\topo{X}}
  = \id_{\fun{T}\topo{X}}$, so $\fun{T}f$ has a (set-theoretic) left-inverse,
  hence is injective.
\end{proof}

  Monotone open predicate lifting (hence also Scott-continuous ones) for an
  endofunctor on $\cat{KHaus}$ are always extendable:

\begin{prop}\label{prop-strong-pred-lift}
  Let $\fun{T}$ be an endofunctor on $\cat{KHaus}$ and $\Lambda$ a monotone
  geometric modal signature for $\fun{T}$. Then $\Lambda$ is monotone extendable.
\end{prop}
\begin{proof}
  Let $\lambda \in \Lambda$. We need to show that $\lambda$ is the restriction
  of some monotone strong predicate lifting. Define
  \begin{equation*}
    \tilde{\lambda}_{\topo{X}}
      : \breve{\fun{P}}^n\fun{U}\topo{X} \to \breve{\fun{P}}\fun{UT}\topo{X}
      : (b_1, \ldots, b_n) \mapsto \bigcap \{ \lambda_{\topo{X}}(a_1, \ldots, a_n)
        \mid a_i \in \fun{\Omega}\topo{X} \text{ and } a_i \supseteq b_i \}.
  \end{equation*}
  Monotonicity of $\lambda_{\topo{X}}$ ensures $\tilde{\lambda}_{\topo{X}}(a)
  = \lambda_{\topo{X}}(a)$ for all $a \in \fun{\Omega}\topo{X}$ and
  $\tilde{\lambda}$ is monotone by construction. So we only need to show that
  $\tilde{\lambda}$ is indeed a strong open predicate lifting, i.e.~a natural
  transformation $\breve{\fun{P}}^n\fun{U}\topo{X} \to
  \breve{\fun{P}}\fun{UT}\topo{X}$. We assume $\lambda$ to be unary,
  the general case being similar.

  For a continuous map $f : \topo{X} \to \topo{X}'$ between compact Hausdorff
  spaces we need to show that
  $\tilde{\lambda}_{\topo{X}} \circ f^{-1}
  = (\fun{T}f)^{-1} \circ \tilde{\lambda}_{\topo{X}'}$. Since, by naturality
  of $\lambda$, the right hand side is equal to
  $\bigcap \{ \lambda_{\topo{X}}(f^{-1}(a'))
  \mid a' \in \fun{\Omega}\topo{X}' \text{ and } b' \subseteq a' \}$,
  it suffices to show
  \begin{equation}\label{eq1}
    \bigcap \{ \lambda_{\topo{X}}(c) \mid c \in \fun{\Omega}\topo{X}
    \text{ and }  f^{-1}(b') \subseteq c \}
	= \bigcap \{ \lambda_{\topo{X}}(f^{-1}(a')) \mid a' \in \fun{\Omega}\topo{X}'
	\text{ and } b' \subseteq a' \}.
  \end{equation}
  If $a'$ is an open superset of $b'$ then clearly $f^{-1}(b') \subseteq
  f^{-1}(a')$. So every element in the intersection of the right hand side is
  contained in the one on the left hand side and therefore we have $\subseteq$
  in~\eqref{eq1}.
  For the converse, suppose $c \in \fun{\Omega}\topo{X}$ and $f^{-1}(b')
  \subseteq c$. Then the set $a' = X' \setminus f[X \setminus c]$ is open,
  contains $b'$, and satisfies $f^{-1}(b') \subseteq f^{-1}(a') \subseteq c$.
  Therefore $\lambda_{\topo{X}}(f^{-1}(a'))$ is one of the elements in the
  intersection on the left hand side of~\eqref{eq1}. Since
  $\lambda_{\topo{X}}(f^{-1}(a')) \subseteq \lambda_{\topo{X}}(c)$ this shows
  ``$\supseteq$'' in~\eqref{eq1}.
\end{proof}

\section{The monotone neighbourhood functor on \texorpdfstring{$\cat{KHaus}$}{KHaus}}\label{sec:mon}

  In this section we define the monotone neighbourhood functor on $\cat{Frm}$ and
  show that it (individually) preserves regularity and compactness. This functor
  is a variation of the Vietoris Locale~\cite[Section~1]{Joh85}.
  Subsequently, we give a functor on $\cat{KHaus}$ which is dual to the
  restriction of the monotone neighbourhood functor to $\cat{KRFrm}$.

\subsection{The monotone neighbourhood frame}

\begin{defi}\label{def-functor-M}
  For a frame $F$, let $\fun{M}F$ be the frame generated by $\Box a, \Diamond a$,
  where $a$ ranges over $F$, subject to the relations
  \begin{multicols}{2}
    \begin{enumerate}[($M_1$)]
      \item\label{eq-fun-M-1} $\Box (a \wedge b) \leq \Box a$
      \item\label{eq-fun-M-2} $\Box a \wedge \Diamond b = \bot \;
                              \text{ whenever } a \wedge b = \bot$
      \item\label{eq-fun-M-3} $\Box \bigveeup A = \bigveeup
                              \{ \Box a \mid a \in A \}$
      \item\label{eq-fun-M-4} $\Diamond a \leq \Diamond(a \vee b)$
      \item\label{eq-fun-M-5} $\Box a \vee \Diamond b = \top \; \text{ whenever }
                              a \vee b = \top$
      \item\label{eq-fun-M-6} $\Diamond \bigveeup A = \bigveeup \{ \Diamond a
                              \mid a \in A \}$,
    \end{enumerate}
  \end{multicols}
  \noindent
  where $a, b \in F$ and $A$ is a directed subset of $F$.
  For a homomorphism $f : F \to F'$ define $\fun{M}f : \fun{M}F \to \fun{M}F'$ on
  generators by $\Box a \mapsto \Box f(a)$ and $\Diamond a \mapsto \Diamond f(a)$.
  The assignment $\fun{M}$ defines a functor on $\cat{Frm}$.
\end{defi}

  The proof of the following proposition closely resembles that of Proposition~%
  III4.3 in~\cite{Joh82}. In a similar manner one can show that $\fun{M}$
  preserves complete regularity and zero-dimensionaity.

\begin{prop}\label{prop-M-pres-reg}
  If $F$ is a regular frame, then so is $\fun{M}F$.
\end{prop}
\begin{proof}
  We need to show that for all $c \in \fun{M}F$ we have $c = \bigvee \{ d \in
  \fun{M}F \mid d \wi c \}$. It follows from Lemma~\ref{lem-reg-elt} that it
  suffices to focus on the generators of $\fun{M}F$. Let $a \in F$, then we know
  $\bigvee \{ d \in \fun{M}F \mid d \wi \Box a \} \leq \Box a$. Suppose $b \wi a$
  in $F$, then by Lemma~\ref{lem-wi-nega} $\nega b \vee a = \top$ and hence
  $\Diamond\nega b \vee \Box a = \top$. Also $\nega b \wedge b = \bot$ so it
  follows from~\ref{eq-fun-M-2} that $\Diamond\nega b \wedge \Box b = \bot$. This
  proves $\Box b \wi \Box a$, because the element $\Diamond\nega b$ is such that
  $\Diamond\nega b \vee \Box a = \top$ and $\Diamond\nega b \wedge \Box b =
  \bot$. Since $F$ is regular and $\{ b \in F \mid b \wi a \}$ is directed, it
  follows that
  \begin{equation*}
    \Box a
      = \Box \bigveeup \{ b \in F \mid b \wi a \}
      = \bigveeup \{ \Box b \in \fun{M}F \mid b \wi a \}
      \leq \bigvee \{ d \in \fun{M}F \mid d \wi \Box a \}
      \leq \Box a
  \end{equation*}
  so $\Box a = \bigvee \{ d \in \fun{M}F \mid d \wi \Box a \}$. In a similar
  fashion one may show that $\Diamond a = \bigvee \{ d \in \fun{M}F \mid d \wi
  \Diamond a \}$. This proves the proposition.
\end{proof}

  We now prove that the functor $\fun{M}$ preserves compactness. We proceed
  in a similar manner as~\cite[Theorem~4.2]{VVV13}. This relies on an auxiliary definition
  and lemma (Definition~\ref{def:M-prime} and Lemma~\ref{lem:M-prime}), in
  which we give an alternative description of $\fun{M}F$. We then prove that
  this alternative description preserves compactness.

  In~\cite[Corollary~3.42]{Gro18} we proved the same result by first giving a duality
  result between frames and topological spaces, and then proving preservation
  of compactness on the topological side. The main difference between that
  proof and the one we present here is that the current one is constructive.

  Write $\fun{P}_{\omega}$ of the finite powerset functor and recall that
  $\fun{Z} : \cat{Frm} \to \cat{Set}$ is the forgetful functor.

\begin{defi}\label{def:M-prime}
  For a frame $F$ define $\fun{M'}F$ to be the free frame generated by
  $\fun{P}_{\omega}\fun{Z}F \times \fun{P}_{\omega}\fun{Z}F$, qua join-semilattice
  (that is, the
  join in $\fun{M'}F$ is given by $(\gamma, \delta) \vee (\gamma', \delta')
  = (\gamma \cup \gamma', \delta \cup \delta')$),
  subject to
  \begin{multicols}{2}
    \begin{enumerate}[($M'_1$)]
      \item\label{eq-fun-Mp-1} 
            $(\gamma \cup \{ a \wedge b \}, \delta) \leq (\gamma \cup \{ a \}, \delta)$
      \item\label{eq-fun-Mp-2} 
            $(\gamma \cup \{ a \}, \delta) \wedge (\gamma, \delta \cup \{ b \})
            \leq (\gamma, \delta)$ \\ if $a \wedge b = 0$
      \item\label{eq-fun-Mp-3} 
            $(\gamma \cup \{ \bigveeup A \}, \delta) \leq \bigveeup_{a \in A}
            (\gamma \cup \{ a \}, \delta)$
      \item\label{eq-fun-Mp-4} 
            $(\gamma, \delta \cup \{ a \}) \leq (\gamma, \delta \cup \{ a \vee b \})$
      \item\label{eq-fun-Mp-5} 
            $\top \leq (\gamma \cup \{ a \}, \delta \cup \{ b \})$ \\ if $a \vee b = 1$
      \item\label{eq-fun-Mp-6} 
            $(\gamma, \{ \bigveeup A \} \cup \delta) \leq
            \bigveeup_{a \in A}(\gamma, \{ a \} \cup \delta)$
    \end{enumerate}
  \end{multicols}
  \noindent
\end{defi}

  This results in a frame isomorphic to $\fun{M}F$:

\begin{lem}\label{lem:M-prime}
  Let $F$ be a frame. Then $\fun{M}F \cong \fun{M'}F$.
\end{lem}
\begin{proof}
  Define $\fun{M'}F \to \fun{M}F : (\gamma, \delta) \mapsto \bigvee_{c \in \gamma}
  \Box c \vee \bigvee_{d \in \delta} \Diamond d$ and
  \[
    \fun{M}F \to \fun{M'}F : \left\{\begin{array}{l}
                               \Box a \mapsto (\{ a \}, \emptyset) \\
                               \Diamond a \mapsto (\emptyset, \{ a \})
                             \end{array}\right.
  \]
  Clearly these define a bijection. Furthermore it
  is straightforward to verify that these maps are well defined by checking that
  the images of generators satisfy relations of the respective frame.
\end{proof}

\begin{thm}\label{thm-M-cpt}
  Suppose $F$ is compact. Then $\fun{M}F$ is compact.
\end{thm}
\begin{proof}
  The frame $\fun{M}D$ is compact iff there is a preframe homomorphism
  $\phi : \fun{M}D \to 2$ that is right adjoint to the unique frame homomorphism
  $! : 2 \to \fun{M}D$, where $2 = \{ 0, 1 \}$ is the two-element frame.

  By Proposition~\ref{lem:M-prime} we have $\fun{M}D \cong \fun{M'}D$, and since
  all the relations in Definition~\ref{def:M-prime} are join-stable, we can use the
  preframe coverage theorem~\cite[Theorem~5.1]{JohVic06} to find that $\fun{M'}F$ viewed
  as a preframe is the preframe generated by $\fun{P}_{\omega}\fun{Z}F \times
  \fun{P}_{\omega}\fun{Z}F$ qua poset, subject to the
  the relations from Definition~\ref{def:M-prime}.

  Define
  \[
    \phi : \fun{M'}F \to 2
         : (\gamma, \delta) \mapsto
           \left\{\begin{array}{ll} 1 & \text{ iff there exists a $c \in \gamma$ such
                                               that $c \vee (\bigvee \delta) = 1$} \\
                                    0 & \text{ otherwise}
           \end{array} \right. . 
  \]
  First we check that $\phi$ is indeed a pre-frame homomorphism. Since $\phi$ is
  defined on generators, it suffices to show that it preservers the
  relations~\ref{eq-fun-Mp-1} to~\ref{eq-fun-Mp-6}, because if it does it can be lifted in a
  unique way to a preframe homomorphism $\fun{M'}F \to 2$.
  It is clear that $\phi$ is a monotone morphism (hence preserves the poset
  structure of the generators). We check that $\phi$ preserves the
  relations one by one.
  \begin{enumerate}[($M'_1$)]
    \item If $\phi(\gamma \cup \{ a \}, \delta) = 0$ then $c \vee \bigvee \delta = \top_F$
          for all $c \in \gamma$ and $(a \wedge b) \vee \bigvee \delta \leq a \vee
          \bigvee \delta = \top_F$.

    \item Suppose $\phi(\gamma \cup \{a \}, \delta) = 1$ and $\phi(\gamma, \delta
          \cup \{ b \}) = 1$. Then either there is some $c \in \gamma$ such that
          $c \vee \bigvee \delta = \top_F$, which implies $\phi(\gamma, \delta) = 1$, or
          $a \vee \bigvee \delta = \top_F$. In the latter case, note that we also have some
          $c' \in \gamma$ such that $c' \vee \bigvee \delta \vee b = \top_F$, so that
          \[
            c' \vee \bigvee \delta = c' \vee \bigvee \delta \vee (a \wedge b) =
            (a \vee \bigvee \delta \vee c') \wedge (c' \vee \bigvee \delta \vee b) =
            \top_F \wedge \top_F = \top_F.
          \]
          The first equality holds because $a \wedge b = \bot_F$. Again we find
          $\phi(\gamma, \delta) = 1$.

    \item Suppose $\phi(\gamma \cup \{\bigveeup A \}, \delta) = 1$, then either
          $c \vee (\bigvee \delta) = \top_F$ for some $c \in \gamma$, or $\top_F = (\bigveeup A)
          \vee (\bigvee \delta) = \bigveeup_{a \in A}(a \vee (\bigvee \delta))$ (note
          that the latter is indeed a directed set, because $A$ is). By compactness of
          $F$ this gives $a \vee (\bigvee \delta) = \top_F$ for some $a \in A$. So both
          cases yield $\phi\big(\bigveeup_{a \in A}(\gamma \cup \{ a \}, \delta)\big) = 1$.

    \item If $\phi(\gamma, \delta \cup \{ a \}) = 1$, then $c \vee \bigvee (\delta
          \cup \{ a\}) = \top_F$ for some $c \in \gamma$, so $c \vee \bigvee (\delta \cup
          \{ a \vee b\}) \geq c \vee \bigvee (\delta \cup \{ a \}) = \top_F$.

    \item If $a \vee b = \top_F$, then $a \vee \bigvee (\delta \cup \{ b \}) = \top_F$ so
          $\phi(\gamma \cup \{ a \}, \delta \cup \{ b \}) = 1$.

    \item Suppose $\phi(\gamma, \{ \bigveeup A \} \cup \delta) = 1$, then, for some
          $c \in \gamma$, we have
          \[
            \top_F = c \vee \bigvee(\{ \bigveeup A \} \cup \delta)
              = \bigveeup(c \vee a \vee \bigvee \delta)
          \]
          and by compactness we must have $c \vee \bigveeup(\{ a \} \cup \delta) = \top_F$
          for one of the $a$.
          (The set $\{ c \vee a \vee \bigvee \delta \mid a \in A \}$ is directed
          and by~\ref{eq-fun-Mp-4}.)
\end{enumerate}

\noindent
  Lastly, we need to verify that $\phi$ is right-adjoint to $! : 2 \to \fun{M'}L$
  (defined by $1 \mapsto \top_{\fun{M'}F} = (\{ \top_F \}, \{ \top_F \})$, and
  $0 \mapsto \bot_{\fun{M'}F} = (\emptyset, \emptyset)$). It suffices to show that
  $\phi(!(p)) \geq p$ and $!(\phi(\gamma, \delta)) \leq (\gamma, \delta)$. For the first,
  suppose $p = 1$, then $!(p)$ is the equivalence class of $(\{\top_F\}, \{ \top_F\})$
  and $\phi(!(p)) = 1$. For the second, if $\phi(\gamma, \delta) = 1$, then there are
  $c \in \gamma$ such that $c \vee (\bigvee \delta) = \top_F$ (in particular
  $\delta \neq \emptyset$) and hence
  \[
    \top_{\fun{M'}F} = (\{ \top_F \}, \delta)
                     = (\{ c \vee (\bigvee \delta) \}, \delta)
                     \leq (\{ c \}, \delta)
                     \leq (\gamma, \delta).
  \]
  The first inequality follows from recalling that $\delta$ is a finite set and
  applying~\ref{eq-fun-Mp-6} repeatedly.
  This completes the proof.
\end{proof}

  We now know that $\fun{M}$ restricts to an endofunctor on $\cat{KRFrm}$.
  We write $\fun{M}_{\cat{kr}}$ for this restriction.

\begin{rem}
  The category $\cat{Loc}$ of locales and locale morphisms is the opposite of
  $\cat{Frm}$. Therefore, we can also view $\fun{M}$ as an endofunctor on
  locales and $\fun{M}A$ as the monotone neighbourhood locale, where $A$ is
  a locale.
\end{rem}

\subsection{Monotone neighbourhood functor on \texorpdfstring{$\cat{KHaus}$}{KHaus}}

  We now describe the topological manifestation of the monotone neighbourhood
  functor.

\begin{defi}\label{def-monotone-functor}\label{def-mon-khaus}
  Let $\topo{X} = (X, \tau)$ be a compact Hausdorff space. Let
  $\fun{D}_{\cat{kh}}\topo{X}$ be the collection of sets $W \subseteq \fun{P}X$
  such that $u \in W$ iff there exists a closed $c \subseteq u$ such that every
  open superset of $c$ is in $W$. Endow $\fun{D}_{\cat{kh}}\topo{X}$ with the
  topology generated by the subbase
  \begin{equation*}
    \dbox a := \{ W \in \fun{D}_{\cat{kh}}\topo{X} \mid a \in W \},
    \quad
    \ddiamond a := \{ W \in \fun{D}_{\cat{kh}}\topo{X} \mid X \setminus a
      \notin W \},
  \end{equation*}
  where $a$ ranges over $\fun{\Omega}\topo{X}$.
  For continuous functions $f : \topo{X} \to \topo{X}'$ define
  $\fun{D}_{\cat{kh}}f
    : \fun{D}_{\cat{kh}}\topo{X} \to \fun{D}_{\cat{kh}}\topo{X}'
    : W \mapsto \{ a \in \fun{P}X \mid f^{-1}(a) \in W \}$.
\end{defi}

\begin{lem}\label{lem-Df-well}
  If $f : \topo{X} \to \topo{X}'$ is a morphism in $\cat{KHaus}$, then
  $\fun{D}_{\cat{kh}}f$ is a well-defined continuous function from
  $\fun{D}_{\cat{kh}}\topo{X}$ to $\fun{D}_{\cat{kh}}\topo{X}'$.
\end{lem}
\begin{proof}
  {\it $\fun{D}_{\cat{kh}}f$ is well-defined.} \; Let
  $W \in \fun{D}_{\cat{kh}}\topo{X}$. We need to show that $\fun{D}_{\cat{kh}}f(W)
  \in \fun{D}_{\cat{kh}}\topo{X}'$. Suppose $a' \in \fun{D}_{\cat{kh}}f(W)$.
  Then $f^{-1}[a'] \in W$, so there exists a closed $c \subseteq f^{-1}[a']$ such
  that $c \in W$. Since $\topo{X}$ is compact and $\topo{X}'$ is Hausdorff,
  $f[c]$ is a closed subset of $\topo{X}'$. In addition we have $f[c] \subseteq a'$.
  Suppose $f[c] \subseteq b$ for some open $b \in \fun{\Omega}\topo{X}'$, then
  $c \subseteq f^{-1}[b]$ so $f^{-1}[b] \in W$ and hence
  $b \in \fun{D}_{\cat{kh}}f(W)$. So all open supersets of $f[c]$ are in
  $\fun{D}_{\cat{kh}}f(W)$, and therefore $f[c] \in \fun{D}_{\cat{kh}}f(W)$.
  Thus, for $a' \in \fun{D}_{\cat{kh}}f(W)$, there exists a closed subset (in
  this case $f[c]$) of $a'$ with the property that every open superset is in
  $\fun{D}_{\cat{kh}}f(W)$.

\bigskip\noindent
  {\it $\fun{D}_{\cat{kh}}f$ is continuous.} \; For continuity we need to show
  that both $(\fun{D}_{\cat{kh}}f)^{-1}(\dbox a')$ and
  $(\fun{D}_{\cat{kh}}f)^{-1}(\ddiamond a')$ are open in
  $\fun{D}_{\cat{kh}}\topo{X}$, whenever $a' \in \fun{\Omega}(\topo{X}')$. It
  follows from a straightforward computation that $(\fun{D}_{\cat{kh}}f)^{-1}
  (\dbox a') = \dbox f^{-1}(a')$, which is open in $\fun{D}_{\cat{kh}}\topo{X}$
  by definition, and similarly $(\fun{D}_{\cat{kh}}f)^{-1}(\ddiamond a') =
  \ddiamond f^{-1}(a') \in \fun{\Omega}\fun{D}_{\cat{kh}}\topo{X}$.
\end{proof}

  For the time being, we regard $\fun{D}_{\cat{kh}}$ as a functor
  $\cat{KHaus} \to \cat{Top}$, because we have no evidence yet that it
  restricts to an endofunctor on $\cat{KHaus}$.
  We aim to prove that $\fun{D}_{\cat{kh}}$ is dual to the restriction of
  $\fun{M}$ to $\cat{KRFrm}$. As a corollary, we then obtain that
  $\fun{D}_{\cat{kh}}$ indeed restricts to $\cat{KHaus}$.

\begin{thm}\label{thm:D-M-dual-obj}
  If $\topo{X}$ is a compact Hausdorff space then
  \begin{equation*}
    \fun{pt}(\fun{M}(\fun{opn}\topo{X})) \cong \fun{D}_{\cat{kh}}\topo{X}.
  \end{equation*}
\end{thm}

  We temporarily fix a compact Hausdorff space $\topo{X}$ and define the two maps
  constituting a homeomorphism.

\begin{defi}
  For a compact Hausdorff space $\topo{X}$, define $\zeta : \fun{pt} \circ \fun{M}
  \circ \fun{opn}\topo{X} \to \fun{D}_{\cat{kh}}\topo{X}$ by sending a prime
  filter $p$ to
  \begin{equation*}
    W_p := {\uparrow}\{ X \setminus a \mid p(\Diamond a) = \bot \}.
  \end{equation*}
\end{defi}

  We have $W_p \in \fun{D}_{\cat{kh}}\topo{X}$ because it is the up-set of a
  collection of closed sets; indeed, for each $b \in W_p$ there exists a closed
  subset $X \setminus a \subseteq b$ with $p(\Diamond a) = \bot$ and by
  definition all open supersets of $X \setminus a$ are in $W_p$. Therefore
  $\zeta$ is well defined.
  In the converse direction we define:

\begin{defi}\label{def:theta}
  For a compact Hausdorff space $\topo{X}$, define
  \[
    \theta : \fun{D}_{\cat{kh}}\topo{X} \to \fun{pt} \circ \fun{M} \circ
    \fun{opn}\topo{X} : W \mapsto p_W,
  \]
  where $p_W$ is given on generators by
  \[
    p_W
      : \fun{M} \circ \fun{opn}\topo{X} \to 2
      : \left\{\begin{array}{ll}
          \Box a \mapsto \top &\text{ iff } a \in W  \\
          \Diamond a \mapsto \bot &\text{ iff } X \setminus a \in W
        \end{array}\right.
  \]
\end{defi}

\begin{lem}
  The assignment $\theta$ is well defined.
\end{lem}
\begin{proof}
  Since $p_W$ is a frame homomorphisms defined on generators, it suffices to
  check that the $p_W(\Box a)$ and $p_W(\Diamond a)$ (where the $a$ range over
  $\fun{\Omega}\topo{X}$) satisfy~\ref{eq-fun-M-1} through~\ref{eq-fun-M-6} from
  Definition~\ref{def-functor-M}. Let us check~\ref{eq-fun-M-1}, \ref{eq-fun-M-2}
  and~\ref{eq-fun-M-3}, items~\ref{eq-fun-M-4}, \ref{eq-fun-M-5}
  and~\ref{eq-fun-M-6} being similar.
  \begin{enumerate}[($M_1$)]
    \item If $p_W(\Box(a \cap b)) = \top$ then $a \cap b \in W$. Since $W$ is
          upward closed $a \in W$, so $p_W(\Box a) = \top$.
    \item If $a \cap b = \emptyset$ then $a \subseteq X \setminus b$. Suppose
          $p_W(\Box a) = \top$ then $a \in W$ so $X \setminus b \in W$ so
          $p_W(\Diamond b) = \bot$ hence $p_W(\Box a) \wedge p_W(\Diamond b) =
          \bot$.
    \item We claim that for all $W \in \fun{D}_{\cat{kh}}\topo{X}$ and directed
          sets $A \subseteq \fun{\Omega}\topo{X}$ we have $\bigcupup A \in W$ iff
          there is $a \in A$ with $a \in W$. The direction from right to left
          follows from the fact that $W$ is upwards closed. Conversely, suppose
          $\bigcupup A \in W$, then there is a closed set $k \subseteq \bigcupup
          A$ with $k \in W$. The elements of $A$ now cover the closed therefore
          compact set $k$, so there is a finite $A' \subseteq A$ with
          $k \subseteq \bigcup A'$ and since $A$ is directed there is $a \in A$
          with $\bigcup A' \subseteq a$. As $k \in W$ and $k \subseteq a$ it
          follows that $a \in W$.

          Now we have $p_W(\Box \bigcupup A) = \top$ iff $\bigcupup A \in W$ iff
          there is $a \in A$ with $a \in W$ iff $\bigveeup \{ p_W(\Box a) \mid a
          \in A \} = \top$. \qedhere
  \end{enumerate}
\end{proof}

\noindent
  The following lemma is key for proving that $\zeta$ and $\theta$ are
  continuous and each other's inverses.

\begin{lem}\label{lem:DM-2}
  For all $p \in \fun{pt} \circ \fun{M} \circ \fun{opn}\topo{X}$ we have
  $X \setminus a \in W_p$ iff $p(\Diamond a) = \bot$ and $a \in W_p$ iff
  $p(\Box a) = \top$.
\end{lem}
\begin{proof}
  If $p(\Diamond a) = \bot$ then $X \setminus a \in W_p$. Conversely, Suppose
  $X \setminus a \in W_p$, then there is some $b$ with $p(\Diamond b) = \bot$ and
  $X \setminus b \subseteq X \setminus a$. Therefore $a \subseteq b$ and
  $p(\Diamond a) \leq p(\Diamond b) = \bot$. This proves $X \setminus a \in W_p$
  iff $p(\Diamond a) = \bot$.

  If $a \in W_p$ then there is $X \setminus b \subseteq a$ in $W_p$, so
  $p(\Diamond b) = \bot$. Then $a \cup b = X$, so it follows
  from~\ref{eq-fun-M-5} of Definition~\ref{def-functor-M} that $p(\Box a) = \top$.
  If $a \notin W_p$ and $a' \wi a$, then there exists $b$ with $b \cap a' =
  \emptyset$ and $b \cup a = \topo{X}$. Since $X \setminus b \subseteq a$, set
  set $X \setminus b$ is not in $W_p$ and hence we must have $p(\Diamond b) =
  \top$. As $a' \cap a = \emptyset$ it follows from~\ref{eq-fun-M-2} that
  $p(\Box a') = p(\emptyset) = \bot$. Now we use~\ref{eq-fun-M-3} and the fact
  that $a = \bigveeup \{ a' \mid a' \wi a \}$ (this is true because $\topo{X}$ is
  assumed to be compact Hausdorff so $\fun{opn}\topo{X}$ is compact regular) to find
  \begin{equation*}
    p(\Box a)
      = \bigveeup \{ p(\Box a') \mid a' \wi a \}
      = \bigveeup \{ \bot \mid  a' \wi a \}
      = \bot.
  \end{equation*}
  It follows that $a \in W_p$ iff $p(\Box a) = \top$.
\end{proof}

  We have now aquired sufficient knowledge to prove Theorem~\ref{thm:D-M-dual-obj}.

\begin{proof}[Proof of Theorem~\ref{thm:D-M-dual-obj}]
  We claim that the maps $\zeta$ and $\theta$ define a homeomorphism between
  $\fun{D}_{\cat{kh}}\topo{X}$ and $\fun{pt}(\fun{M}(\fun{opn}\topo{X}))$. First we prove
  that they are each other's inverses, by showing that for all
  $p \in \fun{pt}(\fun{M}(\fun{opn}\topo{X}))$ and $W \in \fun{D}_{\cat{kh}}\topo{X}$ we
  have $p_{W_p} = p$ and $W_{p_W} = W$.

  In order to prove that (the frame homomorphisms) $p$ and $p_{W_p}$ coincide, it
  suffices to show that they coincide on the generators of
  $\fun{M}(\fun{opn}\topo{X})$. By Definition~\ref{def:theta} and Lemma~\ref{lem:DM-2} we have
  \begin{equation*}
    p(\Box a) = \top
      \iff a \in W_p
      \iff p_{W_p}(\Box a) = \top
  \end{equation*}
  and
  \begin{equation*}
    p(\Diamond a) = \bot
      \iff X \setminus a \notin W_p
      \iff p_{W_p}(\Diamond a) = \bot.
  \end{equation*}

  In order to show that $W = W_{p_W}$ it suffices to show that
  $X \setminus a \in W$ iff $X \setminus a \in W_{p_W}$ for all open sets $a$,
  because elements of $\fun{D}_{\cat{kh}}\topo{X}$ are uniquely determined by the
  closed sets they contain. This follows immediately from the definitions and
  Lemma~\ref{lem:DM-2}, as
  \begin{equation*}
    X \setminus a \in W
      \iff p_W(\Diamond a) = \bot
      \iff X \setminus a \in W_{p_W}.
  \end{equation*}
  Therefore $\zeta = \theta^{-1}$.

  We complete the proof by showing that $\zeta$ and $\theta$ are continuous.
  The opens of $\fun{pt}(\fun{M}(\fun{opn}\topo{X}))$ are generated by
  $\widetilde{\Box a} = \{ p \mid p(\Box a) = \top \}$ and
  $\widetilde{\Diamond a} = \{ p \mid p(\Diamond a) = \top \}$, for
  $a \in \fun{\Omega}\topo{X}$. We have
  \begin{align*}
    \theta^{-1}(\tilde{\Box a})
      = \theta^{-1}(\{ p \mid p(\Box a) = \top\})
      = \{ W \in \fun{D}_{\cat{kh}}\topo{X} \mid a \in W \}
      = \dbox a
  \end{align*}
  and similarly $\theta^{-1}(\tilde{\Diamond a}) = \ddiamond a$.
  Continuity of $\theta$ follows from the fact that $\dbox a$ and $\ddiamond a$
  are open in $\fun{D}_{\cat{kh}}\topo{X}$.
  Conversely, the opens of $\fun{D}_{\cat{kh}}\topo{X}$ are generated by $\dbox a$ and
  $\ddiamond a$, where $a$ ranges over $\fun{\Omega}\topo{X}$. It is routine to
  see that $\zeta^{-1}(\dbox a) = \widetilde{\Box a}$ and $\zeta^{-1}(\ddiamond a)
  = \widetilde{\Diamond a}$. This proves continuity of $\zeta$.

  We showed that $\theta$ is a continuous function with continuous inverse $\zeta$,
  hence a homeomorphism. This completes the proof of the theorem.
\end{proof}

\begin{cor}
  The assignment $\fun{D}_{\cat{kh}}$ defines an endofunctor on $\cat{KHaus}$.
\end{cor}

  Theorem~\ref{thm:D-M-dual-obj}
  yields a map $\fun{M}_{\cat{kr}}(\fun{opn}\topo{X}) \to \fun{opn}(\fun{D}_{\cat{kh}}\topo{X})$
  for a compact Hausdorff space $\topo{X}$ given by
  \begin{equation*}
    \begin{tikzcd}[column sep=3em]
      \fun{M}_{\cat{kr}}(\fun{opn}\topo{X})
            \arrow[r]
        & \fun{opn}(\fun{pt}(\fun{M}_{\cat{kr}}(\fun{opn}\topo{X})))
            \arrow[r, "\fun{opn} \phi" above]
        & \fun{opn}(\fun{D}_{\cat{kh}}\topo{X}).
    \end{tikzcd}
    \end{equation*}
  Unravelling the definitions shows that, on generators, it is given by $\Box a \mapsto \dbox a$
  and $\Diamond a \mapsto \ddiamond a$.

\begin{defi}\label{def-eta}
  For every compact Hausdorff space $\topo{X}$ define $\eta_{\topo{X}} :
  \fun{M}_{\cat{kr}}(\fun{opn}\topo{X}) \to \fun{opn}(\fun{D}_{\cat{kh}}
  \topo{X})$ on generators by $\eta_{\topo{X}}(\Box a) = \dbox a$ and
  $\eta_{\topo{X}}(\Diamond a) = \ddiamond a$. By the preceding discussion
  $\eta_{\topo{X}}$ is a well-defined frame isomorphism.
\end{defi}

It turns out that the maps $\eta_{\topo{X}}$ constitute a natural isomorphism.

\begin{lem}\label{lem-mon-dual}
  The collection $\eta = (\eta_{\topo{X}})_{\topo{X} \in \cat{KHaus}}$ forms a
  natural isomorphism.
\end{lem}
\begin{proof}
  It follows from Theorem~\ref{thm:D-M-dual-obj} that each of the
  $\eta_{\topo{X}}$ is an isomorphism, so we only need to show naturality.
  That is, for any morphism $f : \topo{X} \to \topo{X}'$ in $\cat{KHaus}$, the
  following diagram commutes,
  \begin{equation*}
    \begin{tikzcd}[column sep=4em]
      \fun{M}_{\cat{kr}}(\fun{opn}\topo{X})
            \arrow[r, latex-, "\fun{M}_{\cat{kr}}(\fun{opn} f)"]
            \arrow[d, "\eta_{\topo{X}}" left]
        & \fun{M}_{\cat{kr}}(\fun{opn}\topo{X}')
            \arrow[d, "\eta_{\topo{X}'}"] \\
      \fun{opn}(\fun{D}_{\cat{kh}})\topo{X}
            \arrow[r, latex-, "\fun{opn}(\fun{D}_{\cat{kh}}f)" below]
        & \fun{opn}(\fun{D}_{\cat{kh}}\topo{X}')
    \end{tikzcd}
  \end{equation*}
  (Since $\fun{opn}$ is a contravariant functor, the horizontal arrows are
  reversed.) For this, suppose $\Box a'$ is a generator of
  $\fun{M}_{\cat{kr}}(\fun{opn}\topo{X}')$. Then
  \begin{align*}
    \fun{opn}(\fun{D}_{\cat{kh}}f) \circ \eta_{\topo{X}'}(\Box a)
      &= \fun{opn}(\fun{D}_{\cat{kh}}f)(\dbox a)
      &\text{(Definition~\ref{def-eta})} \\
      &= (\fun{D}_{\cat{kh}}f)^{-1}(\dbox a)
      &\text{(Definition of $\fun{opn}$)} \\
      &= \dbox f^{-1}(a)
      &\text{(Lemma~\ref{lem-Df-well})} \\
      &= \eta_{\topo{X}}(\Box f^{-1}(a))
      &\text{(Definition~\ref{def-eta})} \\
      &= \eta_{\topo{X}} \circ \fun{M}_{\cat{kr}}(f^{-1}(\Box a))
      &\text{(Definition of $\fun{M}$)} \\
      &= \eta_{\topo{X}} \circ \fun{M}_{\cat{kr}} (\fun{opn} f)(\Box a)
      &\text{(Definition of $\fun{opn}$)}
  \end{align*}
  and by analogous reasoning $\fun{\Omega}\fun{D}_{\cat{kh}}f \circ
  \eta_{\topo{X}'}(\Diamond a) = \eta_{\topo{X}} \circ \fun{M}_{\cat{kr}}
  (\fun{opn} f)(\Diamond a)$. This proves that the diagram commutes.
\end{proof}

  As an immediate corollary of Lemma~\ref{lem-mon-dual} we obtain:

\begin{thm}\label{thm:mon-full-duality}
  There is a dual equivalence
  \begin{equation*}
    \cat{Alg}(\fun{M}_{\cat{kr}}) \equiv^{\op} \cat{Coalg}(\fun{D}_{\cat{kh}}).
  \end{equation*}
\end{thm}

\subsection{Logic for the monotone neighbourhood functor}\label{subsec:mon-lan}

  We give predicate liftings for $\fun{D}_{\cat{kh}}$ that give rise to
  \emph{monotone modal geometric logic.}
  Define $\lambda^{\Box}, \lambda^{\Diamond} :
  \Omega \to \Omega \circ \fun{D}_{\cat{kh}}$ by
  \[
    \lambda_{\topo{X}}^{\Box}(a) = \{ W \in \fun{D}_{\cat{kh}}\topo{X}
                                      \mid a \in W \}, \qquad
    \lambda_{\topo{X}}^{\Diamond}(a) = \{ W \in \fun{D}_{\cat{kh}}\topo{X}
                                          \mid \topo{X} \setminus a \notin W \}.
  \]
  It is easy to see that these are monotone extendable.

  Write $\Box$ and $\Diamond$ for the corresponding modal operators and let
  $(\topo{X}, \gamma, V)$ be a $\fun{D}_{\cat{kh}}$-model. Then we have
  \[
    x \Vdash \Box\phi
      \iff \gamma(x) \in \lambda^{\Box}_{\topo{X}}(\llb \phi \rrb)
      \iff \llb \phi \rrb \in \gamma(x)
  \]
  and similarly $x \Vdash \Diamond\phi$ if $\topo{X} \setminus \llb \phi \rrb
  \notin \gamma(x)$. This is the same as neighbourhood semantics for monotone
  modal logic over a classical base~\cite{Che80,Han03}.

\begin{rem}\label{rem-exm-khaus-top}
  We will see in Example~\ref{exm-lift-monotone} that the functor
  $\fun{D}_{\cat{kh}}$ on $\cat{KHaus}$ can be generalised to an endofunctor of
  $\cat{Top}$ which restricts to $\cat{Sob}$.
\end{rem}

\section{Axioms, Soundness and completeness}\label{sec:ax}

  We define the notion of \emph{one-step axioms} (similar
  to~\cite[Definition~3.8]{KupKurPat04}) and \emph{one-step rules} for a
  collection of predicate liftings. These give rise to axioms and rules for the
  language $\GML(\Phi, \Lambda)$ from Definition~\ref{def:language}, and to an endofunctor
  $\fun{L}$ on the category of frames.
  As in Section~\ref{sec:logic} we let $\cat{C}$ be some full subcategory
  of $\cat{Top}$, $\fun{T}$ an endofunctor on $\cat{C}$, and we view
  $\fun{opn}$ as a contravariant functor $\cat{C} \to \cat{Frm}$ and
  $\fun{\Omega}$ as a contravariant functor $\cat{C} \to \cat{Set}$.

  At the end of Subsection~\ref{subsec:comp}, in order to derive a general completeness
  result, we restrict our attention to a language \emph{without} proposition
  letters. This need not be problematic: proposition letters can be introduced
  via (nullary) predicate liftings. In particular, this means that the
  category of $\fun{T}$-models is the same as the category of $\fun{T}$-coalgebras.

  Ultimately, using the duality proved in Lemma~\ref{lem-mon-dual}, we derive that
  monotone modal geometric logic without proposition letters is sound
  and complete with respect to $\fun{D}_{\cat{kh}}$-coalgebras.

\subsection{Axioms and algebraic semantics}

  Let $\Lambda$ be a collection of predicate liftings for an endofunctor $\fun{T}$
  on $\cat{C}$. Recall that $\Phi$ denotes a set of propositional variables.
  The set of \emph{zero-step formulas} of $\GML(\Phi, \Lambda)$ is simply the subcollection
  $\GL(\Phi)$ of $\GML(\Phi, \Lambda)$.
  The \emph{one-step formulas} in $\GML(\Phi, \Lambda)$ are given recursively by
  \[
    \phi ::= \top \mid \bot
                  \mid \phi \wedge \phi
                  \mid \bigvee_{i \in I} \phi_i
                  \mid \heartsuit^{\lambda}(\pi_1, \ldots, \pi_n),
  \]
  where $\lambda \in \Lambda$ is $n$-ary and $\pi_1, \ldots, \pi_n \in \GL(\Phi)$.

  We define the notions of a one-step axiom and a one-step rule for $\GML(\Phi, \Lambda)$:

\begin{defi}
  A \emph{one-step axiom} for $\GML(\Phi, \Lambda)$ is a consequence pair
  $\alpha \cp \beta$, where $\alpha, \beta$ are one-step formulas.
  A \emph{one-step rule} is an expression of the form
  \begin{equation}\label{eq:rule}
    \frac{a_i \cp b_i \quad i \in I}{\alpha \cp \beta},
  \end{equation}
  where $I$ is some index-class, $a_i, b_i$ are zero-step formulas for $i \in I$ and
  $\alpha, \beta$ are one-step formulas.
\end{defi}

First let us investigate how one-step axioms and rules give rise to an (equationally
defined) endofunctor $\fun{L}_{(\Lambda, \Ax)}$ on $\cat{Frm}$---when no
confusion is likely we drop the subscript and simply write $\fun{L}$.
Given a frame $F$, the frame $\fun{L}F$ can be presented as follows.
As \emph{generators} we take the collection
\[
\Lambda(F)
= \{ \und{\lambda}(v_1, \ldots, v_n) \mid \lambda \in \Lambda \text{ $n$-ary, }v_i \in F \}.
\]
The idea is now to \emph{instantiate} the (meta)variables of the axiomatisation $\Ax$ 
with the elements of $F$.
Zero-step formulas then naturally evaluate to elements of $F$.
Consequently, an axiom $\alpha \cp \beta$ gives rise to a relation $\alpha \leq \beta$,
and a rule as in~\eqref{eq:rule} yields the relation $\alpha \leq \beta$ \emph{conditionally},
that is, we only consider the relation in those cases where $a_i \leq b_i$ for all $i$.

\begin{defi}\label{def:ax-fun}
  For a frame $F$, define $\fun{L}_{(\Lambda, \Ax)}F = \fun{L}F$ to be the frame
  \[
    \fun{Fr}(\Lambda(F))/R,
  \]
  where $R$ is the collection of relations that arises from
  substituting the metavariables from the schemata in $\Ax$ with elements from $F$.
  For a morphism $f : F \to F'$ define $\fun{L}f$ on generators by
  \[
    \fun{L}f(\und{\lambda}(a_1, \ldots, a_n)) = \und{\lambda}(f(a_1), \ldots, f(a_n)).
  \]
\end{defi}

\begin{defi}\label{def:sound}
  A collection of one-step axioms and one-step rules is called \emph{sound} if the assignment
  $\rho : \fun{L} \circ \fun{opn} \to \fun{opn} \circ \fun{T}$,
  given for $X \in \cat{C}$ by
  \[
    \rho_X : \fun{L} \circ \fun{opn}X
             \to \fun{opn} \circ \fun{T}X
           : \und{\lambda}(a_1, \ldots, a_n) \mapsto \lambda(a_1, \ldots, a_n),
  \]
  defines a natural transformation.
\end{defi}

\begin{exa}\label{exm:mon-ax}
  Suppose $\fun{T} = \fun{D}_{\cat{kh}}$, the monotone neighbourhood functor
  on $\cat{KHaus}$, and $\lambda^{\Box}, \lambda^{\Diamond}$ are given as
  in Subsection~\ref{subsec:mon-lan}.
  The following collection of axioms and rules is sound:
  \begin{multicols}{2}
    \begin{enumerate}[($m_1$)]\itemsep=8pt
      \item $\dfrac{a \cp b}{\und{\lambda}^{\Box}(a) \cp \und{\lambda}^{\Box}(b)}$
      \item\label{eq-ax-M2} $\dfrac{a \wedge b \cp \bot}
                   {\und{\lambda}^{\Box}a \wedge \und{\lambda}^{\Diamond}b \cp
                     \bot}$
      \item\label{eq-ax-M3} $\und{\lambda}^{\Box} \bigveeup A \cp
              \bigvee \{ \und{\lambda}^{\Box}a \mid a \in A \}$
      \item $\dfrac{a \cp b}{\und{\lambda}^{\Diamond}(a) \cp \und{\lambda}^{\Diamond}(b)}$
      \item $\dfrac{\top \cp a \vee b}
                  {\top \cp
                   \und{\lambda}^{\Box}a \vee \und{\lambda}^{\Diamond}b}$
      \item\label{eq-ax-M6} $\und{\lambda}^{\Diamond} \bigveeup A \cp
              \bigvee \{ \und{\lambda}^{\Diamond}a \mid a \in A \}$
    \end{enumerate}
  \end{multicols}
  \noindent
  where $A$ is a directed set of $\GL(\Phi)$-formulas (cf.~Subsection~\ref{subsec:gl}).
  To be somewhat more precise, we can view both~\ref{eq-ax-M3} and~\ref{eq-ax-M6}
  to be the consequence of a rule, the premises of which are given by a set of
  consequence pairs witnessing the directedness of $A$.

  To see, for example, that~\ref{eq-ax-M3} is valid in a
  $\fun{D}_{\cat{kh}}$-coalgebra $(\topo{X}, \gamma)$, we need to show that
  \begin{equation}\label{eq:D3-sound}
    \dbox \bigcupup A \subseteq \bigcupup \{ \dbox a \mid a \in A \}
  \end{equation}
  in $\fun{D}_{\cat{kh}}\topo{X}$, where $A$ is a directed set of open
  subsets of $\topo{X}$. So suppose $W \in \dbox \bigcupup A$, then
  $\bigcupup A \in W$. By definition there must be a closed
  $c \subseteq \bigcupup A$ such that $c \in W$. Then $\bigcupup A$ is an
  open cover of $c$ and since $c$ is closed, hence compact, there must be a
  finite subcover. But then there must be a single $a \in A$ such that
  $c \subseteq a$, because $A$ is directed, and as $W$ is up-closed under
  inclusions we have $a \in W$. This implies $W \in \dbox a$, i.e.~$W$ is
  in the right hand side of~\eqref{eq:D3-sound}.

  The functor $\fun{M}$ from Definition~\ref{def-functor-M} is obtained
  from the procedure of Definition~\ref{def:ax-fun}.
\end{exa}

\begin{exa}\label{exm:viet-axioms}
  In a similar manner, one can find a collection of sound axioms and rules
  for the Vietoris functor on $\cat{KHaus}$ such that the procedure from
  Definition~\ref{def:ax-fun} yields the Vietoris locale from~\cite[Section~1]{Joh85}.
\end{exa}

  For the remainder of this section we work in the following setting:

\begin{asm}\label{asm:ax}
  Throughout the remainder of this section, we assume given a collection
  $\Lambda$ of predicate liftings for an endofunctor $\fun{T}$ on $\cat{C}$,
  and a set $\Ax$ of axioms and rules for $\GML(\Phi, \Lambda)$ which is sound
  in the sense of Definition~\ref{def:sound}. We write $\fun{L}$ for the
  endofunctor on $\cat{Frm}$ given by the procedure in Definition~\ref{def:ax-fun}
  and $\rho$ is the associated natural transformation.
\end{asm}

  Given a language $\GML(\Phi, \Lambda)$ and a collection of axiom and rule schemata,
  we write $\ms{GML}(\Phi, \Lambda, \Ax)$ for the (minimal) collection of
  consequence pairs $\phi \cp \psi$ of formulas in $\GML(\Phi, \Lambda)$ which
  contains all axioms and is closed under the rules from geometric logic
  (cf.~Subsection~\ref{subsec:gl}) and $\Ax$, and under the congruence rule
  \[
    \dfrac{\phi_1 \leftrightarrow \psi_1 \quad \cdots \quad \phi_n \leftrightarrow \psi_n}
          {\heartsuit^{\lambda}(\phi_1, \ldots, \phi_n) \leftrightarrow
           \heartsuit^{\lambda}(\psi_1, \ldots, \psi_n)}.
  \]
  We write $\phi \vdash_{\ms{GML}(\Phi, \Lambda, \Ax)} \psi$ if the consequence pair
  $\phi \cp \psi$ is in $\ms{GML}(\Phi, \Lambda, \Ax)$. If no confusion is likely,
  we omit the subscript from the turnstyle and simply write $\phi \vdash \psi$.

  Suppose given $\Phi, \Lambda$ and $\Ax$, write $L$ for the language
  $\GML(\Phi, \Lambda)$ modulo the equivalence relation
  $\dashv\vdash$ given by
  $\phi \dashv\vdash \psi$ iff $\phi \vdash \psi$ and $\psi \vdash \phi$.
  Write $[\phi]$ for the equivalence class in $L$ of a formula $\phi \in \GML(\Lambda)$.
  Then $L$ carries a frame structure by setting
  $[\phi] \vee [\psi] = [\phi \vee \psi]$ and similar for the other connectives.
  Furthermore, we can define an $\fun{L}$-algebra structure $\ell : \fun{L}L \to L$
  via
  \[
    \und{\lambda}([\phi_1], \ldots, [\phi_n])
      \mapsto [\heartsuit^{\lambda}(\phi_1, \ldots, \phi_n)].
  \]
  Recall that $\cat{Alg}(\fun{L})$ denotes the category of $\fun{L}$-algebras
  and $\fun{L}$-algebra morphisms (see Subsection~\ref{subsec:coalg}). We have:

\begin{lem}\label{lem:init}
  The $\fun{L}$-algebra $\mc{L} = (L, \ell)$ is initial in $\cat{Alg}(\fun{L})$.
\end{lem}

  Moreover:

\begin{lem}\label{lem:vdash-init}
  For any two formulas $\phi, \psi$ we have $\phi \vdash \psi$ iff $[\phi] \leq [\psi]$
  in $\mc{L}$.
\end{lem}

  The initial $\fun{L}$-algebra $\mc{L}$ gives rise to an interpretation of
  $\GML(\Phi, \Lambda)$ in every $\fun{L}$-algebra:
  The interpretation of a formula $\phi$ in $\mc{A} = (A, \alpha) \in \cat{Alg}(\fun{L})$
  is given by $\llp \phi \rrp_{\mc{A}} = i_{\mc{A}}([\phi])$, where
  $i_{\mc{A}}$ is the unique $\fun{L}$-algebra morphism $\mc{L} \to \mc{A}$.
  The interpretation is related to the semantics via the \emph{complex algebra}:

\begin{defi}
  The \emph{complex $\fun{L}$-algebra} of a $\fun{T}$-coalgebra $\mf{X} = (X, \gamma)$ is
  $\mf{X}^+ = (\fun{opn}X, \fun{opn}\gamma \circ \rho_X)$,
  where $\rho$ is the natural transformation from Definition~\ref{def:sound}.
\end{defi}

  We can view the interpretation of a formula $\phi$ in a $\fun{T}$-coalgebra as
  an element of its complex algebra. Examination of the definitions shows that
  \begin{equation}\label{eq:int-alg-coalg}
    \llb \phi \rrb^{\mf{X}} = \llp \phi \rrp_{\mf{X}^+}.
  \end{equation}
  Furthermore, we have $[\phi] \leq [\psi]$ if and only if
  $\llp \phi \rrp_{\mc{A}} \leq \llp \psi \rrp_{\mc{A}}$ for all $\fun{L}$-algebras
  $\mc{A}$.
  Soundness of the logic now follows from soundness of the axioms: Suppose
  $\phi \vdash \psi$, then $[\phi] \leq [\psi]$ in $\mc{L}$ and hence
  $\llp \phi \rrp_{\mc{A}} \leq \llp \psi \rrp_{\mc{A}}$ in any $\fun{L}$-algebra
  $\mc{A}$. Therefore, by the observation from~\eqref{eq:int-alg-coalg} we have
  $\llb \phi \rrb^{\mf{X}} \subseteq \llb \psi \rrb^{\mf{X}}$ for every
  $\mf{X} \in \cat{Mod}(\fun{T})$ hence $\phi \Vdash_{\fun{T}} \psi$.

\subsection{Completeness}\label{subsec:comp}
  We keep working within the assumptions of~\ref{asm:ax}.
  In order to prove completeness with respect to $\cat{Coalg}(\fun{T})$, we want
  to show that
  \[
    \phi \Vdash_{\fun{T}} \psi \quad\text{ implies }\quad \phi \vdash \psi.
  \]
  That is, if $\phi \Vdash_{\fun{T}} \psi$ then
  $(\phi, \psi) \in \ms{GML}(\Lambda, \Ax)$.
  By Lemma~\ref{lem:vdash-init} it suffices to show that $[\phi] \leq [\psi]$
  in the initial $\fun{L}$-algebra $\mc{L} = (L, \ell)$ whenever
  $\phi \Vdash_{\fun{T}} \psi$.

  If $\phi \Vdash_{\fun{T}} \psi$, then we know that $\llp \phi \rrp_{\mf{X}^+}
  \leq \llp \psi \rrp_{\mf{X}^+}$ in every complex algebra $\mf{X}^+$ for
  $\mf{X} \in \cat{Mod}(\fun{T})$. However, there is no guarantee that $\mc{L}$
  should be the complex algebra of some $\fun{T}$-model, so we do not automatically
  get completeness.

  The next proposition shows that in order to prove completeness it
  suffices to find a $\fun{T}$-coalgebra $\mf{X}$ and an $\fun{L}$-algebra morphism
  $h : \mf{X}^+ \to \mc{L}$.

\begin{prop}\label{prop:comp4}
  If there exists a $\fun{T}$-coalgebra $\mf{X}$ and an
  $\fun{L}$-algebra morphism $h : \mf{X}^+ \to \mc{L}$, then
  \[
    \phi \Vdash_{\fun{T}} \psi \quad\text{ implies }\quad \phi \vdash \psi.
  \]
\end{prop}
\begin{proof}
  Write $i$ for the unique $\fun{L}$-algebra morphism $\mc{L} \to \mf{X}^+$.
  Then initiality of $\mc{L}$ forces $h \circ i = \id_{\mc{L}}$.
  Suppose $\phi \Vdash_{\fun{T}} \psi$, then we have
  $\llb \phi \rrb^{\mf{X}} \subseteq \llb \psi \rrb^{\mf{X}}$, which in turn
  implies $\llp \phi \rrp_{\mf{X}^+} \leq \llp \psi \rrp_{\mf{X}^+}$
  by~\eqref{eq:int-alg-coalg}.
  It follows from monotonicity of $h$ that
  \[
    [\phi] = \llp \phi \rrp_{\mc{L}}
           = h \circ i(\llp \phi \rrp_{\mc{L}})
           = h(\llp \phi \rrp_{\mf{X}^+})
           \leq h(\llp \psi \rrp_{\mf{X}^+})
           = h \circ i(\llp \phi \rrp_{\mc{L}})
           = \llp \phi \rrp_{\mc{L}}
           = [\psi].
  \]
  This proves the proposition.
\end{proof}

  Ideally, one would use a duality of functors to establish that such an $\mf{X}$
  as in Proposition~\ref{prop:comp4} exists. For example,
  this is how one can prove completeness for (classical) normal modal logic:
  The Vietoris functor on $\cat{Stone}$ is the Stone dual of the endofunctor
  on $\cat{BA}$ which determines the logic. However, since we do not
  start with a dual equivalence (like Stone duality) but rather with a dual
  \emph{adjunction}, endofunctors on both categories cannot be dual.

  To remedy this, we will make use of the fact that the dual adjunction between
  $\cat{Top}$ and $\cat{Frm}$ restricts to several dual equivalences
  (see Fact~\ref{fact-pt-opn}). Note that this does not yet guarantee that the initial
  $\fun{L}$ algebra from Lemma~\ref{lem:init} is the complex algebra of some
  $\fun{T}$-coalgebra! Indeed, its underlying frame need not be in the
  restricted dual equivalence. We will see that some of the dual
  equivalences are good enough to ``imitate'' frames that fall outside it.

  We now investigate under which conditions we can use Proposition~\ref{prop:comp4}.
  As announced, we shall restrict our attention to the case where
  $\Phi = \emptyset$, i.e.~we work in a language without proposition letters.
  This means that there is no need for having valuations, hence
  $\cat{Mod}(\fun{T})$ is simply (isomorphic to) $\cat{Coalg}(\fun{T})$.
  If $\Phi = \emptyset$, we shall write $\GML(\Lambda)$ instead of
  $\GML(\Phi, \Lambda)$ and $\ms{GML}(\Lambda, \Ax)$ for
  $\ms{GML}(\Phi, \Lambda, \Ax)$.
  The absence of proposition letters need not pose a big deficiency:
  proposition letters can simply be introduced via predicate liftings.

  Specifically, we look for a subcategory $\cat{A}$ of
  $\cat{Alg}(\fun{L})$ such that:
  \begin{enumerate}
    \item\label{eq:strat-co} Every $\fun{L}$-algebra is the codomain of
          an $\fun{L}$-algebra morphism
          whose domain is in $\cat{A}$;
    \item\label{eq:strat-com} Every algebra in $\cat{A}$ is the complex
          algebra corresponding to
          some $\fun{T}$-coalgebra.
  \end{enumerate}
  Clearly, if this is the case we can employ Proposition~\ref{prop:comp4}.

  For the first item, it turns out useful to consider \emph{coreflective
  subcategories} of $\cat{Frm}$.
  These are full subcategories $\cat{F}$ of $\cat{Frm}$ such that the inclusion
  functor $\cat{F} \to \cat{Frm}$ has a right adjoint.
  We recall an alternative definition from~\cite{AdaHerStr90}
  (which is in turn equivalent to the one in~\cite[Section~IV.3]{Lan71}).

\begin{defiC}[{\cite[Definition~4.25]{AdaHerStr90}}]
  A subcategory $\cat{B}$ of a category $\cat{A}$ is called \emph{coreflective}
  if for every $A \in \cat{A}$ there exists a $B \in \cat{B}$ and a morphism
  $c : B \to A$  (called a coreflection) such that for every morphism $f : B' \to A$
  in $\cat{A}$ with $B' \in \cat{B}$, there exists a unique $f' : B' \to B$ such that
  \[
  \begin{tikzcd}
    B'    \arrow[d, "f'" left]
          \arrow[dr, "f"]
      & \\
    B     \arrow[r, "c"]
      & A
  \end{tikzcd}
  \]
  commutes.
\end{defiC}

  We can now formulate simple conditions that guarantee item~\eqref{eq:strat-co}
  to hold.

\begin{lem}
  Let $\cat{F}$ be a coreflective subcategory of $\cat{Frm}$ and suppose $\fun{L}$
  restricts to an endofunctor $\fun{L'}$ on $\cat{F}$. Then for each
  $\fun{L}$-algebra $(A, \alpha)$ there exists an $\fun{L}$-algebra $(B, \beta)$
  with $B \in \cat{F}$ and an $\fun{L}$-algebra morphism $(B, \beta) \to (A, \alpha)$.
\end{lem}
\begin{proof}
  Let $c : B \to A$ be the coreflection of $A$ in $\cat{F}$. Then we have a diagram
  \[
    \begin{tikzcd}
      \fun{L}B
            \arrow[r, "\fun{L}c"]
            \arrow[d, dashed, "\beta"]
        & \fun{L}A
            \arrow[d, "\alpha"] \\
      B \arrow[r, "c"]
        & A
    \end{tikzcd}
  \]
  where $\fun{L}B$ is in $\cat{F}$ by assumption. By definition there exists
  $\beta : \fun{L}B \to B$ making the diagram commute.
\end{proof}

  Thus, if $\cat{F}$ is a coreflective subcategory of $\cat{Frm}$ and $\fun{L}$
  restricts to $\cat{F}$, then item~\eqref{eq:strat-co} above is satisfied.
  Examples of such coreflective subcategories are
  $\cat{RFrm}$~\cite[Section~4.2]{PicPulToz04} and $\cat{KRFrm}$~\cite[Proposition~3]{BanMul80}.

  Now suppose that all objects in $\cat{F}$ are spatial and write $\cat{T}$ for the subcategory
  of $\cat{Top}$ which is dually equivalent to $\cat{F}$. Furthermore,
  assume that $\cat{T}$ is a subcategory of $\cat{C}$ (for otherwise $\fun{T}$
  is not defined for every space in $\cat{T}$).

  If the restriction $\fun{L'}$ of $\fun{L}$ is dual to the restriction
  $\fun{T'}$ of $\fun{T}$ to $\cat{T}$, then we know that
  $\cat{Coalg}(\fun{T'}) \equiv^{\op} \cat{Alg}(\fun{L'})$. In particular,
  this means that every element of $\cat{Alg}(\fun{L'})$ is the complex
  algebra of some $\fun{T'}$-coalgebra (hence of some $\fun{T}$-coalgebra),
  i.e.~item~\eqref{eq:strat-com} is satisfied.
  Summarising:

\begin{thm}\label{thm:cor-com}
  Suppose there exists a coreflective subcategory $\cat{F}$ of $\cat{Frm}$
  such that
  \begin{itemize}
    \item The dual $\cat{T}$ of $\cat{F}$ is a subcategory of $\cat{C}$;
    \item $\fun{L}$ restricts to an endofunctor $\fun{L'}$ on $\cat{F}$;
    \item $\fun{T}$ restricts to an endofunctor $\fun{T'}$ on $\cat{F}^{\op}$;
    \item $\fun{L'}$ is dual to $\fun{T'}$;
  \end{itemize}
  Then $\ms{GML}(\Lambda, \Ax)$
  is complete with respect to $\cat{Coalg}(\fun{T})$,
  in the sense that for for every consequence pair $\phi \cp \psi$ of closed
  $\GML$-formulas we have that
  \[
    \phi \Vdash_{\fun{T}} \psi
      \quad\text{implies}\quad
      \phi \cp \psi \in \ms{GML}(\Lambda, \Ax).
  \]
\end{thm}

  Let us apply this theorem to normal and monotone modal geometric logic.

\begin{exa}\label{exm:mon-com}
  We denote by $\ms{M}$ the smallest collection of consequence pairs closed
  under the axioms and rules from geometric logic (see Subsection~\ref{subsec:gl})
  and the ones presented in Example~\ref{exm:mon-ax}. (Note that the congruence rules follow from
  the monotonicity rules.) It follows from the duality between
  $\fun{D}_{\cat{kh}}$ and $\fun{M}_{\cat{kr}}$, the fact that $\cat{KRFrm}$ is
  a coreflective subcategory of $\cat{Frm}$, and Theorem~\ref{thm:cor-com}
  that $\ms{M}$ is (sound and) complete with respect to
  $\cat{Coalg}(\fun{D}_{\cat{kh}})$.
\end{exa}

\begin{exa}
  Similar to Example~\ref{exm:mon-com}, one can prove that \emph{normal}
  geometric modal logic $\ms{N}$ is sound and complete with respect to
  $\cat{Coalg}(\fun{V}_{\cat{kh}})$. In this case, the axioms and rules of
  $\ms{N}$ are the ones from geometric logic, those introduced for
  positive modal logic in~\cite[Section~2]{Dun95}, and Scott-continuity.
  We leave the details to the reader.
\end{exa}

\section{Lifting logics from \texorpdfstring{$\cat{Set}$}{Set} to \texorpdfstring{$\cat{Top}$}{Top}}\label{sec:lift}

  In~\cite[Section~4]{KupKurPat05} the authors give a method to lift a
  $\cat{Set}$-functor $\fun{T} : \cat{Set} \to \cat{Set}$, together with a
  collection of predicate liftings $\Lambda$ for $\fun{T}$, to an endofunctor on
  $\cat{Stone}$. We adapt their approach to obtain an endofunctor
  $\hat{\fun{T}}_{\Lambda}$ on $\cat{Top}$, and a collection of Scott-continuous
  open predicate liftings $\hat{\Lambda}$ for $\hat{\fun{T}}_{\Lambda}$.
  In this section the notation $\bigveeup$ is used for directed joins,
  i.e.~joins over directed sets.

\subsection{Lifting functors from \texorpdfstring{$\cat{Set}$}{Set} to \texorpdfstring{$\cat{Top}$}{Top}}
  Let $\fun{T}$ be an endofunctor on $\cat{Set}$ and $\Lambda$ a collection of
  predicate liftings for $\fun{T}$. To define the action of
  $\hat{\fun{T}}_{\Lambda}$ on a topological space
  $\topo{X}$ we take the following steps:
  \begin{enumerate}[\sl Step 1.]
    \item Construct a frame $\dot{\fun{F}}_{\Lambda}\topo{X}$ of the images of
          predicate liftings applied to the open sets of $\topo{X}$ (viewed
          simply as subsets of $\fun{T}(\fun{U}\topo{X})$);
    \item Quotient $\dot{\fun{F}}_{\Lambda}\topo{X}$ with a suitable relation
          that ensures $\bigveeup_{b \in B} \lambda(b) = \lambda(\bigveeup B)$
          whenever $\lambda$ is monotone;
    \item Employ the functor $\fun{pt} : \cat{Frm} \to \cat{Top}$ to obtain a
          (sober) topological space.
  \end{enumerate}
  This is the content of Definitions~\ref{def-Fdot}, \ref{def-Fddot}
  and~\ref{def-KKP-sob}.
  Recall that $\fun{U} : \cat{Top} \to \cat{Set}$ is the forgetful functor and
  that $\fun{Q}$ is the contravariant functor sending a set to its Boolean
  powerset algebra.

\begin{defi}\label{def-Fdot}
  Let $\fun{T} : \cat{Set} \to \cat{Set}$ be a functor and $\Lambda$ a
  collection of predicate liftings for $\fun{T}$. We define a contravariant
  functor $\dot{\fun{F}}_{\Lambda} : \cat{Top} \to \cat{Frm}$.
  For a topological space $\topo{X}$ let $\dot{\fun{F}}_{\Lambda}\topo{X}$ be
  the subframe of $\fun{Q}(\fun{T}(\fun{U}\topo{X}))$ generated by the set
  \begin{equation*}
    \{ \lambda_{\fun{U}\topo{X}}(a_1, \ldots, a_n) \mid \lambda \in \Lambda
    \text{ $n$-ary, } a_1, \ldots, a_n \in \fun{\Omega}\topo{X} \}.
  \end{equation*}
  That is, we close this set under finite intersections and arbitrary unions in
  $\fun{Q}(\fun{T}(\fun{U}\topo{X}))$.
  For a continuous map $f : \topo{X} \to \topo{X}'$ let
  $\dot{\fun{F}}_{\Lambda}f : \dot{\fun{F}}_{\Lambda}\topo{X}'
  \to \dot{\fun{F}}_{\Lambda}\topo{X}$ be the restriction of
  $\fun{Q}(\fun{T}(\fun{U}f))$ to $\dot{\fun{F}}_{\Lambda}\topo{X}'$.
\end{defi}

\begin{lem}
  The assignment $\dot{\fun{F}}_{\Lambda}$ defines a contravariant functor.
\end{lem}
\begin{proof}
  We need to show that $\dot{\fun{F}}_{\Lambda}$ is well defined on morphisms
  and that it is functorial. To show that the action of
  $\dot{\fun{F}}_{\Lambda}$ on morphisms is well-defined, it suffices to show
  that $(\dot{\fun{F}}_{\Lambda}f)(\lambda_{\fun{U}\topo{X}}(a_1', \ldots,
  a_n')) \in \dot{\fun{F}}_{\Lambda}(\topo{X})$ for all generators
  $\lambda_{\fun{U}\topo{X}}(a_1', \ldots, a_n')$ of $\dot{\fun{F}}_{\Lambda}
  \topo{X}'$, because inverse images preserve finite meets and all joins.
  This holds by naturality of $\lambda$:
  \begin{equation*}
    (\dot{\fun{F}}_{\Lambda}f)(\lambda_{\fun{U}\topo{X}'}(a_1, \ldots, a_n))
      = (\fun{T}f)^{-1}(\lambda_{\fun{U}\topo{X}'}(a_1, \ldots, a_n))
      = \lambda_{\fun{U}\topo{X}}(f^{-1}(a_1), \ldots, f^{-1}(a_n)).
  \end{equation*}
  By continuity of $f$ we have $f^{-1}(a_i) \in \fun{\Omega}\topo{X}$ so the
  latter is indeed in $\dot{\fun{F}}_{\Lambda}\topo{X}$.
  Functoriality of $\dot{\fun{F}}_{\Lambda}$ follows from functoriality of
  $\fun{Q} \circ \fun{T} \circ \fun{U}$.
\end{proof}

\begin{defi}\label{def-Fddot}
  Let $\Lambda$ be a collection of predicate liftings for a $\cat{Set}$-functor
  $\fun{T}$.
  For $\topo{X} \in \cat{Top}$, let $\hat{\fun{F}}_{\Lambda}\topo{X}$ be the
  quotient of $\dot{\fun{F}}_{\Lambda}\topo{X}$ with respect to the congruence
  $\sim$ generated by
  \begin{equation*}
    \bigveeup_{b \in B} \lambda(a_1, \ldots, a_{i-1}, b, a_{i+1}, \ldots, a_n)
    \sim \lambda(a_1, \ldots, a_{i-1}, \bigveeup B, a_{i+1}, \ldots, a_n)
  \end{equation*}
  for all $a_i \in \fun{\Omega}\topo{X}$, $B \subseteq \fun{\Omega}\topo{X}$
  directed, and $\lambda \in \Lambda$ monotone in its $i$-th argument. Write
  $q_{\topo{X}} : \dot{\fun{F}}_{\Lambda}\topo{X} \to \hat{\fun{F}}_{\Lambda}
  \topo{X}$ for the quotient map and $[x]$ for the equivalence class in
  $\hat{\fun{F}}_{\Lambda}\topo{X}$ of an element
  $x \in \dot{\fun{F}}_{\Lambda}\topo{X}$.
  For a continuous function $f : \topo{X} \to \topo{X}'$ define
  $\hat{\fun{F}}_{\Lambda}f
    : \hat{\fun{F}}_{\Lambda}\topo{X}' \to \hat{\fun{F}}_{\Lambda}\topo{X}
    : [\lambda_{\fun{U}\topo{X}}(a_1, \ldots, a_n)]
      \mapsto [\dot{\fun{F}}_{\Lambda}(\lambda_{\fun{U}\topo{X}}
      (a_1, \ldots, a_n))]$.
\end{defi}

  Quotienting by the congruence from Definition~\ref{def-Fddot} ensures that the
  lifted versions of monotone predicate liftings are Scott-continuous, see
  Proposition~\ref{prop-lifted-pl} below.
  This is useful when constructing a final model, because Scott-continuity
  ensures that the collection of formulas modulo so-called semantic
  equivalence is set-sized (see Lemma~\ref{lem-norm} below), and consequently
  the final model aids in comparing several equivalence notions in
  Section~\ref{sec-bisim}.

\begin{lem}\label{lem-ddF}
  The assignment $\hat{\fun{F}}_{\Lambda}$ defines a contravariant functor.
\end{lem}
\begin{proof}
  We need to prove functoriality of $\hat{\fun{F}}_{\Lambda}$ and that
  $\hat{\fun{F}}_{\Lambda}f$ is well defined for every continuous map
  $f : \topo{X} \to \topo{X}'$.
  In order to show that $\hat{\fun{F}}_{\Lambda}$ is well defined, it suffices
  to show that $\dot{\fun{F}}_{\Lambda}f$ is invariant under the congruence
  $\sim$. If $f : \topo{X} \to \topo{X}'$ is a continuous, then
  \begin{align*}
    \bigveeup_{b' \in B} (\dot{\fun{F}}_{\Lambda}f)
      &(\lambda_{\fun{U}\topo{X}'}(a'_1, \ldots, a'_{i-1}, b', a'_{i+1},
        \ldots, a'_n)) \\
      &= \bigveeup_{b' \in B} (\fun{T}f)^{-1}(\lambda_{\fun{U}\topo{X}'}
        (a'_1, \ldots, a'_{i-1}, b', a'_{i+1}, \ldots, a'_n)) \\
      &= \bigveeup_{b' \in B} \lambda_{\fun{U}\topo{X}}(f^{-1}(a'_1), \ldots,
        f^{-1}(a'_{i-1}), f^{-1}(b'), f^{-1}(a'_{i+1}), \ldots, f^{-1}(a'_n)) \\
      &\sim \lambda_{\fun{U}\topo{X}}(f^{-1}(a'_1), \ldots, f^{-1}(a'_{i-1}),
        f^{-1}(\bigveeup B), f^{-1}(a'_{i+1}), \ldots, f^{-1}(a'_n)) \\
      &= \dot{\fun{F}}_{\Lambda}f(\lambda_{\fun{U}\topo{X}}(a'_1, \ldots,
        a'_{i-1}, \bigveeup B, a'_{i+1}, \ldots, a'_n))
  \end{align*}
  so $\dot{\fun{F}}_{\Lambda}f$ is invariant under the congruence.
  In the $\sim$-step we use the fact that $\{ f^{-1}(b') \mid b' \in B \}$ is
  directed in $\fun{\Omega}\topo{X}$. Functoriality of
  $\hat{\fun{F}}_{\Lambda}f$ follows from functoriality of
  $\dot{\fun{F}}_{\Lambda}$.
\end{proof}

  We are now ready to define the topological Kupke-Kurz-Pattinson lift of a
  functor on $\cat{Set}$ together with a collection of predicate liftings, to a
  functor on $\cat{Top}$.

\begin{defi}\label{def-KKP-sob}
  Define the \emph{topological Kupke-Kurz-Pattinson lift} (KKP lift for short)
  of $\fun{T}$ with respect to $\Lambda$ to be the functor
  \begin{equation*}
    \hat{\fun{T}}_{\Lambda} = \fun{pt} \,\circ \, {\hat{\fun{F}}_{\Lambda}}.
  \end{equation*}
  This is a functor $\cat{Top} \to \cat{Top}$ and since $\fun{pt}$ lands in
  $\cat{Sob}$ it restricts to an endofunctor on $\cat{Sob}$.
\end{defi}

  Let us put our theory to action. As stated in Section~\ref{sec:mon} we can
  generalise the monotone functor $\fun{D}_{\cat{kh}}$ on $\cat{KHaus}$ from
  Definition~\ref{def-monotone-functor} to an endofunctor on $\cat{Top}$.
  We will show that lifting the monotone $\cat{Set}$-functor $\fun{D}$ with respect to
  the predicate liftings for box and diamond from Example~\ref{exm-mon-set-pred}
  gives rise to a functor on $\cat{Top}$ which restricts to $\fun{D}_{\cat{kh}}$.

\begin{exa}[The monotone functor]\label{exm-lift-monotone}
  Recall the $\cat{Set}$-functor $\fun{D}$ from Example~\ref{exm-mon-set}:
  $\fun{D} : X \to \{ W \in \breve{\fun{P}}\breve{\fun{P}}X \mid W
  \text{ is up-closed under inclusion order} \}$. The box and diamond are given
  by the predicate liftings $\lambda^{\Box}, \lambda^{\Diamond} :
  \breve{\fun{P}} \to \breve{\fun{P}} \circ \fun{D}$ defined by
  \begin{equation*}
    \lambda^{\Box}_X(a) := \{ W \in \fun{D}X \mid a \in W \}, \qquad
    \lambda^{\Diamond}_X(a) := \{ W \in \fun{D}X \mid (X \setminus a) \notin W
    \},
  \end{equation*}
  where $X \in \cat{Set}$. Furthermore recall from
  Definition~\ref{def-monotone-functor} that for a compact Hausdorff space $\topo{X}$ the
  space $\fun{D}_{\cat{kh}}\topo{X}$ is the subset of $\fun{D}(\fun{U}\topo{X})$
  of collections of sets $W$ satisfying for all $u \subseteq \fun{U}\topo{X}$
  that $u \in W$ iff there exists a closed $c \subseteq u$ such that every open
  superset of $c$ is in $W$.
  In particular this means $\fun{U}(\fun{D}_{\cat{kh}}\topo{X}) \subseteq
  \fun{D}(\fun{U}\topo{X})$.
  The set $\fun{D}_{\cat{kh}}\topo{X}$ is topologised by the subbase
  \begin{equation*}
    \dbox a := \{ W \in \fun{D}_{\cat{kh}}\topo{X} \mid a \in W \}, \quad
    \ddiamond a := \{ W \in \fun{D}_{\cat{kh}}\topo{X}
      \mid (\topo{X} \setminus a) \notin W \}.
  \end{equation*}
  By Theorem~\ref{thm:D-M-dual-obj} the functor $\fun{M} : \cat{Frm} \to
  \cat{Frm}$ from Definition~\ref{def-functor-M} is such that
  $\fun{M}(\fun{opn} \topo{X}) \cong \fun{opn}(\fun{D}_{\cat{kh}}\topo{X})$
  whenever $\topo{X}$ is a compact Hausdorff space.

  We claim that
  \begin{equation}\label{eq-exm-lift-monotone-5}
    \fun{D}_{\cat{kh}} = (\hat{\fun{D}}_{\{\lambda^{\Box},
    \lambda^{\Diamond} \}})_{\upharpoonright\cat{KHaus}}
  \end{equation}
  and to prove this we will show that $\hat{\fun{F}}_{\{\lambda^{\Box},
  \lambda^{\Diamond} \}}\topo{X} = \fun{opn}(\fun{D}_{\cat{kh}}\topo{X})$ for
  every compact Hausdorff space $\topo{X}$.
  Define a map $\phi : \fun{M}(\fun{opn}\topo{X}) \to
  \hat{\fun{F}}_{\{\lambda^{\Box}, \lambda^{\Diamond} \}}\topo{X}$ on generators
  by $\Box a \mapsto [\lambda^{\Box}(a)]$ and $\Diamond a \mapsto
  [\lambda^{\Diamond}(a)]$. This is well defined because the
  $[\lambda^{\Box}(a)], [\lambda^{\Diamond}(a)]$ satisfy
  relations~\ref{eq-fun-M-1} to~\ref{eq-fun-M-6} from Definition~\ref{def-functor-M}, and
  it is surjective because the image of $\phi$ contains the generators of
  $\hat{\fun{F}}_{\{\lambda^{\Box}, \lambda^{\Diamond} \}}\topo{X}$.

  So we only need to show injectivity of $\phi$. Our strategy to prove this is
  to define a map $\psi : \hat{\fun{F}}_{\{\lambda^{\Box}, \lambda^{\Diamond}
  \}}\topo{X} \to \fun{opn}(\fun{D}_{\cat{kh}}\topo{X})$ and show that it is inverse
  to $\phi$ on the level of sets. Since a set-theoretic inverse suffices we do
  not need to prove that $\psi$ is a homomorphism; we just want it to be well
  defined. Instead of defining $\psi : \hat{\fun{F}}_{\{\lambda^{\Box},
  \lambda^{\Diamond}\}}\topo{X} \to \fun{opn}(\fun{D}_{\cat{kh}}\topo{X})$
  directly, we will give a well-defined map
  $\psi' : \dot{\fun{F}}_{\{\lambda^{\Box},
  \lambda^{\Diamond}\}}\topo{X} \to \fun{opn}(\fun{D}_{\cat{kh}}\topo{X})$ whose
  kernel contains the kernel of the quotient map $q_{\topo{X}} :
  \dot{\fun{F}}_{\{\lambda^{\Box}, \lambda^{\Diamond}\}}\topo{X} \to
  \hat{\fun{F}}_{\{\lambda^{\Box}, \lambda^{\Diamond}\}}\topo{X}$.
  This in turn yields the map $\psi$ we require. In a diagram:
  \begin{equation}\label{ch4-eq-lift-D-diagram}
    \begin{tikzcd}[column sep=1em]
      \dot{\fun{F}}_{\{\lambda^{\Box}, \lambda^{\Diamond} \}}\topo{X}
            \arrow[rr, "\psi'" above]
            \arrow[rd, "q_{\topo{X}}" swap]
        &
        & \fun{opn}(\fun{D}_{\cat{kh}}\topo{X}) \\
        & \hat{\fun{F}}_{\{\lambda^{\Box}, \lambda^{\Diamond} \}}\topo{X}
            \arrow[ru, dashed, "\psi" swap]
        &
    \end{tikzcd}
  \end{equation}

  Define $\psi' : \dot{\fun{F}}_{\{\lambda^{\Box}, \lambda^{\Diamond} \}}\topo{X}
  \to \fun{M}(\fun{opn}\topo{X})$ on generators by $\lambda^{\Box}(a) \mapsto
  \dbox a$ and $\lambda^{\Diamond}(a) \mapsto \ddiamond a$. In order to show that
  this assignments yields a well-defined map (hence extends to a frame
  homomorphism by Remark~\ref{rem-gen-frm-hom}) we need to show that the
  presentation of an element in $\dot{\fun{F}}_{\{\lambda^{\Box},
  \lambda^{\Diamond} \}}\topo{X}$  does not affect its image under $\psi'$.
  That is, if
  \begin{equation}\label{ch4-eq-lift-D}
    \bigcup_{i \in I}\Big(\bigcap_{j \in J_i}\lambda^{\Box}(a_{i,j}) \cap
    \bigcap_{j' \in J_i'} \lambda^{\Diamond}(a_{i,j'})\Big)
      = \bigcup_{k \in K}\Big(\bigcap_{\ell \in L_k}\lambda^{\Box}(a_{k,\ell})
      \cap \bigcap_{\ell' \in L_k'} \lambda^{\Diamond}(a_{k, \ell'})\Big),
  \end{equation}
  where $J_i, J_i', L_k$ and $L_k'$ are finite index sets,
  then
  \begin{equation*}
    \bigcup_{i \in I}\Big(\bigcap_{j \in J}\psi'(\lambda^{\Box}(a_{i,j}))
      \cap \bigcap_{j' \in J'} \psi'(\lambda^{\Diamond}(a_{i,j'}))\Big)
      = \bigcup_{k \in K}\Big(\bigcap_{\ell \in L}
        \psi'(\lambda^{\Box}(a_{k,\ell}))
      \cap \bigcap_{\ell' \in L'} \psi'(\lambda^{\Diamond}(a_{k, \ell'}))\Big).
  \end{equation*}

  As stated we have $\fun{U}(\fun{D}_{\cat{kh}}\topo{X}) \subseteq
  \fun{D}(\fun{U}\topo{X})$. Observe
  \begin{equation*}
    \psi'(\lambda^{\Box}(a))
      = \dbox a
      = \{ W \in \fun{D}(\fun{U}\topo{X}) \mid a \in W \}
        \cap \fun{U}(\fun{D}_{\cat{kh}}\topo{X})
      = \lambda^{\Box}(a) \cap \fun{U}(\fun{D}_{\cat{kh}}\topo{X}),
  \end{equation*}
  and similarly $\psi'(\lambda^{\Diamond}(a)) = \lambda^{\Diamond}(a) \cap
  \fun{U}(\fun{D}_{\cat{kh}}\topo{X})$.
  Suppose the identity in~\eqref{ch4-eq-lift-D} holds, then we have
  \begin{align*}
    \bigcup_{i \in I}\Big(&\bigcap_{j \in J}\psi'(\lambda^{\Box}(a_{i,j})) \cap
    \bigcap_{j' \in J'} \psi'(\lambda^{\Diamond}(a_{i,j'}))\Big) \\
      &= \bigcup_{i \in I}\Big(\bigcap_{j \in J}(\lambda^{\Box}(a_{i,j}) \cap
         \fun{U}(\fun{D}_{\cat{kh}}\topo{X})) \cap \bigcap_{j' \in J'}
         (\lambda^{\Diamond}(a_{i,j'}) \cap \fun{U}(\fun{D}_{\cat{kh}}\topo{X}))
         \Big) \\
      &= \bigcup_{i \in I}\Big(\fun{U}(\fun{D}_{\cat{kh}}\topo{X}) \cap
         \bigcap_{j \in J}\lambda^{\Box}(a_{i,j}) \cap \bigcap_{j' \in J'}
         \lambda^{\Diamond}(a_{i,j'}) \Big) \\
      &= \fun{U}(\fun{D}_{\cat{kh}}\topo{X}) \cap \bigcup_{i \in I}
         \Big(\bigcap_{j \in J}\lambda^{\Box}(a_{i,j}) \cap
         \bigcap_{j' \in J'} \lambda^{\Diamond}(a_{i,j'}) \Big) \\
      &= \fun{U}(\fun{D}_{\cat{kh}}\topo{X}) \cap \bigcup_{k \in K}
         \Big(\bigcap_{\ell \in L}\lambda^{\Box}(a_{k,\ell}) \cap
         \bigcap_{\ell' \in L'} \lambda^{\Diamond}(a_{k, \ell'})\Big) \\
      &= \bigcup_{k \in K}\Big(\bigcap_{\ell \in L}
         \psi'(\lambda^{\Box}(a_{k,\ell})) \cap \bigcap_{\ell' \in L'}
         \psi'(\lambda^{\Diamond}(a_{k, \ell'}))\Big).
  \end{align*}
  So $\psi'$ is well defined.

  It is easy to see that $\bigveeup_{b \in B} \lambda(b) \sim
  \lambda(\bigveeup B)$ implies $\big(\bigveeup_{b \in B} \lambda(b),
  \lambda(\bigveeup B)\big) \in \ker(\psi)$ for $\lambda \in \{ \lambda^{\Box},
  \lambda^{\Diamond} \}$.  Since these pairs generate the congruence of
  Definition~\ref{def-Fddot}, we have ${\sim} = \ker(q_{\topo{X}}) \subseteq
  \ker(\psi')$ and hence there exists a map
  $\psi : \hat{\fun{F}}_{\{\lambda^{\Box}, \lambda^{\Diamond} \}}\topo{X} \to
  \fun{opn}(\hat{\fun{T}}\topo{X})$ such that the diagram
  in~\eqref{ch4-eq-lift-D-diagram} commutes. Therefore $\psi$ is (well) defined on
  generators by $[\lambda^{\Box}(a)] \mapsto \Box a$ and $[\lambda^{\Diamond}(a)]
  \mapsto \Diamond a$. One can easily check that $\psi \circ \phi = \id$ and
  $\phi \circ \psi = \id$ by looking at the action on the generators. It follows
  that $\phi$ is injective.

  This entails that for compact Hausdorff spaces $\topo{X}$,
  \begin{equation*}
    \hat{\fun{D}}_{\{\lambda^{\Box}, \lambda^{\Diamond}\}}\topo{X}
      = \fun{D}_{\cat{kh}}\topo{X},
  \end{equation*}
  Furthermore, it can be seen that for continuous maps $f : \topo{X} \to
  \topo{X}'$ we have $\fun{F}_{\{\lambda^{\Box}, \lambda^{\Diamond} \}}f
  = \fun{opn}(\fun{D}_{\cat{kh}}f)$. As a consequence, when restricted to
  $\cat{KHaus}$ we have~\eqref{eq-exm-lift-monotone-5} indeed. That is, lifting
  the monotone functor on $\cat{Set}$ with respect to the box/diamond lifting
  yields a generalisation of the monotone functor on $\cat{KHaus}$ from
  Definition~\ref{def-monotone-functor}.
\end{exa}

\begin{exa}\label{exm-lift-viet}
  Using similar techniques as in the previous example, one can show that, when
  restricted to $\cat{KHaus}$, the topological Kupke-Kurz-Pattinson lift of
  $\fun{P}$ with respect to the usual box and diamond lifting coincides with the
  Vietoris functor.
  (An algebraic description similar to the one in Theorem~\ref{thm:D-M-dual-obj}
  is given in~\cite[Proposition~III4.6]{Joh82}.)
\end{exa}

\begin{exa}
  Not every endofunctor on $\cat{Top}$ can be obtained as the lift of a
  $\cat{Set}$-functor with respect to a (cleverly) chosen set of predicate
  liftings in the sense of Definition~\ref{def-KKP-sob}. A trivial counterexample
  is the functor $\fun{F} : \cat{Top} \to \cat{Top}$ from
  Example~\ref{exm-functor-F}. For every topological space $\topo{X}$ we have
  $\fun{F}\topo{X} = \two$, which is not a $T_0$ space, hence not a sober
  space. Therefore $\fun{F}$ does not preserve sobriety, while every lifted
  functor automatically preserves sobriety. Thus $\fun{F}$ is not the lift of any
  $\cat{Set}$-functor.
\end{exa}

\subsection{Lifting predicate liftings}
  We describe how to lift a predicate lifting to an open predicate lifting.
  Recall that $\fun{Z} : \cat{Frm} \to \cat{Set}$ is the forgetful functor which
  sends a frame to its underlying set.

\begin{defi}\label{def-lift-pred-lift}
  Let $\Lambda$ be a collection of predicate liftings for a functor
  $\fun{T} : \cat{Set} \to \cat{Set}$.
  A predicate lifting
  $\lambda : \breve{\fun{P}}^n \to \breve{\fun{P}} \circ \fun{T}$ in $\Lambda$
  induces an open predicate lifting
  $\widehat{\lambda} : \fun{\Omega}^n \to \fun{\Omega} \circ \hat{\fun{T}}_{\Lambda}$ for
  $\hat{\fun{T}}_{\Lambda}$ via
  \begin{equation*}
    \begin{tikzcd}[column sep=4.5em]
      \fun{\Omega}^n\topo{X}
            \arrow[r, "\lambda_{\fun{U}\topo{X}}" above]
        & \fun{Z}(\dot{\fun{F}}_{\Lambda}\topo{X})
            \arrow[r, "\fun{Z}q_{\topo{X}}" above]
        & \fun{Z}(\hat{\fun{F}}_{\Lambda}\topo{X})
            \arrow[r, "\fun{Z}k_{\hat{\fun{F}}_{\Lambda}\topo{X}}" above]
        & \fun{\Omega}(\fun{pt}(\hat{\fun{F}}_{\Lambda}\topo{X}))
          = \fun{\Omega} (\hat{\fun{T}}_{\Lambda}\topo{X}).
    \end{tikzcd}
  \end{equation*}
  By $\lambda_{\fun{U}\topo{X}}$ we actually mean the restriction of
  $\lambda_{\fun{U}\topo{X}}$ to $\fun{\Omega}^n\topo{X} \subseteq
  \breve{\fun{P}}(\fun{U}\topo{X})$. The map $k_{\fun{F}\topo{X}}$ is the frame
  homomorphism given by $a \mapsto \{ p \in \fun{pt}(\fun{F}_{\Lambda}\topo{X})
  \mid p(a) = 1 \}$. Then $\widehat{\Lambda} := \{ \widehat{\lambda} \mid \lambda
  \in  \Lambda \}$ is a geometric modal signature for $\hat{\fun{T}}_{\Lambda}$.
\end{defi}

\begin{lem}
  The assignment $\widehat{\lambda} : \fun{\Omega}^n \to \fun{\Omega}
  \circ \hat{\fun{T}}_{\Lambda}$ is a natural transformation.
\end{lem}
\begin{proof}
  For a continuous function $f : \topo{X} \to \topo{X}'$ the following diagram
  commutes in $\cat{Set}$:
  \begin{equation*}
    \begin{tikzcd}[column sep=4.5em]
      \fun{\Omega}^n\topo{X}' \arrow[d, "(f^{-1})^n" left]
            \arrow[r, "\lambda_{\fun{U}\topo{X}'}" above]
        &\fun{Z}(\dot{\fun{F}}_{\Lambda}\topo{X}')
            \arrow[r, "\fun{Z}q_{\topo{X}'}" above]
            \arrow[d, "(\fun{T}f)^{-1}"]
        & \fun{Z}(\hat{\fun{F}}_{\Lambda}\topo{X}')
            \arrow[d, "(\fun{T}f)^{-1}"]
            \arrow[r, "\fun{Z}k_{\hat{\fun{F}}_{\Lambda}\topo{X}'}" above]
        & \fun{\Omega}(\fun{pt}(\hat{\fun{F}}_{\Lambda}\topo{X}'))
            \arrow[d, "\fun{\Omega}(\fun{pt}((\fun{T}f)^{-1}))" right] \\
      \fun{\Omega}^n\topo{X}
            \arrow[r, "\lambda_{\fun{U}\topo{X}}" below]
        & \fun{Z}(\dot{\fun{F}}_{\Lambda}\topo{X})
            \arrow[r, "\fun{Z}q_{\topo{X}}" below]
        & \fun{Z}(\hat{\fun{F}}_{\Lambda}\topo{X})
            \arrow[r, "\fun{Z}k_{\hat{\fun{F}}_{\Lambda}\topo{X}}" below]
        & \fun{\Omega}(\fun{pt}(\hat{\fun{F}}_{\Lambda}\topo{X}))
    \end{tikzcd}
  \end{equation*}
  Commutativity of the left square follows from naturality of $\lambda$,
  commutativity of the middle square follows from the proof of Lemma~\ref{lem-ddF}
  and commutativity of the right square can be seen as follows:
  let $a_1', \ldots, a_n' \in \fun{\Omega}\topo{X}'$, then
  \begin{align*}
    \fun{\Omega}(\fun{pt}((\fun{T}f)^{-1})) \,\circ\,
      &\fun{Z}k_{\fun{F}_{\Lambda}\topo{X}'}(\lambda_{\fun{U}\topo{X}'}(a_1',
         \ldots, a_n')) \\
      &= \{ q \in \fun{pt}(\fun{F}_{\Lambda}\topo{X}) \mid q \circ
         (\fun{T}f)^{-1}(\lambda_{\fun{U}\topo{X}'}(a_1', \ldots, a_n')) = 1 \}
         \\
      &= \fun{Z}k_{\fun{F}_{\Lambda}\topo{X}}((\fun{T}f)^{-1}
         (\lambda_{\fun{U}\topo{X}'}(a_1', \ldots, a_n'))).
  \end{align*}
  So $\widehat{\lambda}$ is an open predicate lifting.
\end{proof}

  The nature of the definitions of $\hat{\fun{T}}_{\Lambda}$ and
  $\widehat{\Lambda}$ yields the following desirable results.

\begin{prop}\label{prop-lifted-pl} \hfill
  \begin{enumerate}
    \item\label{prop-lifted-pl-1}
          Let $\fun{T} : \cat{Set} \to \cat{Set}$ be a functor and $\Lambda$ a
          collection of predicate liftings for $\fun{T}$. Then
          $\widehat{\Lambda}$ is characteristic for $\hat{\fun{T}}_{\Lambda}$.
    \item\label{prop-lifted-pl-2}
          If $\lambda \in \Lambda$ is monotone, then $\hat{\lambda} \in
          \hat{\Lambda}$ is Scott-continuous.
  \end{enumerate}
\end{prop}
\begin{proof}
  Let $\topo{X}$ be a topological space. For the first item, we need to show that
  the collection
  \begin{equation}\label{eq-lpl-char}
    \{ \widehat{\lambda}(a_1, \ldots, a_n) \mid \lambda \in \Lambda
    \text{ $n$-ary}, a_i \in \fun{\Omega}\topo{X} \}
  \end{equation}
  forms a subbase for the topology on $\hat{\fun{T}}_{\Lambda}\topo{X}$. An
  arbitrary nonempty open set of $\hat{\fun{T}}_{\Lambda}\topo{X}$ is of the form
  $\tilde{x} = \{ p \in \fun{pt}(\hat{\fun{F}}_{\Lambda}\topo{X}) \mid p(x) = 1
  \}$, for $x \in \hat{\fun{F}}_{\Lambda}\topo{X}$. An arbitrary element of
  $\hat{\fun{F}}_{\Lambda}\topo{X}$ is the equivalence class of an arbitrary
  union of finite intersections of elements of the form
  $\lambda_{\fun{U}\topo{X}}(a_1, \ldots, a_n)$, for $\lambda \in \Lambda$ and
  $a_1, \ldots, a_n \in \fun{\Omega}\topo{X}$. So we may write
  $x = \bigcup_{i \in I}(\bigcap_{j \in J_i} [\lambda^{i,j}_{\fun{U}\topo{X}}
  (a^{i,j}_1, \ldots, a^{i,j}_{n_{i,j}})])$
  for some index set $I$, finite index sets $J_i$, $\lambda^{i,j} \in \Lambda$
  and open sets $a^{i,j}_k \in \fun{\Omega}\topo{X}$. We get
  \begin{equation*}
    \tilde{x} = \bigcup_{i \in I}\Big(\bigcap_{j \in J_i}
    \reallywidetilde{[\lambda^{i,j}_{\fun{U}\topo{X}}
    (a^{i,j}_1, \ldots, a^{i,j}_{n_{i,j}})]}\Big)
      = \bigcup_{i \in I}\Big(\bigcap_{j \in J_i}
      \widehat{\lambda}^{i,j}_{\topo{X}}(a^{i,j}_1, \ldots, a^{i,j}_{n_{i,j}})
      \Big).
  \end{equation*}
  The second equality follows from Definition~\ref{def-lift-pred-lift}.
  This shows that the open sets in~\eqref{eq-lpl-char} indeed form a subbase for
  the open sets of $\hat{\fun{T}}_{\Lambda}\topo{X}$.

  The second item follows immediately from the definitions.
\end{proof}

\begin{exa}\label{exa:lift-pred-lift-mon}
  Let $\lambda^{\Box}$ be the box-predicate lifting for the monotone functor
  $\fun{D}$ on $\cat{Set}$ from Example~\ref{exm-mon-set-pred}.
  Then the procedure from Definition~\ref{def-lift-pred-lift} sends an open
  $a$ in a compact Hausdorff space
  $\topo{X}$ to $[\lambda^{\Box}_{\fun{U}\topo{X}}(a)]$ in
  $\hat{\fun{F}}_{\Lambda}\topo{X}$. We know from Example~\ref{exm-lift-monotone}
  that $\hat{\fun{F}}_{\Lambda}\topo{X}$ is isomorphic to $\fun{M}(\fun{opn}\topo{X})$
  (provided $\topo{X}$ is compact Hausdorff) and the element
  $[\lambda_{\fun{U}\topo{X}}^{\Box}(a)]$ corresponds to $\Box a \in
  \fun{M}(\fun{opn}\topo{X})$. Therefore
  \[
    \hat{\lambda}_{\topo{X}}^{\Box}(a)
      = \tilde{\Box a}
      \in \fun{\Omega} \circ \fun{pt} \circ \fun{M} \circ \fun{opn} \topo{X}
      = \fun{\Omega} \circ \fun{pt} \circ \hat{\fun{F}}_{\Lambda}\topo{X}
  \]
  We know from Theorem~\ref{thm:D-M-dual-obj} that
  $\fun{pt} \circ \fun{M} \circ \fun{opn} \cong \fun{D}_{\cat{kh}}$
  and under this isomorphism $\tilde{\Box a} \in \fun{\Omega} \circ \fun{pt}
  \circ \fun{M} \circ \fun{opn} \topo{X} $ is sent to
  $\dbox a \in \fun{\Omega}(\fun{D}_{\cat{kh}}\topo{X})$. Therefore
  \[
    \hat{\lambda}^{\Box}_{\topo{X}}(a) = \dbox a,
  \]
  which yields the open predicate liftings defined in Subsection~\ref{subsec:mon-lan}.
  Similarly, the predicate lifting $\lambda^{\Diamond}$ from
  Example~\ref{exm-mon-set-pred} lifts the similarly-named open predicate lifting
  for $\fun{D}_{\cat{kh}}$ from Subsection~\ref{subsec:mon-lan}.
\end{exa}

\begin{exa}
  Similar to Example~\ref{exa:lift-pred-lift-mon}, the predicate liftings
  from Example~\ref{exm-krip-pred} for the powerset functor lift to the open
  predicate liftings from Example~\ref{exm-viet} for the Vietoris functor on
  $\cat{KHaus}$.
\end{exa}

  Let $\fun{T}$ be an endofunctor on $\cat{Set}$ and $\Lambda$ a set of
  monotone predicate liftings for $\fun{T}$. Moreover, suppose that we have a
  collection $\Ax$ of one-step axioms and rules for the (classical) modal
  language $\ML(\Lambda)$.
  Then $\Ax$ yields axioms and rules for the language $\GML(\hat{\Lambda})$
  (by simply replacing every occurrence of $\lambda$ with $\hat{\lambda}$).
  We write $\Ax'$ for the collection of these axioms, together with the
  Scott-continuity axioms for all open predicate liftings in $\hat{\Lambda}$.
  Then by construction $\Ax'$ is sound for $\hat{\fun{T}}_{\Lambda}$-coalgebras,
  and the pair $(\hat{\Lambda}, \Ax')$ gives rise to an endofunctor
  $\fun{L}'$ on $\cat{Frm}$ via the procedure from Definition~\ref{def:ax-fun}.

  It is now natural to ask whether we can apply Theorem~\ref{thm:cor-com}
  to this setting to obtain a completeness result.
  However, while soundness of $\Ax'$ implies the existence of a
  natural transformation
  $\fun{L}' \circ \fun{opn} \to \fun{opn} \circ \hat{\fun{T}}_{\Lambda}$,
  there seems to be no guarantee that this yields a natural isomorphism.
  A potential duality of restrictions of the functors to a suitable
  subcategory is furthermore frustrated by the fact that we have no generic
  preservation results, so that $\hat{\fun{T}}_{\Lambda}$ need not even
  restrict to $\cat{KSob}$ or $\cat{KHaus}$.
  We highlight this question as an interesting direction for further research.

\section{A final model construction}\label{sec-final}\label{sec:final}

  We construct a final model in $\cat{Mod}(\fun{T})$ for a functor $\fun{T}$
  where either $\fun{T}$ is an endofunctor on $\cat{Sob}$, or $\fun{T}$ is an
  endofunctor on $\cat{Top}$ which preserves sobriety. This assumption need not
  be problematic: If a functor on $\cat{Top}$ does not preserve sobriety we can
  look at its sobrification.

  There are various ways of obtaining final
  coalgebras~\cite{KupKurVen04,AdaMilVel05,MosVig04,MosVig06}.
  Our strategy for obtaining a final coalgebra is similar to that by
  Moss and Viglizzo~\cite{MosVig04,MosVig06} in the sense that we use a
  language (in our case induced by the geometric similarity type) without a
  logical system.
  Specifically,
  given an endofunctor $\fun{T}$ on $\cat{Top}$, we use a notion of semantic
  equivalence to obtain an equivalence relation on $\GML(\Phi, \Lambda)$.
  We prove that $\GML(\Phi, \Lambda)$ modulo this equivalence relation forms
  a frame, denoted by $\mathbf{E}$.
  Subsequently, we set
  $\fun{L} = \fun{opn} \circ \fun{T} \circ \fun{pt} : \cat{Frm} \to \cat{Frm}$
  and endow $\mathbf{E}$ with an $\fun{L}$-algebra structure $\delta$.
  The topological space $\fun{pt}\mathbf{E}$ is the state space for the
  final model, and the map $\delta$ gives rise to a $\fun{T}$-algebra
  structure on $\fun{pt}\mathbf{E}$. Finally, we show that with the
  canonical valuation this gives rise to a final model in $\cat{Mod}(\fun{T})$.

  We hasten to say that the functor $\fun{L}$ used in this construction
  need not coincide with the one discussed at the end of Section~\ref{sec:lift}.
  From a more abstract point of view~\cite{BonKur05,BonKur06,Kli07}, the functor $\fun{L}$
  defined by $\fun{L} = \fun{opn} \circ \fun{T} \circ \fun{pt}$ gives rise
  to a logic in its own right. We discuss the implications of the final
  model construction on $\fun{L}$ in more detail in the conclusion.

  It is well known that coalgebras for endofunctors on $\cat{Set}$ do not
  in general have final coalgebras. In particular, there exists no final
  $\fun{P}$-coalgebra,
  as this would imply the existence of a set $X$ such that $X$ and $\fun{P}X$
  have the same cardinality.
  One way to remedy this is by using topology to restrict the collection of
  admissible morphisms.
  We show that in presence of a Scott-continuous
  characteristic geometric modal signature we \emph{can} construct a
  final coalgebra. Intuitively, Scott-continuity gives us a finitary handle
  on the final coalgebra we construct.
  (For comparison, observe that the set of predicate liftings
  $\Lambda = \{ \lambda^{\Box}, \lambda^{\Diamond} \}$ for the powerset functor
  $\fun{P} : \cat{Set} \to \cat{Set}$ is not Scott-continuous.)

\begin{asm}\label{asm:final}
  Throughout this section, fix an endofunctor $\fun{T}$ on $\cat{Top}$ which
  preserves sobriety, and a Scott-continuous characteristic geometric modal
  signature $\Lambda$ for $\fun{T}$. Recall that $\Phi$ is a set of proposition
  letters.
\end{asm}

  Suppose that the functor $\fun{T} : \fun{Top} \to \fun{Top}$ with
  geometric modal signature $\Lambda$ arises from an endofunctor
  $\fun{T}'$ on $\cat{Set}$ with predicate liftings $\Lambda'$ via the
  procedure in Section~\ref{sec:lift}. Then $\fun{T}$ preserves sobriety by
  definition and $\Lambda$ is characteristic as a consequence of
  Proposition~\ref{prop-lifted-pl}\eqref{prop-lifted-pl-1}.
  If moreover all the predicate liftings in $\Lambda'$ are monotone,
  then Proposition~\ref{prop-lifted-pl}\eqref{prop-lifted-pl-2} tells us that
  the open predicate liftings in $\Lambda$ are Scott-continuous.

\begin{defi}\label{def-final-equiv}
  Call two formulas $\phi$ and $\psi$ \emph{equivalent} in $\cat{Mod}(\fun{T})$
  with respect to $\Lambda$, notation: $\phi \equiv_{\fun{T}, \Lambda} \psi$, if
  $\mf{X}, x \Vdash \phi$ iff $\mf{X}, x \Vdash \psi$ for all
  $\mf{X} \in \cat{Mod}(\fun{T})$ and $x \in \mf{X}$. Denote the equivalence
  class of $\phi$ in $\GML(\Phi, \Lambda)$ by $[\phi]$. Let $\mathbf{E} =
  \mathbf{E}(\fun{T}, \Phi, \Lambda)$ be the collection of formulas modulo
  $\equiv_{\fun{T}, \Lambda}$.
\end{defi}

  Recall that a coherent formula is one which does not involve arbitrary
  disjunctions.

\begin{lem}[Normal form]\label{lem-norm}
  Under the assumption, every formula is equivalent to a formula of the form
  $\bigvee_{i \in I} \phi_i$, where all the $\phi_i$ are coherent.
\end{lem}
\begin{proof}
  The proof proceeds by induction on the complexity of the formula.
  Suppose $\phi = \phi_1 \vee \phi_2$. By induction we may assume that
  $\phi_1 \equiv_{\fun{T}, \Lambda} \bigvee_{i \in I}\psi_i$ and
  $\phi_2 \equiv_{\fun{T}, \Lambda} \bigvee_{j \in J}\psi_j$, where all the
  $\psi_i$ and $\psi_j$ are coherent, and we have $\phi \equiv_{\fun{T}, \Lambda}
  \bigvee_{i \in I \cup J} \psi_i$, as desired. If $\phi = \phi_1 \wedge \phi_2$,
  then $\phi \equiv_{\fun{T}, \Lambda} (\bigvee_{i \in I}\psi_i) \wedge
  (\bigvee_{j \in J}\psi_j) \equiv_{\fun{T}, \Lambda}
  \bigvee_{(i,j) \in I \times J} \psi_i \wedge \psi_j$. Lastly, suppose
  $\phi = \heartsuit^{\lambda}(\bigvee_{i \in I}\psi_i)$, where all the $\psi_i$
  are coherent. Then we have
  $\bigvee_{i \in I}\psi_i = \bigvee \{ \bigvee_{i \in I'} \psi_i \mid I'
  \subseteq I \text{ finite} \}$ and by construction the set
  $\{ \llb \bigvee_{i \in I'} \phi_i \rrb^{\mf{X}} \mid I' \subseteq I, I'
  \text{ finite} \}$ is directed for every $\fun{T}$-model $\mf{X} = (\topo{X},
  \gamma, V)$. Hence by Scott-continuity of $\lambda$ we obtain
  \begin{equation*}
    \lambda_{\topo{X}}(\llb \bigvee_{i \in I}\psi_i \rrb^{\mf{X}})
      = \lambda_{\topo{X}}(\bigcup \{ \llb \bigvee_{i \in I'}\psi_i \rrb^{\mf{X}}
        \mid I' \subseteq I \text{ finite} \})
      = \bigcup \{ \lambda_{\topo{X}}(\llb \bigvee_{i \in I'}\psi_i \rrb^{\mf{X}})
        \mid I' \subseteq I \text{ finite} \}.
  \end{equation*}
  Therefore $\phi \equiv_{\fun{T}, \Lambda} \bigvee \{ \heartsuit^{\lambda}
  (\bigvee_{i \in I}\psi_i) \mid I' \subseteq I \text{ finite} \}$, i.e.~$\phi$
  is equivalent to an arbitrary disjunction of coherent formulas. The case for
  $n$-ary modalities is similar.
\end{proof}

\begin{cor}
  The collection $\mathbf{E}$ from Definition~\ref{def-final-equiv} is a set.
\end{cor}
\begin{proof}
  This follows immediately from Lemma~\ref{lem-norm} and the fact that the
  collection of coherent formulas is a set.
\end{proof}

\begin{defi}
  Define top, bottom, conjunction and arbitrary disjunction on $\mathbf{E}$ by
  $\top_{\mathbf{E}} = [\top]$, $\bot_{\mathbf{E}} = [\bot]$,
  $[\phi] \wedge [\psi] := [\phi \wedge \psi]$ and
  $\bigvee_{i \in I} [\phi_i] := [\bigvee_{i \in I}\phi_i ]$.
\end{defi}
  It is easy to check that $\mathbf{E}$ now forms a frame.
  The theory of a point $x$ in a geometric $\fun{T}$-model $\mf{X}$ is the
  collection of formulas that are true at $x$. The theory of $x$ defines a
  point of $\mathbf{E}$. This motivates the next definition.

\begin{defi}\label{ch3-def-th-map}
  Let $\topo{Z} = \fun{pt} \mathbf{E}$. For every geometric $\fun{T}$-model
  $\mf{X} = (\topo{X}, \gamma, V)$ define the theory map by
  \begin{align*}
    \th_{\mf{X}}
      : \topo{X} \to \topo{Z}
      : x \mapsto \{ [\phi] \in \mathbf{E} \mid \mf{X}, x \Vdash \phi \}.
      &\qedhere
  \end{align*}
\end{defi}

  The space $\topo{Z}$ will turn out to be the state space of a final model in
  $\cat{Mod}(\fun{T})$ and we will see in Proposition~\ref{prop-Z-exists} that
  the theory maps are $\fun{T}$-model morphisms.

\begin{defi}
  Set $\fun{L} = \fun{opn} \circ\, \fun{T} \circ \fun{pt} : \cat{Frm} \to \cat{Frm}$.
  This functor restricts to an endofunctor on $\cat{SFrm}$ which is dual to the
  restriction of $\fun{T}$ to $\cat{Sob}$.
  Since $\Lambda$ is characteristic, the frame $\fun{L}\mathbf{E}$ is generated
  by $\{ \lambda_{\topo{X}}(\widetilde{[\phi_1]}, \ldots, \widetilde{[\phi_n]})
  \mid \lambda \in \Lambda, \phi_i \in \GML(\Phi, \Lambda) \}$.
  Define an $\fun{L}$-algebra structure $\delta : \fun{L}\mathbf{E} \to
  \mathbf{E}$ on generators by
  \begin{equation*}
    \delta
      : \fun{L}\mathbf{E} \to \mathbf{E}
      : \lambda_{\fun{pt}\mathbf{E}}(\widetilde{[\phi_1]}, \ldots,
        \widetilde{[\phi_n]})
        \mapsto [\heartsuit^{\lambda}(\phi_1, \ldots, \phi_n)].
  \end{equation*}
\end{defi}

  We need to show that $\delta$ is well defined. For this purpose it suffices to
  show that the images of the generators of $\mathbf{E}$ satisfy the same
  relations that they satisfy in $\fun{L}\mathbf{E}$. Recall
  $\topo{Z} = \fun{pt}\mathbf{E}$, and therefore $\fun{L}\mathbf{E} = \fun{opn}(\fun{T}\topo{Z})$.

\begin{lem}\label{lem-delta-weldef}
  If
  \begin{equation}\label{eqiv}
    \bigcup_{i \in I}\Big( \bigcap_{j \in J_i} \lambda^{i,j}_{\topo{Z}}
    (\tilde{\phi}_1^{i,j}, \ldots, \tilde{\phi}_{n_{i,j}}^{i,j})\Big)
    = \bigcup_{k \in K}\Big( \bigcap_{\ell \in L_k} \lambda^{k,\ell}_{\topo{Z}}
    (\tilde{\phi}_1^{k,\ell}, \ldots, \tilde{\phi}_{n_{k,\ell}}^{k,\ell})\Big)
  \end{equation}
  then
  \begin{equation}\label{eq}
    \bigvee_{i \in I}\Big( \bigwedge_{j \in J_i} \heartsuit^{\lambda^{i,j}}
        ({\phi}_1^{i,j}, \ldots, {\phi}_{n_{i,j}}^{i,j})\Big)
      \equiv_{\fun{T}, \Lambda} \bigvee_{k \in K}\Big( \bigwedge_{\ell \in L_k}
        \heartsuit^{\lambda^{k,\ell}}({\phi}_1^{k,\ell}, \ldots,
        {\phi}_{n_{k,\ell}}^{k,\ell})\Big),
  \end{equation}
  where the $J_i$ and $L_k$ are finite index sets and $I$ and $K$ are index sets
  of arbitrary size.
\end{lem}
\begin{proof}
  We will see that this follows from naturality of $\lambda$. Our strategy is to
  show that the truth sets of the left hand side and right hand side
  of~\eqref{eq} coincide in every geometric $\fun{T}$-model
  $\mf{X} = (\topo{X}, \gamma, V)$.

  Observe that the map $\th_{\mf{X}} : \topo{X} \to \topo{Z}$, which sends a
  point to its theory, is continuous because
  \begin{equation}\label{eqiii}
    \th_{\mf{X}}^{-1}(\tilde{\phi}) = \llb \phi \rrb^{\mf{X}},
  \end{equation}
  which is open in $\topo{X}$ for all formulas $\phi$.
  Compute
  \begin{align*}
    \bigcup_{i \in I}\Big( \bigcap_{j \in J_i} \lambda^{i,j}_{\topo{X}}
      &(\llb\phi_1^{i,j}\rrb^{\mf{X}}, \ldots,
         \llb\phi_{n_{i,j}}^{i,j}\rrb^{\mf{X}})\Big) \\
      &= \bigcup_{i \in I}\Big( \bigcap_{j \in J_i} \lambda^{i,j}_{\topo{X}}
         (\th_{\mf{X}}^{-1}(\tilde{\phi}_1^{i,j}), \ldots, \th_{\mf{X}}^{-1}
         (\tilde{\phi}_{n_{i,j}}^{i,j}))\Big)
      &\text{(By~\eqref{eqiii})}\\
      &= \bigcup_{i \in I}\Big( \bigcap_{j \in J_i} (\fun{T}
         \th_{\mf{X}})^{-1}\big(\lambda^{i,j}_{\topo{Z}}(\tilde{\phi}_1^{i,j},
         \ldots, \tilde{\phi}_{n_{i,j}}^{i,j})\big)\Big)
      &\text{(Naturality of $\lambda$)} \\
      &= (\fun{T}\th_{\mf{X}})^{-1}\Big(\bigcup_{i \in I}\Big(
         \bigcap_{j \in J_i} \lambda^{i,j}_{\topo{Z}}(\tilde{\phi}_1^{i,j},
         \ldots, \tilde{\phi}_{n_{i,j}}^{i,j})\Big)\Big)
      &\text{($\star$)}\\
      &= (\fun{T}\th_{\mf{X}})^{-1}\Big(\bigcup_{k \in K}\Big(
         \bigcap_{\ell \in L_k} \lambda^{k,\ell}_{\topo{Z}}
         (\tilde{\psi}_1^{k,\ell}, \ldots, \tilde{\psi}_{n_{k,\ell}}^{k,\ell})
         \Big)\Big)
      &\text{(Assumption~\eqref{eqiv})} \\
      &= \bigcup_{k \in K}\Big( \bigcap_{\ell \in L_k}
         (\fun{T}\th_{\mf{X}})^{-1}\big(\lambda^{k,\ell}_{\topo{Z}}
         (\tilde{\psi}_1^{k,\ell}, \ldots, \tilde{\psi}_{n_{k,\ell}}^{k,\ell})
         \big)\Big)
      &\text{($\star$)} \\
      &= \bigcup_{k \in K}\Big( \bigcap_{\ell \in L_k}
         \lambda^{k,\ell}_{\topo{X}}(\th_{\mf{X}}^{-1}(\psi_1^{k,\ell}), \ldots,
         \th_{\mf{X}}^{-1}(\psi_{n_{k,\ell}}^{k,\ell}))\Big)
      &\text{(Naturality of $\lambda$)} \\
      &= \bigcup_{k \in K}\Big( \bigcap_{\ell \in L_k}
         \lambda^{k,\ell}_{\topo{X}}(\llb\psi_1^{k,\ell}\rrb^{\mf{X}}, \ldots,
         \llb\psi_{n_{k,\ell}}^{k,\ell}\rrb^{\mf{X}})\Big).
      &\text{(By~\eqref{eqiii})}
\end{align*}
  The steps with ($\star$) hold because inverse images of functions preserve all
  unions and intersections.
  This entails that for all geometric $\fun{T}$-models and all states $x$ in
  $\mf{X}$ we have
  \begin{equation*}
      \mf{X}, x \Vdash \bigvee_{i \in I}\Big( \bigwedge_{j \in J_i}
      \heartsuit^{\lambda^{i,j}}({\phi}_1^{i,j}, \ldots, {\phi}_{n_{i,j}}^{i,j})
      \Big)
    \quad\text{iff}\quad
      \mf{X}, x \Vdash \bigvee_{k \in K}\Big( \bigwedge_{\ell \in L_k}
      \heartsuit^{\lambda^{k,\ell}}({\phi}_1^{k,\ell}, \ldots,
       {\phi}_{n_{k,\ell}}^{k,\ell})\Big),
  \end{equation*}
  and hence~\eqref{eq} holds. Therefore $\delta$ is well defined.
\end{proof}

  The algebra structure on $\mathbf{E}$ entails a coalgebra structure on
  $\topo{Z}$.

\begin{defi}
  Let $\zeta : \topo{Z} \to \fun{T}\topo{Z}$ be the composition
  \begin{equation*}
    \begin{tikzcd}
      \fun{pt}\mathbf{E}
            \arrow[r, "\fun{pt} \delta" above]
        & \fun{pt}(\fun{L}\mathbf{E})
            \arrow[r, equal]
        & \fun{pt}(\fun{opn}(\fun{T}(\fun{pt}\mathbf{E})))
            \arrow[r, "k_{\fun{T}(\fun{pt}\mathbf{E})}^{-1}" above]
        &\fun{T}(\fun{pt}\mathbf{E}).
    \end{tikzcd}
  \end{equation*}
  Here $k_{\fun{T}(\fun{pt}\mathbf{E})} : \fun{T}(\fun{pt}\mathbf{E}) \to \fun{pt}(\fun{opn}(\fun{T}
  (\fun{pt}\mathbf{E})))$ is the isomorphism given in Remark~\ref{rem-isbell-iso}.
  Since $\topo{Z} = \fun{pt}\mathbf{E}$ this indeed defines a map
  $\topo{Z} \to \fun{T}\topo{Z}$.
\end{defi}

  For an object $\Gamma \in \topo{Z}$, the element $(\fun{pt} \delta)(\Gamma)$ is
  the completely prime filter
  \begin{equation*}
    F = \{ \lambda(\tilde{\phi}_1, \ldots, \tilde{\phi}_n) \in \fun{pt}(\fun{opn}
        (\fun{T}(\fun{pt}\mathbf{E}))) \mid  [\heartsuit^{\lambda}(\phi_1,
        \ldots, \phi_n)] \in \Gamma \}
  \end{equation*}
  in $\fun{pt}(\fun{opn}(\fun{T}(\fun{pt}\mathbf{E})))$.
  The element $\zeta(\Gamma)$ is the unique element in
  $\fun{T}(\fun{pt}\mathbf{E})$ corresponding to $F$ under the isomorphism
  $k_{\fun{T}(\fun{pt}\mathbf{E})}$. By definition of
  $k_{\fun{T}(\fun{pt}\mathbf{E})}$, this is the unique element in the
  intersection of
  \begin{equation*}
    \{ \lambda(\tilde{\phi}_1, \ldots, \tilde{\phi}_n) \mid [\heartsuit^{\lambda}
    (\phi_1, \ldots, \phi_n)] \in \Gamma \}.
  \end{equation*}
  Moreover, it follows from the definition of $k_{\fun{T}(\fun{pt}\mathbf{E})}$
  that $[\heartsuit^{\lambda}(\phi_1, \ldots, \phi_n)] \notin \Gamma$ implies
  $\zeta(\Gamma) \notin \lambda(\tilde{\phi}_1, \ldots, \tilde{\phi}_n)$.

\begin{nota}
  If no confusion is likely to occur we will omit the square brackets that
  indicate equivalence classes of formulas in $\mathbf{E}$. That is, we shall
  write $\phi \in \mathbf{E}$ instead of $[\phi] \in \mathbf{E}$.
\end{nota}

  We can now endow the $\fun{T}$-coalgebra $(\topo{Z}, \zeta)$ with a valuation.
  Thereafter we will show that $(\topo{Z}, \zeta)$ together with this valuation
  is final in $\cat{Mod}(\fun{T})$.

\begin{defi}
  Let $V_{\topo{Z}} : \Phi \to \fun{\Omega}\topo{Z}$ be the valuation
  $p \mapsto \tilde{p}$.
\end{defi}

  The triple $\mf{Z} = (\topo{Z}, \zeta, V_{\topo{Z}})$ is a geometric
  $\fun{T}$-model, simply because it is a $\fun{T}$-coalgebra with a valuation.
  We can prove a truth lemma for $\mf{Z}$:

\begin{lem}[Truth lemma]
  We have $\mf{Z}, \Gamma \Vdash \phi$ iff $\phi \in \Gamma$.
\end{lem}
\begin{proof}
  Use induction on the complexity of the formula. The propositional case follows
  immediately from the definition of $V_{\topo{Z}}$. The cases
  $\phi = \phi_1 \wedge \phi_2$ and $\phi = \bigvee_{i \in i}\phi_i$ are routine.
  So suppose $\phi = \heartsuit^{\lambda}(\phi_1, \ldots, \phi_n)$. We have
  \begin{align*}
    \mf{Z}, \Gamma \Vdash \heartsuit^{\lambda}(\phi_1, \ldots, \phi_n)
      \quad&\text{iff}\quad \zeta(\Gamma) \in \lambda_{\topo{Z}}
        (\llb \phi_1 \rrb^{\mf{Z}}, \ldots, \llb \phi_n \rrb^{\mf{Z}})
      &\text{(Definition of $\Vdash$)} \\
      &\text{iff}\quad \zeta(\Gamma) \in \lambda_{\topo{Z}}
        (\tilde{\phi}_1, \ldots, \tilde{\phi}_n)
      &\text{(Induction)} \\
      &\text{iff}\quad \heartsuit^{\lambda}(\phi_1, \ldots, \phi_n) \in \Gamma.
      &\text{(Definition of $\zeta$)}
  \end{align*}
  This proves the lemma.
\end{proof}

\begin{prop}\label{prop-Z-exists}
  For every geometric $\fun{T}$-model $\mf{X} = (\topo{X}, \gamma, V)$ the map
  $\th_{\mf{X}} : \topo{X} \to \topo{Z}$ is a $\fun{T}$-model morphism.
\end{prop}
\begin{proof}
  We need to show that $\th_{\mf{X}}$ is a $\fun{T}$-coalgebra morphism and that
  $\th_{\mf{X}}^{-1} \circ V_{\mf{Z}} = V$. The latter follows from the fact that
  for every proposition letter $p$ we have
  \begin{equation*}
    V(p)
      = \{ x \in \mathbb{X} \mid \mf{X}, x \Vdash p \}
      = \th_{\mf{X}}^{-1}(\tilde{p})
      = \th_{\mf{X}}^{-1}(V_{\mf{Z}}(p)).
  \end{equation*}

  In order to show that $\th_{\mf{X}}$ is a $\fun{T}$-coalgebra morphism, we have
  to show that the following diagram commutes:
  \begin{equation*}
    \begin{tikzcd}
      \topo{X} \arrow[r, "\th_{\mf{X}}" above]
            \arrow[d, "\gamma" left]
        & \topo{Z}
            \arrow[d, "\zeta" right] \\
      \fun{T}\topo{X}
            \arrow[r, "\fun{T}\th_{\mf{X}}" below]
        & \fun{T}\topo{Z}
    \end{tikzcd}
  \end{equation*}
  Let $x \in \topo{X}$. Since $\fun{T}\topo{Z}$ is sober, hence $T_0$, it
  suffices to show that $\fun{T}\th_{\mf{X}}(\gamma(x))$ and
  $\zeta(\th_{\mf{X}}(x))$ are in precisely the same opens of $\fun{T}\topo{Z}$.
  Moreover, we know that the open sets of $\fun{T}\topo{Z}$ are generated by the
  sets $\lambda_{\topo{Z}}(\tilde{\phi}_1, \ldots, \tilde{\phi}_n)$, so it
  suffices to show that for all $\lambda \in \Lambda$ and $\phi_i \in
  \GML(\Phi, \Lambda)$ we have
  \begin{equation*}
    \fun{T}\th_{\mf{X}}(\gamma(x))
      \in \lambda_{\topo{Z}}(\tilde{\phi}_1, \ldots, \tilde{\phi}_n)
    \quad\text{iff}\quad
      \zeta(\th_{\mf{X}}(x)) \in \lambda_{\topo{Z}}(\tilde{\phi}_1, \ldots,
      \tilde{\phi}_n).
  \end{equation*}
  This follows from the following computation,
  \begin{align*}
    \fun{T}\th_{\mf{X}}(\gamma(x))
      &\in \lambda_{\topo{Z}}(\tilde{\phi}_1, \ldots, \tilde{\phi}_n) \\
      &\text{iff}\quad \gamma(x) \in (\fun{T}\th_{\mf{X}})^{-1}
        (\lambda_{\topo{Z}}(\tilde{\phi}_1, \ldots, \tilde{\phi}_n)) \\
      &\text{iff}\quad \gamma(x) \in \lambda_{\topo{X}}(\th_{\mf{X}}^{-1}
        (\tilde{\phi}_1), \ldots, \th_{\mf{X}}^{-1}(\tilde{\phi}_n))
      &\text{(Naturality of $\lambda$)} \\
      &\text{iff}\quad \gamma(x) \in \lambda_{\topo{X}}
        (\llb \phi_1 \rrb^{\mf{X}}, \ldots, \llb \phi_n \rrb^{\mf{X}})
      &\text{(By~\eqref{eqiii})} \\
      &\text{iff}\quad \mf{X}, x \Vdash \heartsuit^{\lambda}
        (\phi_1, \ldots, \phi_n)
      &\text{(Definition of $\Vdash$)} \\
      &\text{iff}\quad \heartsuit^{\lambda}(\phi_1, \ldots, \phi_n) \in
        \th_{\mf{X}}(x)
      &\text{(Definition of $\th_{\mf{X}}$)} \\
      &\text{iff}\quad \zeta(\th_{\mf{X}}(x)) \in
        \lambda_{\topo{Z}}(\tilde{\phi}_1, \ldots, \tilde{\phi}_n)
      &\text{(Definition of $\zeta$)}
  \end{align*}
  This proves the proposition.
\end{proof}

  The developed theory results in the following theorem.

\begin{thm}\label{thm-final-modT}
  Let $\fun{T}$ be an endofunctor on $\cat{Top}$ which preserves sobriety, and
  $\Lambda$ a Scott-continuous characteristic geometric modal signature for
  $\fun{T}$. Then the geometric $\fun{T}$-model
  $\mf{Z} = (\topo{Z}, \zeta, V_{\topo{Z}})$ is final in $\cat{Mod}(\fun{T})$.
\end{thm}
\begin{proof}
  Proposition~\ref{prop-Z-exists} states that for every geometric $\fun{T}$-model
  $\mf{X} = (\topo{X}, \gamma, V)$ there exists a $\fun{T}$-coalgebra morphism
  $\th_{\mf{X}} : \mf{X} \to \mf{Z}$, so we only need to show that this morphism
  is unique.

  Let $f : \mf{X} \to \mf{Z}$ be any coalgebra morphism. We know from
  Proposition~\ref{prop-mor-pres-truth} that coalgebra morphisms preserve truth, so for all
  $x \in \topo{X}$ we have $\phi \in f(x)$ iff $\mf{Z}, f(x) \Vdash \phi$ iff
  $\mf{X}, x \Vdash \phi$. Therefore we must have $f(x) = \th_{\mf{X}}(x)$.
\end{proof}

  As an immediate corollary we obtain the following theorem. Recall from
  Definition~\ref{def-beh-eq} that two states $x$ and $x'$ are behaviourally
  equivalent in $\cat{Mod}(\fun{T})$ if there are $\fun{T}$-model morphisms $f$
  and $f'$ with $f(x) = f'(x')$.

\begin{thm}\label{thm-me-be}
  Under the assumptions of Theorem~\ref{thm-final-modT}, we have
  ${\equiv_{\Lambda}} = \, {\simeq_{\cat{Mod}(\fun{T})}}$.
\end{thm}
\begin{proof}
  If $x$ and $x'$ are behaviourally equivalent, then they are modally equivalent
  by Proposition~\ref{prop-mor-pres-truth}. Conversely, if they are modally
  equivalent, then $\th_{\mf{X}}(x) = \th_{\mf{X}'}(x')$ by construction, so they
  are behaviourally equivalent.
\end{proof}

\begin{rem}\label{rem-top-sob}
  If $\fun{T}$ is an endofunctor on $\cat{Sob}$ instead of $\cat{Top}$, then the
  same procedure yields a final model in $\cat{Mod}(\fun{T})$. In particular,
  $\fun{T}$ need not be the restriction of a $\cat{Top}$-endofunctor.
  However, if $\fun{T}$ is an endofunctor on $\cat{KSob}$ or $\cat{KHaus}$ the
  procedure above does not guarantee a final coalgebra in $\cat{Mod}(\fun{T})$;
  indeed the state space $\topo{Z}$ of the final coalgebra $\mf{Z}$ we just
  constructed need not be compact sober nor compact Hausdorff.

  Of course, there may be a different way to attain similar results for
  $\cat{KSob}$ or $\cat{KHaus}$.  We leave this as an interesting open question.
  In Theorem~\ref{thm-ksob-me-be} we prove an analogue of Theorem~\ref{thm-me-be}
  for endofunctors on $\cat{KSob}$.
\end{rem}

\section{Bisimulations}\label{sec-bisim}

  This section is devoted to bisimulations and bisimilarity between coalgebraic
  geometric models.
  Bisimulations are important tools in the study of modal logics.
  They provide a structural notion of semantic equivalence:
  bisimilar worlds satisfy precisely the same logical formulas.
  If the converse is also true then the logical language is powerful enough
  to distinguish non-bisimilar states. This is called the
  \emph{Hennessy-Milner property}~\cite{HenMil85}.

  We compare two notions of bisimilarity, modal equivalence
  (from Definition~\ref{def-gml-sem}) and behavioural equivalence
  (Definition~\ref{def-beh-eq}). Again, where $\cat{C}$ is a full subcategory of
  $\cat{Top}$ and $\fun{T}$ an endofunctor on $\cat{C}$, we give definitions and
  propositions in this generality where possible. When necessary, we will
  restrict our scope to particular instances of $\cat{C}$.

\begin{defi}
  Let $\mf{X} = (\topo{X}, \gamma, V)$ and $\mf{X}' = (\topo{X}', \gamma', V')$
  be two geometric $\fun{T}$-models.
  Let $\topo{B}$ be an object in $\cat{C}$ such that $\topo{B} \subseteq
  \topo{X} \times \topo{X}'$, with projections $\pi : \topo{B} \to
  \topo{X}$ and $\pi' : \topo{B} \to \topo{X}'$. Then $\topo{B}$ is called an
  \emph{Aczel-Mendler bisimulation} between $\mf{X}$ and $\mf{X}'$ if for all
  $(x, x') \in \topo{B}$ we have $x \in V(p)$ iff $x' \in V'(p)$ and there exists
  a transition map $\beta : \topo{B} \to \fun{T}\topo{B}$ that makes $\pi$ and
  $\pi'$ coalgebra morphisms. That is, $\beta$ is such that the following diagram
  commutes:
  \begin{equation*}
    \begin{tikzcd}
      \topo{X}
            \arrow[d, "\gamma" left]
        & \topo{B} \arrow[l, "\pi" above]
            \arrow[r, "\pi'" above]
            \arrow[d, "\beta" right]
        & \topo{X}'
            \arrow[d, "\gamma'" right] \\
      \fun{T}\topo{X}
        & \fun{T}\topo{B}
            \arrow[l, "\fun{T}\pi" below]
            \arrow[r, "\fun{T}\pi'" below]
        & \fun{T}\topo{X}'
    \end{tikzcd}
  \end{equation*}
  Two states $x \in \fun{U}\topo{X}, x' \in \fun{U}\topo{X}'$ are called
  \emph{bisimilar}, notation $x \leftrightarroweq x'$, if they are linked by an
  Aczel-Mendler bisimulation.
\end{defi}

  Note that $\leftrightarroweq$ defines a relation between the sets underlying
  the models $\mf{X}$ and $\mf{X}'$. So while Aczel-Mendler bisimulations are
  defined as topological spaces, the resulting notion of bisimilarity is
  simply a relation.

  It follows from Proposition~\ref{prop-mor-pres-truth} that bisimilar states
  satisfy the same formulas.
  Furthermore, if $\cat{C}$ has pushouts then
  it follows from taking pushouts that Aczel-Mendler
  bisimilarity implies behavioural equivalence.
  If moreover $\fun{T}$ preserves weak pullbacks, the converse holds
  as well~\cite{Rut00}.

  However, we do not wish to make this assumption on topological spaces, since
  few functors seem to preserve weak pullbacks. For example, the Vietoris functor
  does not preserve weak pullbacks~\cite[Corollary~4.3]{BezFonVen10} and neither
  does the monotone functor from Definition~\ref{def-mon-khaus}. (To see the
  latter statement, consider the example given in Section~4 of~\cite{HanKup04}
  and equip the sets in use with the discrete topology.)
  Therefore we define $\Lambda$-bisimulations for $\cat{Top}$-coalgebras as an
  alternative to Aczel-Mendler bisimulations. This notion is an adaptation of
  ideas in~\cite{BakHan17, GorSch13}. Under some conditions on $\Lambda$,
  $\Lambda$-bisimilarity coincides with behavioural equivalence.

  In the next definition we need the concept of coherent pairs: If $X$ and $X'$
  are two sets and $B \subseteq X \times X'$ is a relation, then a pair
  $(a, a') \in \fun{P}X \times \fun{P}X'$ is called \emph{$B$-coherent} if
  $B[a] \subseteq a'$ and $B^{-1}[a'] \subseteq a$. For details and properties
  see Section~2 in~\cite{HanKupPac09}.

\begin{defi}\label{def-lambis}
  Let $\fun{T}$ be an endofunctor on $\cat{C}$, $\Lambda$ a geometric modal
  signature for $\fun{T}$ and $\mf{X} = (\topo{X}, \gamma, V)$ and
  $\mf{X}' = (\topo{X}', \gamma', V')$ two geometric $\fun{T}$-models.
  A \emph{$\Lambda$-bisimulation} between $\mf{X}$ and $\mf{X}'$ is a relation
  $B \subseteq \fun{U}\topo{X} \times \fun{U}\topo{X}'$ such that for all
  $(x, x') \in B$ and $p \in \Phi$ we have
  \begin{equation*}
    x \in V(p) \quad\text{ iff }\quad x' \in V'(p)
  \end{equation*}
  and for all $\lambda \in \Lambda$ and all tuples of $B$-coherent pairs of opens
  $(a_i, a_i') \in \fun{\Omega}\topo{X} \times \fun{\Omega}\topo{X}'$:
  \begin{equation}\label{ch3-eq-lambis2}
    \gamma(x) \in \lambda_{\topo{X}}(a_1, \ldots, a_n)
      \iff \gamma'(x') \in \lambda_{\topo{X}'}(a_1', \ldots, a_n').
  \end{equation}
  Two states are called $\Lambda$-bisimilar if there is a $\Lambda$-bisimulation
  linking them, notation: $x \leftrightarroweq_{\Lambda} x'$.
\end{defi}

  We give an alternative characterisation of~\eqref{ch3-eq-lambis2} to elucidate
  the connection with~\cite{BakHan17,GroHanKur20}.

\begin{rem}
  In the abstract setting of~\cite{GroHanKur20}, bisimulations are taken to
  be spans in the base category satisfying certain conditions.
  By contrast, in Definition~\ref{def-lambis} above we define bisimulations
  using relations between topological spaces, rather than spans in $\cat{Top}$.
  In order to explain the connection between our approach and that
  in~\cite{BakHan17,GroHanKur20},
  we equip a relation between topological spaces with the subspace
  topology.

  So let $\topo{B} \subseteq \topo{X} \times \topo{X}'$ be a relation endowed with the
  subspace topology and let $\pi : \topo{B} \to \topo{X}$ and $\pi' : \topo{B} \to \topo{X}'$
  be projections. Then $(a, a') \in \fun{\Omega}\topo{X} \times
  \fun{\Omega}\topo{X}'$ is $\topo{B}$-coherent iff $\pi^{-1}(a) = (\pi')^{-1}(a')$.
  Let $P$ be the pullback of the cospan
  $\begin{tikzcd}[column sep=2.4em, cramped]
    \fun{opn}\topo{X} \arrow[r, "\fun{opn}\pi" above]
    & \fun{opn}\topo{B} & \fun{opn}\topo{X}' \arrow[l, "\fun{opn}\pi'" above]
  \end{tikzcd}$ 
  in $\cat{Frm}$ and let $p : P \to \fun{opn}\topo{X}$ and
  $p' : P \to \fun{opn}\topo{X}'$ be the
  corresponding projections.
  Then the $\topo{B}$-coherent pairs are precisely $(p(x), p'(x))$, where $x$ ranges
  over $P$. It follows from the definitions that equation~\eqref{ch3-eq-lambis2}
  holds for all $\topo{B}$-coherent pairs if and only if
  \begin{equation*}
    \fun{opn}\pi \circ \fun{opn}\gamma \circ \lambda_{\topo{X}} \circ p^n
      = \fun{opn}\pi' \circ \fun{opn}\gamma' \circ \lambda_{\topo{X}'}
        \circ (p')^n,
  \end{equation*}
  where $\lambda$ is $n$-ary.
\end{rem}

  As desired, $\Lambda$-bisimilar states always satisfy the same formulas.

\begin{prop}\label{prop-lambis-me}
  Let $\fun{T}$ be an endofunctor on $\cat{C}$ and $\Lambda$ a geometric modal
  signature for $\fun{T}$. Then ${\leftrightarroweq_{\Lambda}} \subseteq
  {\equiv_{\Lambda}}$.
\end{prop}
\begin{proof}
  Let $B$ be a $\Lambda$-bisimulation between geometric $\fun{T}$-models $\mf{X}$
  and $\mf{X}'$, and suppose $xBx'$. Using induction on the complexity of the
  formula, we show that $\mf{X}, x \Vdash \phi$ iff $\mf{X}', x' \Vdash \phi$ for
  all $\phi \in \GML(\Phi, \Lambda)$. The propositional case is by definition, and
  $\wedge$ and $\bigvee$ are routine.
  Suppose $\mf{X}, x \Vdash \heartsuit^{\lambda}(\phi_1, \ldots, \phi_{n})$, then
  $\gamma(x) \in \lambda_{\topo{X}}(\llb \phi_1 \rrb^{\mf{X}}, \ldots,
  \llb \phi_{n} \rrb^{\mf{X}})$. By the induction hypothesis
  $(\llb \phi_i \rrb^{\mf{X}}, \llb \phi_i \rrb^{\mf{X}'})$ is $B$-coherent for
  all $i$.
  Then $\gamma'(x') \in \lambda_{\topo{X}'}(\llb \phi_1 \rrb^{\mf{X}'}, \ldots,
  \llb \phi_{n} \rrb^{\mf{X}'})$ since $B$ is a $\Lambda$-bisimulation, hence
  $\mf{X}', x' \Vdash \heartsuit^{\lambda}(\phi_1, \ldots, \phi_{n})$. The
  converse is proven symmetrically.
\end{proof}

\begin{prop}
  Let $\fun{T}$ be an endofunctor on $\cat{C}$ and $\Lambda$ a geometric modal
  signature for $\fun{T}$. Then
  ${\leftrightarroweq} \subseteq {\leftrightarroweq_{\Lambda}}$.
\end{prop}
\begin{proof}
  It suffices to show that every Aczel-Mendler bisimulation gives rise to a
  $\Lambda$-bisimulation.
  Suppose $\topo{B}$ is an Aczel-Mendler bisimulation and let $\beta$ be the map that
  turns $\topo{B}$ into a coalgebra, then the following diagram commutes:
  \begin{equation}\label{eq-am-lambis}
    \begin{tikzcd}
      \topo{X}
            \arrow[d, "\gamma" left]
        & \topo{B}
            \arrow[l, "\pi" above]
            \arrow[r, "\pi'" above]
            \arrow[d, "\beta"]
        & \topo{X}'
            \arrow[d, "\gamma'" right] \\
      \fun{T}\topo{X}
        & \fun{T}\topo{B}
            \arrow[l, "\fun{T}\pi" below]
            \arrow[r, "\fun{T}\pi'" below]
        & \fun{T}\topo{X}'
    \end{tikzcd}
  \end{equation}

  We will show that the set $B$ underlying the topological space $\topo{B}$
  is a $\Lambda$-bisimulation.
  By definition $x \in V(p)$ iff $x' \in V'(p)$ whenever $xBx'$.
  We prove the forth condition from Definition~\ref{def-lambis}. Let
  $\lambda \in \Lambda$ and $(x, x') \in B$. Suppose $(a_1, a_1'), \ldots,
  (a_n, a_n')$ are $B$-coherent pairs of opens and $\gamma(x) \in
  \lambda_{\topo{X}}(a_1, \ldots, a_{n})$. Then we have
  \begin{align*}
    \beta(x,x')
      &\in (\fun{T}\pi)^{-1}(\lambda_{\topo{X}}(a_1, \ldots, a_{n}))
      &\text{(Follows from~\eqref{eq-am-lambis})} \\
      &= \lambda_B(\pi^{-1}(a_1), \ldots, \pi^{-1}(a_{n}))
      &\text{(Naturality of $\lambda$)}\\
      &\subseteq \lambda_B\big((\pi')^{-1} \circ \pi' [\pi^{-1}(a_1)], \ldots,
        (\pi')^{-1} \circ \pi' [\pi^{-1}(a_{n})]\big)
      &\text{(Monotonicity of $\lambda$)} \\
      &= \lambda_B\big((\pi')^{-1}(B[a_1]), \ldots, (\pi')^{-1}(B[a_{n}])\big)
      &\text{($B[a] = \pi' \circ \pi^{-1}(a)$)} \\
      &\subseteq \lambda_B((\pi')^{-1}(a_1'), \ldots, (\pi')^{-1}(a_{n}'))
      &\text{(Monotonicity of $\lambda$)} \\
      &= (\fun{T}\pi')^{-1}(\lambda_{\topo{X}'}(a_1', \ldots, a_{n}')).
      &\text{(Naturality of $\lambda$)}
  \end{align*}
  Therefore
  \begin{equation*}
    \gamma'(x')
      = (\fun{T}\pi')(\beta(x,x')) \in \lambda_{\topo{X}'}(a_1', \ldots, a_{n}'),
  \end{equation*}
  as desired.
\end{proof}

  The collection of $\Lambda$-bisimulations between two models enjoys the
  following interesting property.

\begin{prop}
  Let $\Lambda$ be a geometric modal signature of a functor $\fun{T}: \cat{Top}
  \to \cat{Top}$ and let $\mf{X} = (\topo{X}, \gamma, V)$ and
  $\mf{X}' = (\topo{X}', \gamma', V')$ be two geometric $\fun{T}$-models.
  The collection of $\Lambda$-bisimulations between $\mf{X}$ and $\mf{X}'$ forms
  a complete lattice.
\end{prop}
\begin{proof}
  It is obvious that the collection of $\Lambda$-bisimulations is a poset. We
  will show that this collection is closed under taking arbitrary unions; the
  result then follows from the fact that any complete semilattice is also a
  complete lattice, see e.g.~\cite[Theorem~4.2]{BurSan81}.

  Let $J$ be some index set and for all $j \in J$ let $B_j$ be
  $\Lambda$-bisimulations between $\mf{X}$ and $\mf{X}'$ and set
  $B = \bigcup_{j \in J} B_j$. We claim that $B$ is a $\Lambda$-bisimulation.

  Let $(a_i, a_i')$ be $B$-coherent pairs of opens. Suppose $xBx'$ and
  $\gamma(x) \in \lambda_{\topo{X}}(a_1, \ldots, a_n)$. Then there is $j \in J$
  with $xB_{j}x'$ hence $x \in V(p)$ iff $x' \in V'(p)$. As $B_j[a_i] \subseteq
  B[a_i] \subseteq a_i'$ and $B_j^{-1}[a'_i] \subseteq B^{-1}[a'_i] \subseteq a_i$, all
  $B$-coherent pairs $(a_i, a_i')$ are also $B_j$-coherent. Since $B_j$ is a
  $\Lambda$-bisimulation we get $\gamma'(x') \in \lambda_{\topo{X}'}(a_1',
  \ldots, a_n')$. The converse direction is proven symmetrically.
\end{proof}

  We know by now that $\Lambda$-bisimilarity implies modal equivalence.
  Furthermore, we have seen in Theorem~\ref{thm-me-be} that modal equivalence
  coincides with behavioural equivalence whenever $\fun{T}$
  is an endofunctor on $\cat{Top}$ which preserves sobriety and $\Lambda$ is a
  Scott-continuous characteristic geometric modal signature.
  In order to prove
  a converse, i.e.~that behavioural equivalence implies $\Lambda$-bisimilarity,
  we need to assume that the geometric modal signature is extendable.

  Recall that two elements $x, x'$ in two models are behaviourally equivalent in
  $\cat{Mod}(\fun{T})$, notation: $x \simeq_{\cat{Mod}(\fun{T})} x'$, if there exist
  morphisms $f, f'$ in $\cat{Mod}(\fun{T})$ such that $f(x) = f'(x')$.

\begin{prop}\label{prop-be-lambis}
  Let $\fun{T}$ be an endofunctor on $\cat{C}$ and $\Lambda$ a monotone extendable
  geometric modal signature for $\fun{T}$. Let $\mf{X} = (\topo{X}, \gamma, V)$
  and $\mf{X}' = (\topo{X}', \gamma', V')$ be two geometric $\fun{T}$-models.
  Then ${\simeq_{\cat{Mod}(\fun{T})}} \subseteq {\leftrightarroweq_{\Lambda}}$.
\end{prop}
\begin{proof}
  Suppose $x$ and $x'$ are behaviourally equivalent. Then there are some
  geometric $\fun{T}$-model $\mf{Y} = (\topo{Y}, \nu, V_{\topo{Y}})$ and
  $\fun{T}$-model morphisms $f : \mf{X} \to \mf{Y}$ and $f' : \mf{X}' \to \mf{Y}$
  such that $f(x) = f'(x')$. We will show that
  \begin{equation}
    B = \{ (u, u') \in X \times X' \mid f(u) = f'(u') \}
  \end{equation}
  is a $\Lambda$-bisimulation linking $x$ and $x'$.

  Clearly $xBx'$. It follows from Proposition~\ref{prop-mor-pres-truth} that $u$
  and $u'$ satisfy precisely the same formulas whenever $(u, u') \in B$.
  Suppose $\lambda \in \Lambda$ is $n$-ary and for $1 \leq i \leq n$ let
  $(a_i, a_i')$ be a $B$-coherent pair of opens. Suppose $uBu'$ and
  $\gamma(u) \in \lambda_{\topo{X}}(a_1, \ldots, a_n)$. We will show that
  $\gamma'(u') \in \lambda_{\topo{X}'}(a_1', \ldots, a_n')$. The converse
  direction is similar.

  By monotonicity and naturality of $\lambda$ we obtain
  \begin{equation*}
    \gamma(u) \in \lambda_{\topo{X}}(a_1, \ldots, a_n)
      \subseteq \lambda_{\topo{X}}(f^{-1}(f[a_1]), \ldots, f^{-1}(f[a_n]))
      = (\fun{T}f)^{-1}(\lambda_{\topo{Y}}(f[a_1], \ldots, f[a_n])),
  \end{equation*}
  so $(\fun{T}f)(\gamma(u)) \in \lambda_{\topo{Y}}(f[a_1], \ldots, f[a_n])$.
  (The $f[a_i]$ need not be open in $\topo{Y}$, but since $\lambda$ is extendable,
  $\lambda_{\topo{Y}}(f[a_1], \ldots, f[a_n])$ is defined.) Because $f$ and $f'$
  are coalgebra morphisms and $f(u) = f'(u')$ we have
  $(\fun{T}f)(\gamma(u)) = \nu(f(u)) = \nu(f'(u')) = (\fun{T}f')(\gamma'(u'))$.
  Finally, we get
  \begin{align*}
    \gamma'(u')
      &\in (\fun{T}f')^{-1}(\lambda_{\topo{Y}}(f[a_1], \ldots, f[a_n])) \\
      &= \lambda_{\topo{X}'}((f')^{-1}(f[a_1]), \ldots, (f')^{-1}(f[a_n]))
      &\text{(Naturality of $\lambda$)} \\
      &= \lambda_{\topo{X}'}(B[a_1], \ldots, B[a_n])
      &\text{(Monotone extendability of $\lambda$)} \\
      &\subseteq \lambda_{\topo{X}'}(a_1', \ldots, a_n').
      &\text{(Coherence of $(a_i, a_i')$)}
  \end{align*}
  This proves the proposition.
\end{proof}

\begin{rem}
  If $\cat{C} = \cat{KHaus}$ in the proposition above, then
  Proposition~\ref{prop-strong-pred-lift} allows us to drop the assumption
  that $\Lambda$ be extendable.
\end{rem}

  Let $\fun{T}$ be an endofunctor on $\cat{Top}$ and let $\Lambda$ be a geometric
  modal signature for $\fun{T}$. The following diagram summarises the results
  from Propositions~\ref{prop-lambis-me} and~\ref{prop-be-lambis} and Theorem~\ref{thm-me-be}.
  The arrows indicate that one form of equivalence implies the
  other. Here (1) holds if $\fun{T}$ preserves weak pullbacks, (2) is true when
  $\Lambda$ is Scott-continuous and characteristic and $\fun{T}$ preserves
  sobriety (cf.~Theorem~\ref{thm-me-be}),
  and (3) holds when $\Lambda$ is monotone extendable. Note that the
  converse of (2) always holds, because morphisms preserve truth
  (Proposition~\ref{prop-mor-pres-truth}).
  \begin{equation}
    \begin{tikzcd}
      {\leftrightarroweq}
            \arrow[r]
        & {\leftrightarroweq_{\Lambda}}
            \arrow[r]
        & {\equiv_{\Lambda}}
            \arrow[r, -latex, shift left=2pt, "\text{(2)}"]
        & {\simeq_{\cat{Mod}(\fun{T})}}
            \arrow[l, -latex, shift left=2pt]
            \arrow[lll, bend right=30, "\text{(1)}" above]
            \arrow[ll, bend left=30, "\text{(3)}" below]
    \end{tikzcd}
  \end{equation}

  As stated in the introduction we are not only interested in endofunctors on
  $\cat{Top}$, but also in endofunctors on full subcategories of $\cat{Top}$, in
  particular $\cat{KHaus}$.

  The implications in the diagram hold for endofunctors on $\cat{Sob}$ as well
  (use Remark~\ref{rem-top-sob}). Moreover, with some extra effort it can be made
  to work for endofunctors on $\cat{KSob}$ as well. In order to achieve this, we
  have to redo the proof for the bi-implication between modal equivalence and
  behavioural equivalence. This is the content of the following theorem.

\begin{thm}\label{thm-ksob-me-be}
  Let $\fun{T}$ be an endofunctor on $\cat{KSob}$, $\Lambda$ a Scott-continuous
  characteristic geometric modal signature for $\fun{T}$ and
  $\mf{X} = (\topo{X}, \gamma, V)$ and $\mf{X}' = (\topo{X}', \gamma', V')$ two
  geometric $\fun{T}$-models. Then
  ${\equiv_{\Lambda}} = \, {\simeq_{\cat{Mod}(\fun{T})}}$.
\end{thm}
\begin{proof}
  If $x$ and $x'$ are behaviourally equivalent then they are modally equivalent
  by Proposition~\ref{prop-mor-pres-truth}. The converse direction can be proved
  using similar reasoning as in Section~\ref{sec-final}. The major difference is
  the following: We define an equivalence relation $\equiv_2$ on $\GML(\Phi, \Lambda)$
  by $\phi \equiv_2 \psi$ iff $\llb \phi \rrb^{\mf{X}} = \llb \psi \rrb^{\mf{X}}$
  and $\llb \phi \rrb^{\mf{X}'} = \llb \psi \rrb^{\mf{X}'}$. (Note that $\mf{X}$
  and $\mf{X}'$ are now fixed.) That is, $\phi \equiv_2 \psi$ iff $\phi$ and
  $\psi$ are satisfied by precisely the same states in $\mf{X}$ and $\mf{X}'$
  (compare Definition~\ref{def-final-equiv}). The frame $\mathbf{E}_2 :=
  \GML(\Phi, \Lambda) /{\equiv_2}$ can then be shown to be a compact frame and hence
  $\topo{Z}_2 := \fun{pt} \mathbf{E}_2$ is a compact sober space. The remainder
  of the proof is analogous to the proof of Theorem~\ref{thm-me-be}. A detailed
  proof can be found in~\cite[Theorem~3.34]{Gro18}.
\end{proof}

  We summarise the results for $\cat{Top}$ and two of its full subcategories:

\begin{thm}\label{thm-me-be-lambis2}
  Let $\fun{T}$ be a sobriety-preserving endofunctor on $\cat{Top}$
  or an endofunctor on $\cat{Sob}$ or $\cat{KSob}$, and
  $\Lambda$ a monotone extendable Scott-continuous characteristic geometric modal
  signature for $\fun{T}$. If $x$ and $x'$ are two states in two geometric
  $\fun{T}$-models, then
  \begin{equation*}
    x \leftrightarroweq_{\Lambda} x'
      \iff x \equiv_{\Lambda} x'
      \iff x \simeq_{\cat{Mod}(\fun{T})} x'.
  \end{equation*}
\end{thm}

\section{Conclusion}

  We have started building a framework for coalgebraic \emph{geometric} logic and
  we have investigated examples of concrete functors. There are still many
  unanswered and interesting questions. We outline possible directions for
  further research.

\begin{description}
  \item[Modal equivalence versus behavioural equivalence]
        From Theorem~\ref{thm-me-be-lambis2} we know that modal equivalence and
        behavioural equivalence coincide in $\cat{Mod}(\fun{T})$ if $\fun{T}$ is
        an endofunctor on $\cat{KSob}$, $\cat{Sob}$ or an endofunctor on
        $\cat{Top}$ which preserves sobriety. A natural question is whether the
        same holds when $\fun{T}$ is an endofunctor on $\cat{KHaus}$.
  \item[When does a lifted functor restrict to $\cat{KHaus}$?]
        We know of two examples, namely the powerset functor with the box and
        diamond lifting, and the monotone functor on $\cat{Set}$ with the box and
        diamond lifting, where the lifted functor on $\cat{Top}$ restricts to
        $\cat{KHaus}$. It would be interesting to investigate whether there are
        explicit conditions guaranteeing that the KKP lift of a functor restricts to
        $\cat{KHaus}$. These conditions could be either for the $\cat{Set}$-functor
        one starts with, or the collection of predicate liftings for this functor,
        or both.
  \item[Modalities and finite observations]
        Geometric logic is generally introduced as the logic of finite
        observations, and this explains the choice of connectives ($\land$,
        $\bigvee$ and, in the first-order version, $\exists$).
        We would like to understand to which degree \emph{modalities} can safely
        be added to the base language, without violating the (semantic) intuition
        of finite observability.
        Clearly there is a connection with the requirement of
        \emph{Scott-continuity} (preservation of directed joins), and we would
        like to make this connection precise, specifically in the topological setting.
  \item[Connection to geometric predicate logic]
        The extension of geometric logic with predicates is discussed in
        e.g.~\cite[Chapter D1]{Joh02} and~\cite{Vic07}. In this setting
        toposes replace the r\^{o}le of frames. This raises the question of
        how modal geometric logic and geometric predicate logics relate.
        For example, can we find an analogue of the Van Benthem characterisation
        theorem, characterising modal geometric logic as a bisimulation-invariant
        fragment of geometric predicate logic?
  \item[Order-enriched category theory]
        The fact that all examples in this paper use open predicate liftings
        that are monotone suggests that we could also work in the setting
        of order-enriched category theory.
        The richer structure of order-enriched category theory was successfully
        used in the study of positive coalgebraic logics~\cite{KapKurVel14,BalKurVel15}.
  \item[Connection with domain theory]
        In~\cite{Abr91} Abramsky investigates distributive lattices, rather than frames.
        This raises the question which of his examples have extensions to frames.
  \item[An abstract view on coalgebraic logic]
        If we abstract away from predicate liftings, a logic for an endofunctor
        $\fun{T}$ on $\cat{Top}$ may be viewed as a functor
        $\fun{L} : \cat{Frm} \to \cat{Frm}$ together with a natural transformation
        $\rho : \fun{L} \circ \fun{opn} \to \fun{opn} \circ \fun{T}$~\cite{Kli07,KurRos12}.
        A generic choice for
        $\fun{L}$ is $\fun{L} = \fun{opn} \circ \fun{T} \circ \fun{pt}$,
        which is also used in Section~\ref{sec:final} to construct a final
        $\fun{T}$-coalgebra.
        Moreover, in the setting of Section~\ref{sec:final} the algebra
        $(\mathbf{E}, \delta)$ is the initial $\fun{L}$-algebra,
        which gives rise to a completeness result.
        A natural question is whether or not an endofunctor $\fun{L} : \cat{Frm} \to \cat{Frm}$
        can be presented by generators and relations.
        Similar questions for endofunctors on the categories of Boolean
        algebras and distributive lattices are addressed in~\cite{KurRos12}
        and~\cite{BalKurVel15}, respectively.
\end{description}

\section*{Acknowledgements}
  We would to thank the anonymous referees for many constructive and helpful
  comments, which helped embed our paper more closely into the body of
  existing research and opened many potential avenues for further research.


\bibliography{biblio.bib}{}
\bibliographystyle{alphaurl}

\end{document}